\documentclass[10pt]{article}
\usepackage{ifthen}
\newboolean{Conference}
\setboolean{Conference}{false}
\newcommand{\NotForConf}[1]{\ifthenelse{\boolean{Conference}}{}{#1}}
\newcommand{\IfConf}[2]{\ifthenelse{\boolean{Conference}}{#1}{#2}}

\usepackage{hslTR}
\usepackage{fullpage}
\usepackage{array} % for better arrays (eg matrices) in maths
\usepackage{graphics} % support the \includegraphics fcommand and options
\usepackage{graphicx} 
\usepackage{amssymb}
\usepackage{amsxtra}
\usepackage{algorithm}
\usepackage{algorithmic}
\usepackage{color}
\usepackage{psfrag}
\usepackage{subfigure}
\usepackage{balance}
\usepackage{wrapfig}
\usepackage[crop=pdfcrop]{pstool}
\allowdisplaybreaks[1] %Allows page breaking for equation arrays

\def\startmodif{\color{black}}
\def\stopmodif{\color{black}\normalcolor}
\definecolor{mygreen}{RGB}{0,128,0}

\newcommand{\sean}[1]{{\color{black} #1}}
\newcommand{\seannew}[1]{{\color{black} #1}}
\newcommand{\A}{\mathcal{A}}

\newcommand{\tb}{\bar{\tau}}
\newcommand{\ton}{\tau_1}
\newcommand{\ttw}{\tau_2}

\newcommand{\X}{\mathcal{X}}
\newcommand{\one}{\mathbf 1}
\newcommand{\e}{\varepsilon}
\newcommand{\mat}[1]{\left[\begin{array}{c} #1 \end{array}\right]}
\newcommand{\wt}{\widetilde}

\renewcommand{\S}{\mathcal{S}}
\renewcommand{\hbar}{\overline{h}}
\newcommand{\G}{\mathcal{G}}
\renewcommand{\cal}{\mathcal}

\newcommand{\EndPF}{\hfill $\blacksquare$}

\newcommand{\HS}{{\cal H}}
\newcommand{\reals}{\mathbb{R}}
\newcommand{\ball}{\mathbb{B}}
\newcommand{\tto}{\rightrightarrows}
\newcommand{\realsgeq}{{\reals_{\geq 0}}}
\newcommand{\dom}{\mathop{\rm dom}\nolimits}      
\newcommand{\nats}{\mathbb{N}}
\newcommand{\non}{\nonumber}
\newtheorem{helptheorem}{Theorem}[section]     
 
\newtheorem{helplemma}[helptheorem]{Lemma}     
\newtheorem{helpcorollary}[helptheorem]{Corollary}     

\newtheorem{helpexample}[helptheorem]{Example}     

\newtheorem{helpproposition}[helptheorem]{Proposition}     

\newtheorem{helpremark}[helptheorem]{Remark}     

\newtheorem{helpdefinition}[helptheorem]{Definition}

\newtheorem{helpassumption}[helptheorem]{Assumption}

\newtheorem{helpstassumption}[helptheorem]{Standing Assumption}

\newenvironment{theorem}     
{\vskip.1cm\begin{helptheorem}\it}
{\end{helptheorem} \vskip.1cm}   

\newenvironment{lemma}     
{\vskip.1cm\begin{helplemma}\it}     
{\end{helplemma}\vskip.1cm}

\newenvironment{definition}
{\vskip.1cm\begin{helpdefinition}}     
{\end{helpdefinition}\vskip.1cm}   

\newenvironment{proof}[1][Proof]{\begin{trivlist}
\item[\hskip \labelsep {\bfseries #1}]}{\end{trivlist}}

\begin{document}
%
%\ititle{{\bf EVIL FIREFLIES:} Desynchronization of Impulsive Oscillators}
\NotForConf{
\ititle{Robust Asymptotic Stability of Desynchronization in Impulse Coupled Oscillators}
\iauthor{
  Sean Phillips \\
  {\normalsize seaphill@ucsc.edu} \\
  Ricardo G. Sanfelice\\
  {\normalsize ricardo@ucsc.edu}}
\idate{\today{}} 
\iyear{2015}
\irefnr{001}
\makeititle}{
\title{}
\date{}
%\title{Robust Asymptotic Stability of Desynchronization in Impulse-Coupled Oscillators}
\author{Sean Phillips and Ricardo G. Sanfelice\thanks{Department of Computer Engineering, University of California, Santa Cruz, CA 95064.
      Email: {\tt\small seaphill,ricardo@ucsc.edu}. 
      This research has been partially supported by the National Science Foundation under CAREER Grant no. ECS-1150306 and by the Air Force Office of Scientific Research under Grant no. FA9550-12-1-0366.
}}}

\newpage
\tableofcontents
\newpage

\maketitle
\begin{abstract}
The property of desynchronization in an all-to-all network of homogeneous impulse-coupled oscillators is studied. Each impulse-coupled oscillator is modeled as a hybrid system with a single timer state that self-resets to zero when it reaches a threshold, at which event all other impulse-coupled oscillators adjust their timers following a common reset law. In this setting, desynchronization is considered as each impulse-coupled oscillator's timer having equal separation between successive resets.
We show that, for the considered model, desynchronization is an asymptotically stable property. For this purpose, we recast desynchronization as a set stabilization problem and employ Lyapunov stability tools for hybrid systems. Furthermore, several perturbations are considered showing that desynchronization is a robust property. Perturbations on both the continuous and discrete dynamics are considered. Numerical results are presented to illustrate the main contributions.
\end{abstract} 
\section{Introduction}
%\subsection{Background}

Impulse-coupled oscillators are multi-agent systems with state variables consisting of timers that evolve continuously until a state-dependent event triggers an instantaneous update of their values. 
Networks of such oscillators have been employed to model the dynamics of a wide 
range of biological and engineering systems. In fact, impulse-coupled oscillators 
have been used to model groups of fireflies \cite{Mirollo.90.SIAMJAM.BiologicalOscillators}, 
spiking neurons \cite{Pikovsky.ea.03.BiologicalOscillators,gerstner2002spiking},  
muscle cells \cite{Peskin.75.BiologicalOscillators}, wireless networks 
\cite{hong2010cooperative}, and sensor networks \cite{Liu05adynamic}. \startmodif
 With synchronization being a property of particular interest, \stopmodif 
such 
complex networks have been found to coordinate the values of their state variables 
by sharing information only at the times the events/impulses occur \cite{Mirollo.90.SIAMJAM.BiologicalOscillators,Abbott1993}.
 
The opposite of synchronization is {\em desynchronization}. In simple words, desynchronization in multi-agent systems is the notion that the agents' periodic actions are separated ``as far apart'' as possible in time.  Desynchronization is similar to clustering or splay-state configurations, and is sometimes referred in the literature as inhibited behavior \cite{mauroy:037122,glass1988clocks}.
\startmodif For impulse-coupled oscillators, desynchronization is given as the behavior in which the separation between all of the timers impulses is equal \cite{Patel.Desync2007}.
\stopmodif This behavior has been found to be present in communication schemes in fish \cite{Benda.Neuron} and in  networks of spiking neurons \cite{Pfurtscheller19991842,1997Natur.390.70S}. Desynchronization of oscillators has recently been shown to be of importance in the understanding of Parkinson's disease \cite{Mabi.Moehlis2010,Majtanik.Dolan2004}, in
the design of algorithms that limit the amount of overlapping data transfer and data loss in wireless digital networks \cite{hong2010cooperative}, and in the design of round-robin scheduling schemes for  sensor networks \cite{Liu05adynamic}.

%\sean{Since there is a such small attention to this area of research we compare ourselves with \cite{Patel.Desync2007} which defines several algorithms for the desynchronization of pulse-coupled oscillators. The paper used a network of non-linear oscillators that was introduced in \cite{Mirollo.90.SIAMJAM.BiologicalOscillators}. We will compare our algorithm with this paper and introduce some stability results into this research area. This paper analyzes the convergence of the network of pulse-coupled oscillators and characterizes the rate of convergence for $N > 2$ oscillators. Furthermore, our framework allows for certain analytical robustness conditions. }

%\subsection{Contributions}

Motivated by the applications mentioned above and the lack of a full understanding 
of desynchronization in multi-agent systems, this paper pertains to the study of 
the dynamical properties of desynchronization in a network of impulse-coupled 
oscillators with an all-to-all communication graph.  The uniqueness of the 
approach emerges from the use of hybrid systems tools, which not only conveniently 
capture the continuous and impulsive behavior in the networks of interest, but 
also are suitable for analytical study of asymptotic stability and robustness to 
perturbations.

More precisely, the dynamics of the proposed hybrid system capture the (linear) 
continuous evolution of the states as well their impulsive/discontinuous behavior 
due to state triggered events. 
Analysis of the asymptotic \startmodif behavior \stopmodif of the trajectories (or solutions) to these systems is performed using the framework of hybrid systems introduced in \cite{teel2012hybrid,Goebel.ea.09.CSM}. To this end, we recast the study of desynchronization as a set stabilization problem. 
Unlike synchronization, for which the set of points to stabilize is obvious,
the complexity of desynchronization requires first to determine such a collection of points, which 
we refer to as the {\em desynchronization set}. We propose an algorithm to compute such set of points. Then, using Lyapunov stability theory for hybrid systems, we prove that the desynchronization set is asymptotically stable by defining a Lyapunov-like function as the distance between the state and (an inflated version of) the desynchronization set. In our context, asymptotic stability of the desynchronization set implies that the distance between the state and the desynchronization set converges to zero as the amount of time and the number of jumps get large. 
Using the proposed Lyapunov-like function and invoking an invariance principle, the basin of attraction is characterized and shown to be 
the entire state space minus a set of measure zero, which turns out to actually be an exact estimate of the basin of attraction.
Furthermore, also exploiting the availability of a Lyapunov-like function, we analytically characterize the time for the solutions to reach a neighborhood of the desynchronization set.  In particular, this characterization provides key insight for the design of algorithms used in applications in which desynchronization is crucial, such as wireless digital networks and sensor networks.

The asymptotic stability property of the desynchronization configuration is shown 
to be robust to several types of perturbations. 
%
%From the definition of the hybrid system modeling the desynchronizing impulse coupled oscillator network, certain robustness to perturbations are guaranteed. 
The perturbations studied here include a generic perturbation in the form of an inflation of the dynamics of the proposed hybrid system model of the network of interest and several kinds of perturbations on the timer rates.
Using the tools presented in \cite{teel2012hybrid,Goebel.ea.09.CSM}, we analytically characterize the effect of these perturbations on the already established asymptotic stability property of the desynchronization set. In particular, these perturbations capture situations where the agents in the network are heterogeneous due to having differing timer rates, threshold values, and update laws.
To verify the analytical results, we simulate networks of impulse-coupled oscillators under several classes of perturbations. Specifically, we show numerical results when perturbations
affect the update laws and the timer rates.  \IfConf{Complete numerical results can be found in an extended version of this paper \cite{Phillips2013TechReport}.}

%Specific perturbation models are simulated to verify the results. 

%\subsection{Organization of the Paper}
The remainder of this paper is organized as follows. Section~\ref{sec:hs} is devoted to hybrid modeling of networks of impulse-coupled oscillators. Section~\ref{sec:AN} introduces an algorithm to determine the desynchronization set. 
Section~\ref{sec:lyapunov} presents the stability results while the time to convergence is characterized in Section~\ref{sec:timetoconverge}. The robustness results are in Section~\ref{sec:robustness}. 
Section~\ref{sec:numerics} presents numerical results illustrating our results. Final remarks are given in Section~\ref{sec:conclusion}.

%\subsection{Notation}

\noindent
{\bf Notation}
\IfConf{
The set $\reals$ denotes the space of real numbers.
The set $\reals^n$ denotes the $n$-dimensional Euclidean space.
The set $\mathbb{N}$ denotes the natural numbers including zero, i.e., $\mathbb{N} = \{0, 1, 2, ... \}$.
\startmodif For an interval ${K} = [0,1]$ and $n \in \mathbb{N}\setminus\{0\}$, ${K}_n$  is the $n$-product of the interval ${K}$, i.e., ${K}_n = [0,1]\times[0,1]\times \ldots \times[0,1]$. \stopmodif
The set $\mathbb{B}$ is the closed unit ball centered around the origin in Euclidean space.
The symbol $\one$ represents the $N$-dimensional column vector of ones.
The symbol $\underline\one$ represents the $N\times N$ matrix full of ones.
The symbol ${\bf I}$ denotes the $N \times N$ identity matrix.
 Given a closed set $\A \subset\reals^n$ and $x \in \reals^n$, $|x|_\A := \min_{z \in \A} |x - z|$.
Given $x \in \reals^{n}$, $|x|$ denotes the Euclidean norm of $x$.
The $c$-level set of $V : \dom V \to \reals$ is given by $L_V(c) := \{x \in \dom V : V(x) = c \}$.
%\item The $c$-sublevel set of $V$ is given by $\wt L_V(c) = \{x : V(x) \leq c\}$,
%The set $\overline{\mbox{con}} \ \Sigma$ is the closure of the convex hull of the set $\Sigma$.
}{
\begin{itemize}
\item $\reals$ denotes the space of real numbers.
\item $\reals^n$ denotes the $n$-dimensional Euclidean space.
\item $\mathbb{N}$ denotes the natural numbers including zero, i.e., $\mathbb{N} = \{0, 1, 2, ... \}$.
\item \startmodif For an interval $\mathcal{K} = [0,1]$ and $n \in \mathbb{N}\setminus\{0\}$, $\mathcal{K}_n$  is the $n$-product of the interval $\mathcal{K}$, i.e., $\mathcal{K}_n = [0,1]\times[0,1]\times \ldots \times[0,1]$. \stopmodif
\item $\mathbb{B}$ is the closed unit ball centered around the origin in Euclidean space.
\item $\one$ is an $N$ column vector of ones.
\item $\underline\one$ is an $N\times N$ matrix full of ones.
\item ${\bf I}$ is the $N \times N$ identity matrix.
\item Given a closed set $\A \subset\reals^n$ and $x \in \reals^n$, $|x|_\A := \min_{z \in \A} |x - z|$.
\item Given $x \in \reals^{n}$, $|x|$ denotes the Euclidean norm of $x$.
%\item \startmodif Given two sets $S_{1}, S_{2} \subset \reals^{n}$, $d_{H}(x,y)$ denotes the Hausdorff distance between $S_{1}$ and $S_{2}$, that is, $d_{H}(S_{1},S_{2}) = \max\{\sup_{x \in S_{1}}|x|_{S_{2}},\sup_{x \in S_{2}}|x|_{S_{1}}\}$. \stopmodif
\item The $c$-level set of $V : \dom V \to \reals$ is given by $L_V(c) := \{x \in \dom V : V(x) = c \}$,
%\item The $c$-sublevel set of $V$ is given by $\wt L_V(c) = \{x : V(x) \leq c\}$,
%\item $\overline{\mbox{con}} \ \Sigma$ is the closure of the convex hull of a set $\Sigma$.
\end{itemize}
}
%while Section~\ref{sec:application} introduces two applications of the results.
%Final remarks and future directions are discussed in Section~\ref{sec:discussion}.
%\noindent
%\IfConf{
%\medskip
%
%{\bf Notation:} We use the following notation: 
%$\reals$ denotes the real numbers space. $\reals^{n}$  denotes the $n$-dimensional Euclidean space. $\mathbb{N}$ denotes the natural numbers including zero. Given an interval $\mathcal{S} = [0,t]$ and $n \in \mathbb{N} \setminus \{0\}$, $\mathcal{S}^{n}$ is the cartesian product of the interval, i.e., $[0,t]^{3} = [0,t]\times [0,t] \times [0,t]$. Finally, $\mathbb{B}$ is the closed unit ball centered around the origin in Euclidean space. The Euclidean distance from $x \in \reals^n$ and a set $S \subset \reals^n$ is denoted by $d(x,S)$. A column vector of $N$ ones is denoted by $\one$. The $c$-level set of $V : \dom V \to \reals$ is given by $L_V(c) = \{x : V(x) = c \}$. The $c$-sublevel set of $V$ is given by $\wt L_V(c) = \{x : V(x) \leq c\}$.
%}{\section{Notations}
%}

\section{Hybrid System Model of Impulse-Coupled Oscillators}
\label{sec:hs}
\subsection{Mathematical Model}\label{sec:hybridmodel}
In this paper, we consider a model of $N$ impulse-coupled oscillators. Each impulse-coupled oscillator has a continuous state ($\tau_i$ for the $i$-th oscillator) defining its internal timer. Once the timer of any oscillator reaches a threshold ($\tb$), it triggers an impulse and is reset to zero. At such an event, all the other impulse-coupled oscillators \startmodif rescale their timer by a factor given \stopmodif by $(1+ \varepsilon)$ times the value of their timer,
where $\varepsilon \in (-1,0)$.\footnote{Cf. the model for synchronization in \cite{Mirollo.90.SIAMJAM.BiologicalOscillators} where $\varepsilon > 0$.} 
Figure~\ref{fig:example} shows a trajectory of two impulse-coupled oscillators {with states} $\ton$ and $\ttw$. In this figure, the \startmodif dark \stopmodif  red circles indicate when a timer state has reached the threshold and, thus, resets to zero. The \startmodif light \stopmodif green circles indicate when an oscillator is externally reset and, hence,  decreases its timer by $(1+\varepsilon)$ times its current state. \stopmodif

\IfConf{
% For Journal
\begin{figure}
\centering
\psfrag{tau1}[][][.7]{\hspace{0cm}$\ton$}
\psfrag{tau2}[][][.7]{\hspace{0cm}$\ttw$}
\psfrag{tb}[][][.8]{$\tb$}
\psfrag{t1t2}[][][.8][-90]{\hspace{-.5cm}$\tau_{1},\ttw$}
\psfrag{t0}[][][.6]{\hspace{-.2cm}$\Delta t_{0}$}
\psfrag{t1}[][][.6]{\hspace{-.2cm}$\Delta t_{1}$}
\psfrag{t2}[][][.6]{\hspace{-.2cm}$\Delta t_{2}$}
\psfrag{t3}[][][.6]{\hspace{-.2cm}$\Delta t_{3}$}
\psfrag{t4}[][][.6]{\hspace{-.2cm}$\Delta t_{4}$}
\psfrag{t5}[][][.6]{\hspace{-.2cm}$\Delta t_{5}$}
\psfrag{et1}[][][.6]{$(\varepsilon+1)\tau_1$}
\psfrag{et2}[][][.6]{$(\varepsilon+1)\tau_2$}
\psfrag{t}[][][.8]{$t$ [seconds]}
\psfrag{T8}[][][.8]{$t_{8}$}
\includegraphics[width = .43\textwidth]{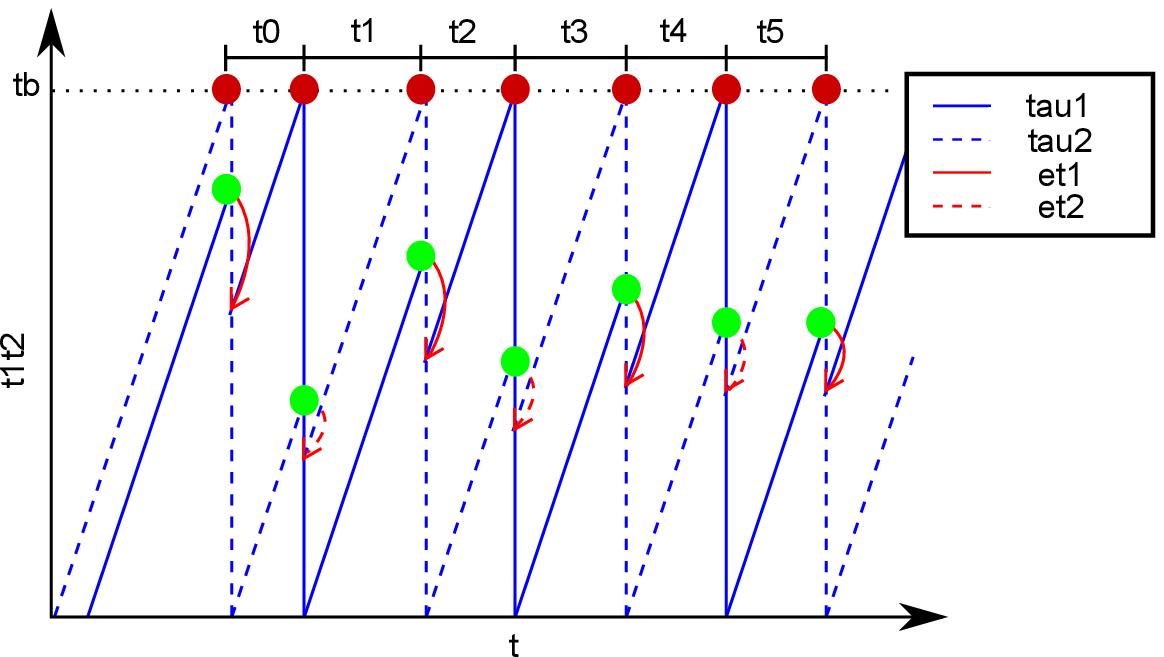}
\caption{\stopmodif An example of two {impulse-coupled }oscillators reaching desynchronization (as $\Delta t_{i}$ converges to a constant.) The internal resets (\startmodif dark \stopmodif red circles) map the timers to zero. The external resets (\startmodif light \stopmodif green circles) map the timers to a fraction $(1+\e)$ of their current value.}
\label{fig:example}
\end{figure}
}{
% For Tech Report
\begin{figure}
\centering
\psfragfig*[width = .7\textwidth]{Figures/example}{
\psfrag{tau1}[][][.9]{\hspace{.1cm}$\ton$}
\psfrag{tau2}[][][.9]{\hspace{.1cm}$\ttw$}
\psfrag{tb}[][][1]{$\tb$}
\psfrag{t1t2}[][][1][90]{\hspace{-.5cm}$\tau_{1},\ttw$}
\psfrag{t0}[][][.8]{$\Delta t_{0}$}
\psfrag{t1}[][][.8]{$\Delta t_{1}$}
\psfrag{t2}[][][.8]{$\Delta t_{2}$}
\psfrag{t3}[][][.8]{$\Delta t_{3}$}
\psfrag{t4}[][][.8]{$\Delta t_{4}$}
\psfrag{t5}[][][.8]{$\Delta t_{5}$}
\psfrag{et1}[][][.9]{$(\varepsilon+1)\tau_1$}
\psfrag{et2}[][][.9]{$(\varepsilon+1)\tau_2$}
\psfrag{t}[][][1]{$t$ [seconds]}
\psfrag{T8}[][][.8]{$t_{8}$}}
\caption{An example of two {impulse-coupled }oscillators reaching desynchronization (as $\Delta t_{i}$ converges to a constant.) The internal resets ( dark red circles) map the timers to zero. The external resets ( light green circles) map the timers to a fraction $(1+\e)$ of their current value.}
\label{fig:example}
\end{figure}}

According to this outline of the model, the dynamics of the impulse-coupled oscillators involve impulses and timer resets, which are treated as true discrete events and instantaneous updates, while the smooth evolution of the timers before/after these events define the continuous dynamics.  We follow the hybrid formalism of \cite{teel2012hybrid,Goebel.ea.09.CSM}, where a hybrid system is given by four objects $(C,f,D,G)$ defining its \textit{data}:
\begin{itemize}
\item \textit{Flow set:} a set $C \subset \reals^{N}$ specifying the points where flows are possible (or continuous evolution).
\item \textit{Flow map:} a single-valued map $f: \reals^{N} \to \reals^{N}$ defining the flows.
\item \textit{Jump set:} a set $D \subset \reals^{N}$ specifying the points where jumps are possible (or discrete evolution).
\item \textit{Jump map:} a set-valued map $G: \reals^{N} \rightrightarrows\reals^{N}$ defining the jumps.
\end{itemize}
A hybrid system capturing the dynamics of the impulse-coupled oscillators is denoted as $\HS_N := (C,f,D,G)$ and can be written in the compact form
\begin{equation}\HS_N: \qquad \tau \in \reals^{N} \qquad \left\{ \begin{array}{llll} \dot{\tau} &=& f(\tau) &\quad \tau \in C \\ \tau^{+} &\in& G(\tau) & \quad \tau \in D \end{array}\right. , \label{eqn:HS}\end{equation}
where $N \in \nats \setminus \{0,1\}$ is the number of impulse-coupled oscillators. The state of $\HS_{N}$ is given by
\IfConf{$\tau:= \left[  \ton \ \ \ttw \ \ \ldots \ \ \tau_{N} \right]^\top \in P_N := [0,\tb]^{N}.$}{$$
\tau:= \left[  \ton \ \ \ttw \ \ \ldots \ \ \tau_{N} \right]^\top \in P_N := [0,\tb]^{N} .
%\tau:= \left[ \begin{array}{c} \ton \\ \ttw \\ \vdots \\ \tau_{N}\end{array} \right] \in P_N := [0,\tb]^{N} .
$$}
The flow and jump sets are defined to constrain the evolution of the timers. The flow set is defined by
\begin{equation}
\startmodif C := P_N, \stopmodif
\label{eqn:flawiest}
\end{equation}
where $I := \{1,2, \ldots , N\}$ and $\tb > 0$ is the threshold.
\startmodif During flows, an internal clock gradually increases based on the homogeneous rate, $\omega$. \stopmodif Then, the flow map is defined as 
\IfConf{$f(\tau) := \omega\one$ for all $\tau \in C$}{$$
f(\tau) := \omega\one \qquad \forall \tau \in C
$$}
with $\omega > 0$ \sean{defining the natural frequency of each impulse-coupled oscillator}.
The impulsive events are captured by a jump \sean{set $D$ and a jump} map $G$. Jumps occur when the state is in the jump set $D$ defined as
\begin{equation}
D := \left\{ \tau \in P_N : \ \exists i \in I \ \mbox{s.t.} \ \tau_i = \tb \right\} . 
\label{eqn:D}
\end{equation}
From such points, the $i$-th timer is reset to zero and forces a jump of all other timers. Such discrete dynamics are captured by the following jump map: for each $\tau \in D$ define
$
G(\tau) = \left[ g_1(\tau) \ \ g_2(\tau)\ \ \ldots \ \ g_N(\tau)  \right]^\top,
$ 
where, for each $i \in I$, 
%\NotForConf{
%\begin{equation}
%g_i (\tau) = \left\{ \begin{array}{ll} 0 & \mbox{if } \tau_{i} = \tb, \tau_r < \tb \ \ \ \ \forall j \in I\setminus \{i\} \\ \{0 , \tau_{i}(1+\varepsilon) \} & \mbox{if } \tau_{i} = \tb \ \ \ \ \exists j \in I \setminus \{i\} \  \mbox{s.t.} \  \tau_r = \tb \\ (1+\varepsilon)\tau_{i} & \mbox{if } \tau_i < \tb \ \ \ \ \exists j \in I \  \mbox{s.t.} \  \tau_r = \tb\end{array} \right. \label{eqn:gi}
%\end{equation}
%}
\begin{equation}
g_i (\tau) = \left\{ \begin{array}{l} 0 \qquad \qquad \ \ \ \ \mbox{if } \tau_{i} = \tb, \tau_r < \tb \ \ \forall r \in I\setminus \{i\} \\ 
\{0 , \tau_{i}(1+\varepsilon) \} \ \mbox{if } \tau_{i} = \tb \ \exists r \in I \setminus \{i\} \  \mbox{s.t.} \  \tau_r = \tb \\ 
(1+\varepsilon)\tau_{i} \qquad \  \mbox{if } \tau_i < \tb \ \exists r \in I \setminus \{i\} \  \mbox{s.t.} \  \tau_r = \tb\end{array} \right. \label{eqn:gi}
\end{equation}
with parameters $\varepsilon \in (-1,0)$ and $\tb > 0$; for $\tau \in D$, $g_i$ is not empty. 
When a jump is triggered, the state $\tau_i$ jumps according to the $i$-th component of the jump map $g_i$. When a state reaches the threshold $\tb$, it is reset to zero only when all other states are less than that threshold; otherwise, if multiple timers reach the threshold simultaneously, the jump map is set valued to indicate that either $g_i(\tau) = 0$ or $g_i(\tau) = (1+\varepsilon)\tau_{i}$ is possible. This is to ensure that the jump map satisfies the regularity conditions outlined in Section~\ref{sec:ModelForAnalysis}.\footnote{In \cite{mauroy:037122}, a more general flow map and a jump map incrementing $\tau_i$ by $\e > 0$ are considered.}

%%%% BEGIN CHANGE - RICARDO
%\IfConf{
%Solutions to the hybrid system $\HS_N$
%evolve continuously (flow) and/or discretely (jump) depending on the continuous and discrete dynamics and the sets where those dynamics apply. 
%As in \cite{Goebel.ea.09.CSM}, 
%we treat the number of jumps as an independent variable $j$ and the amount of time of flows by the independent variable $t$. Then,
%solutions $\tau$ to $\HS_{N}$ are given by {\em hybrid arcs} 
%parameterized by $(t,j)$ which takes values on 
%the {\em hybrid time domain} $\dom \tau$; see \cite{teel2012hybrid,Goebel.ea.09.CSM} for more details.\footnote{A solution $\tau$ is said to be {\it nontrivial} if $\dom \tau$ contains at least one point different from $(0,0)$, {\it maximal} if there does not exist a solution $\tau'$ such that $\tau$ is a truncation of $\tau'$ to some proper subset of $\dom \tau'$, {\it complete} if $\dom \tau$ is unbounded, and {\it Zeno} if it is complete but the projection of $\dom \tau$ onto $\realsgeq$ is bounded.}
%}{}
%%%% END CHANGE - RICARDO

%\NotForConf{
%%%% BEGIN CHANGE - RICARDO
%Next, we will present two desynchronizing impulse-coupled oscillator examples using the above model and also present numerical simulations demonstrating the evolution of these systems.
%%%% END CHANGE - RICARDO
\NotForConf{For example, consider the case $N=2$ the hybrid system $\HS_N = (C,f,D,G)$ has state given by
$$\tau=\left[ \begin{array}{c} \tau_{1} \\ \tau_{2}\end{array} \right] \in P_2 := [0,\tb]\times[0,\tb] .$$ 
The states $\tau_{1}$ and $\tau_{2}$ are the timers for both of the oscillators. The hybrid system $\HS_2$ has the following data:

%\NotForConf{
%$$
%\HS_2 = \left\{\begin{array}{ll}C = P_2, & f(\tau) = \left[ \begin{array}{c} 1 
%\\ 1 \end{array}\right] \forall \tau \in C, \\ \noalign{\medskip}
%D = \left\{ \tau \in P_2 \ : \ \max_{i \in \{1,2\}} \tau_{i} = \tb \right\} , & G(\tau) = \left[ \begin{array}{c} g_1(\tau) \\ g_2(\tau) \end{array} \right] \forall \tau \in D \end{array} \right. ,
%$$
%}

$$
\HS_2 = \left\{\begin{array}{ll} C = P_2, & \qquad f(\tau) = \left[ \begin{array}{c} 1
\\ 1 \end{array}\right] \forall \tau \in C, \\ \noalign{\medskip} 
D = \left\{ \tau \in P_2 \ : \ \exists i \in \{1,2\} \ s.t. \ \tau_{i} = \tb \right\}, & \qquad
G(\tau) = \left[ \begin{array}{c} g_1(\tau) \\ g_2(\tau) \end{array} \right] \forall \tau \in D , \end{array} \right. 
$$
where the functions $g_1$ and $g_2$ are defined as
$$
g_1 (\tau) = \left\{ \begin{array}{ll}  0 & \mbox{if } \tau_{1} = \tb, \tau_2 < \tb \\ \{0 , \tau_{1}(1+\varepsilon) \} & \mbox{if } \tau_{1} = \tb, \tau_2 = \tb \\
(1+\varepsilon)\tau_{1} & \mbox{if } \tau_1 < \tb, \tau_2 = \tb 
\end{array} \right.
\qquad \qquad  g_2 (\tau) = \left\{ \begin{array}{ll}  0 & \mbox{if } \tau_{2} = \tb, \tau_1 < \tb \\ \{0 , \tau_{2}(1+\varepsilon) \} & \mbox{if } \tau_{2} = \tb, \tau_1 = \tb \\
(1+\varepsilon)\tau_{2} & \mbox{if } \tau_2 < \tb, \tau_1 = \tb 
\end{array} \right. .
$$}
\subsection{Basic Properties of $\HS_{N}$}\label{sec:BasicCond}
\label{sec:ModelForAnalysis}
\subsubsection{Hybrid Basic Conditions}
% Checked - RS
To apply analysis tools for hybrid systems in \cite{teel2012hybrid}, which will be summarized in Section~\ref{sec:analysis},
the data of the hybrid system $\HS_{N}$
must meet certain mild conditions. 
These conditions, referred to as the {\em hybrid basic conditions}, are as follows:
%\IfConf{ A1) $C$ and $D$ are closed sets in $\reals^N$.
%A2) $f: \reals^N\to\reals^N$ is continuous on $C$.
%A3) $G :\reals^N\tto\reals^N$ is an outer semicontinuous\footnote{A set-valued mapping $G : \reals^N\tto\reals^N$ is {\em outer semicontinuous} if its graph $\{(x,y): x \in \reals^n, y \in G(x)\}$ is closed, see \cite[Lemma 5.10]{teel2012hybrid} and \cite{RockafellarWets98}.
%}
% set-valued mapping, locally bounded on $D$, and such that $G(x)$ is nonempty for each $x \in D$. }{
\begin{enumerate}
\item[A1)] $C$ and $D$ are closed sets in $\reals^N$.
\item[A2)] $f: \reals^N\to\reals^N$ is continuous on $C$.
\item[A3)] $G :\reals^N\tto\reals^N$ is an outer semicontinuous\footnote{A set-valued mapping $G : \reals^N\tto\reals^N$ is {\em outer semicontinuous} if its graph $\{(x,y): x \in \reals^N, y \in G(x)\}$ is closed, see \cite[Lemma 5.10]{teel2012hybrid} and \cite{RockafellarWets98}.
}
 set-valued mapping, locally bounded on $D$, and such that $G(x)$ is nonempty for each $x \in D$.
\end{enumerate}
%}
\begin{lemma}\label{lem:BasicConds}
%The data of 
$\HS_{N}$ satisfies the hybrid basic conditions.
\end{lemma}
\IfConf{
\begin{proof}
For a proof of Lemma~\ref{lem:BasicConds} see \cite{Phillips2013TechReport}.
\end{proof}
}{
\begin{proof} Condition (A1) is satisfied since $C$ and $D$ are closed. The function $f$ is constant and therefore continuous on $C$, satisfying (A2). With $G$ as in \eqref{eqn:gi}, the graph of each $g_i$ is defined as 
\IfConf{\begin{align*}
\mbox{gph}(g_i) &= \{(x,y) : y \in g_i(x), x \in D\} \\
&= \{(x,y) : y = 0, x_{i}  = \tb, x_{r} \leq \tb \ \forall r \neq i, x \in D\}  \\ & \cup \{(x,y) : y = (1+\varepsilon)x_{i}, x_{i} \leq \tb \ \exists x_{r} = \tb, x \in D\} 
\end{align*}
}{
\begin{align*}
\mbox{gph}(g_i) &= \{(x,y) : y \in g_i(x), x \in D\} \\
&= \{(x,y) : y = 0, x_{i}  = \tb, x_{r} \leq \tb \ \forall r \neq i, x \in D\} \cup \{(x,y) : y = (1+\varepsilon)x_{i}, x_{i} \leq \tb \ \exists x_{r} = \tb, x \in D\} 
\end{align*}
}
which is closed. Then the set-valued mapping $G$ is outer semicontinuous. By definition, $G$ is bounded and nonempty for each $\tau \in D$, and hence it satisfies (A3). \EndPF
\end{proof}}

Note that satisfying the hybrid basic conditions \sean{implies that $\HS_N$ is well-posed \cite[Theorem 6.30]{teel2012hybrid}, which automatically gives robustness to vanishing state disturbances; see \cite{teel2012hybrid,Goebel.ea.09.CSM}.
Section~\ref{sec:robustness} considers different types of perturbations that $\HS_N$ \startmodif can withstand. \stopmodif}

\subsubsection{Solutions to $\HS_N$}
\IfConf{
%%%% BEGIN CHANGE - RICARDO
Solutions to general hybrid systems $\HS$ ($\HS_N$ in particular)
can evolve continuously (flow) and/or discretely (jump) depending on the continuous and discrete dynamics and the sets where those dynamics apply. We treat the number of jumps as an independent variable $j$ and the time of flow by the independent variable $t$. More precisely, we parameterize the state by $(t,j)$. Solutions to $\HS$ will be given by {\em hybrid arcs} on {\em hybrid time domains} \cite{teel2012hybrid,Goebel.ea.09.CSM}.
%%%% END CHANGE - RICARDO
}{}
% Checked - RS
\IfConf{In the context of hybrid systems, a subset of $\reals_{\geq 0}\times \nats$ is a hybrid time domain if it is of the form $\cup_{j = 0}^{J}([t_{j},t_{j+1}] \times \{j\})$, with $0 = t_{0} \leq t_{1} \leq t_{2} \leq \hdots$, where $J \in \nats \cup \{\infty\}$.}{
Solutions to generic hybrid systems $\HS$ with state $x \in \reals^n$ will be given by {\em hybrid arcs} on {\em hybrid time domains} defined as follows:
\begin{definition}{(hybrid time domain)}
\label{hybrid time domain definition}                
A subset $S\subset\realsgeq\times\nats$ is a {\it compact hybrid time domain} if       
$$S=\bigcup_{j=0}^{J-1} \left([t_j,t_{j+1}],j\right)$$
for some finite sequence of times $0=t_0\leq t_1 \leq t_2\ ...\leq t_J$.  A subset $S\subset\realsgeq\times\nats$ is a {\it hybrid time domain} if for all $(T,J)\in S$, $S\ \cap\ \left( [0,T]\times\{0,1,...J\}\right)$ is a compact hybrid time domain.              
\end{definition}
\begin{definition}{(hybrid arc)}
\label{hybrid arc}                
 A function $x:\dom x\to\reals^n$ is a {\it hybrid arc} if $\dom x$ is a hybrid time domain and if for each $j\in\nats$, the function $t\mapsto x(t,j)$ is locally absolutely continuous. 
\end{definition}
\begin{definition}{(solution)}\label{def:solution}
\label{solution}                
A hybrid arc $x$ is a {\em solution to the hybrid system $\HS$} if $x(0,0) \in C \cup D$ and:
\begin{enumerate}
\item[(S1)] For all $j\in\nats$ and almost all $t$ such that 
$(t,j)\in \dom x$,  
\begin{equation} \non x(t,j) \in C, \ \ \ \dot{x}(t,j) = f(x(t,j))\ . \end{equation} 
\item[(S2)] For all $(t,j)\in \dom x$ such that $(t,j+1)\in \dom x$, 
\begin{equation} \non x(t,j) \in D, \ \ \ x(t,j+1)\in G(x(t,j))\ . \end{equation} 
\end{enumerate}
\end{definition}
}
% Checked - RS
\IfConf{}{
%%%% BEGIN CHANGE - RICARDO
A solution $x$ is said to be {\it nontrivial} if $\dom x$ contains at least one point different from $(0,0)$, {\it maximal} if there does not exist a solution $x'$ such that $x$ is a truncation of $x'$ to some proper subset of $\dom x'$, {\it complete} if $\dom x$ is unbounded, and {\it Zeno} if it is complete but the projection of $\dom x$ onto $\realsgeq$ is bounded.
%%%% END CHANGE - RICARDO
}

\begin{lemma}\label{lem:solutions2HS}
From every point in $C \cup D$, there exists a solution and every maximal solution to $\HS_{N}$ is complete and bounded. 
\end{lemma}

\IfConf{
\begin{proof}
For a proof of Lemma~\ref{lem:solutions2HS} see \cite{Phillips2013TechReport}.
\end{proof}
}{
\begin{proof}
The result follows from Proposition 2.10 in \cite{teel2012hybrid} using the following properties. For each point such that $\tau \in C$, the components of the flow map $f$ are positive and induce solutions that flow towards $D$. For each $\tau \in D$, the jump map satisfies $G(\tau) \subset C$. Since it is impossible for solutions with initial conditions $\tau(0,0) \in C\cup D$ to escape $C \cup D$, all maximal solutions are complete and bounded.  \EndPF
\end{proof}
}
\startmodif
Due to the jump map $G$, if the elements of the solution are initially equal (denote this set as $\S := \{\tau \in P_N : \exists i,r \in I, i \neq r, \tau_{i} = \tau_r \}$) it is 
possible for them to 
remain equal for all time.
Furthermore, it is also possible for solutions to be 
initialized on the jump set such that one element is at the threshold and another 
is equal to zero
then after the jump they will be equal, e.g. let 
$\tau_{1} = \tb$, $\tau_{2} = 0$ then $\tau_{1}^{+} = \tau_{2}^{+} = 0$. We denote 
this set as 
$\G := \{\tau \in D\setminus \S : \exists i,r \in I, i \neq r, \tau_{i} = 0, \tau_r = \tb \}$. The next result considers solutions initialized on the set $\X := \S \cup \G$. 
\stopmodif
\begin{lemma}\label{lem:defX}
For each $\tau(0,0) \in \X$, there exists a solution 
$\tau$ to $\HS_{N}$ from $\tau(0,0)$ such that, for some $M \in \{0,1\}$, $\tau(t,j) \in \S$ for all $t+j \geq M$, $(t,j) \in \dom \tau$.
\end{lemma}
\begin{proof}
Consider a solution $\tau$ to the hybrid system $\HS_N$ with initial condition $\tau(0,0) \in \startmodif \S \stopmodif$. Due to the flow map for each state being equal, $\tau$ remains in $\S$ during flows. 
Furthermore, at points $\tau \in \S \cap D$, the jump map $G$ is set valued by the definition of $g_i$ in \eqref{eqn:gi}. From these points,  $G(\tau) \cap \S \neq \emptyset$. In fact, for each $\tau(0,0) \in \S$, there exists at least one solution such that $\tau(t,j) \in \S$ for all $t + j \geq 0$, with $(t,j) \in \dom \tau$.
Consider the case of solutions initialized at
$\tau(0,0) \in \startmodif \G \stopmodif$ (Note that $\tau(0,0) \in D$). It follows that for some $r \in I$, $\tau_r(0,0) = \tb$ and $g_r(\tau(0,0)) = 0$. Therefore, after the initial jump, we have that $G(\tau(0,0)) \cap \S \neq \emptyset,$ by which using previous arguments implies that $\tau(t,j) \in \S$ for all $t + j \geq 1$. 
\end{proof}

Furthermore, there is a distinct ordering to the jumps. If $\tau$ is such that $\tau_i \neq \tau_r$ for all $i\neq r$ then the ordering of each $\tau_i$ is preserved after $N$ jumps. More specifically, we have the following result.
\begin{lemma} \label{lem:ordering}
For every solution $\tau$ to $\HS_N$ with $\tau(0,0) \notin \X$, if at $(t_{j},j) \in \dom \tau$ we have
\IfConf{$
0 \leq \tau_{i_1}(t_j,j) < \tau_{i_2}(t_j,j) < ... < \tau_{i_N}(t_j,j) \leq \tb
$}{$$
0 \leq \tau_{i_1}(t_j,j) < \tau_{i_2}(t_j,j) < ... < \tau_{i_N}(t_j,j) \leq \tb
$$}
for some sequence of nonrepeated elements $\{i_m\}^N_{m = 1}$ 
of $I$
%with $i_m \in I$ for all $m \in I$, 
(that is, a reordering of the elements of the set $I = \{1,2,\ldots,N\}$) then, after $N$ jumps, it follows that
\IfConf{$
0 \leq \tau_{i_1}(t_{j+N},j+N) < \tau_{i_2}(t_{j+N},j+N) < ... < \tau_{i_N}(t_{j+N},j+N) \leq \tb.
$}{$$
0 \leq \tau_{i_1}(t_{j+N},j+N) < \tau_{i_2}(t_{j+N},j+N) < ... < \tau_{i_N}(t_{j+N},j+N) \leq \tb.
$$}
\end{lemma}

\begin{proof}
Let $\tau$ be a solution to $\HS_{N}$ from $P_{N}\setminus \X$. There exists a sequence $i_{k}$ of distinct elements with $i_{k} \in I$ for each $k \in I$, such that $0 \leq \tau_{i_{1}}(t,j) < \tau_{i_{2}}(t,j) < \ldots < \tau_{i_{N}}(t,j) \leq \tb$ over $[t_{0},t_{1}]\times \{0\}$. After the jump at $(t,j) = (t_{1},0)$ we have $0 = \tau_{i_{N}}(t,j+1) < \tau_{i_{1}}(t,j+1) < \tau_{i_{2}}(t,j+1) < \ldots <  \tau_{i_{N-1}}(t,j+1) < \tb$. Continuing this way for each jump, it follows that after $N-1$ more jumps, the solution is such that $0 \leq \tau_{i_{1}}(t_{N},j+N) < \tau_{i_{2}}(t_{N},j+N) < \ldots < \tau_{i_{N}}(t_{N},j+N) \leq \tb$ and the order at time $(t,j)$ is preserved.
\end{proof}
Using these properties of solutions to $\HS_N$, the next section defines the set to which these solutions converge and establishes its stability properties.

\section{Dynamical Properties of $\HS_{N}$}
\label{sec:analysis}

Our goal is to show that the desynchronization configuration \sean{of $\HS_N$, which is defined in Section~\ref{sec:AN}}, is asymptotically stable.
We recall from \cite{teel2012hybrid,Goebel.ea.09.CSM} the following definition of asymptotic stability for general hybrid systems with state $x \in \reals^n$.

\begin{definition}[stability]
\label{def:AS}
A closed set $\A \subset \reals^n$ is said to be
\begin{itemize}
\item {\em stable} if for each $\varepsilon>0$ there exists
$\delta>0$ such that each solution $x$ with $|x(0,0)|_{\A}\leq \delta$
satisfies
$|x(t,j)|_{\A} \leq \varepsilon$ for all $(t,j)\in\dom x$; 
\item {\em attractive} if there exists $\mu > 0$ such that every maximal solution
$x$
with $|x(0,0)|_{\A}\leq \mu$ is 
complete and
satisfies \\ $\lim_{(t,j) \in \dom x, t+j\to\infty} |x(t,j)|_{\A}=0$;
\item {\em asymptotically stable} if stable and attractive;
\item {\em weakly globally asymptotically stable} if $\A$ is stable and if, for every initial condition, there exists a maximal solution that is complete and satisfies $\lim_{(t,j) \in \dom x, t+j\to\infty} |x(t,j)|_{\A}=0$.
\end{itemize}
\end{definition}

The set of points from where the attractivity property holds is
the basin of attraction and excludes all points where the system trajectories may never converge to $\A$. In fact, 
it will be established in Section~\ref{sec:lyapunov} that the basin of attraction for asymptotic stability of desynchronization of $\HS_N$ does not include any point $\tau$ such that any two or more timers are equal or become equal after a jump, which is the set $\X$ defined in Lemma~\ref{lem:defX}. \IfConf{For this purpose, a Lyapunov-like function will be constructed in Section~\ref{sec:lyapunov} to show that a compact set denoted $\A$, defining the desynchronization condition, is asymptotically stable and weakly globally asymptotically stable.}{}
%This set is denoted by $\X$ and is defined as
%$
%\X := \S \cup \G,
%$
%where 
%$\S := \{\tau \in P_N : \exists i,j \in I, i \neq j, \tau_{i} = \tau_r \}$
%are the points where at least two timers are equal and 
%$\G := \{\tau \in P_N : \exists i,j \in I, i \neq j, \tau_{i} = g_{j}(\tau) \}$
%is the set where timers become equal after a jump.
\NotForConf{For example, consider the case $N = 2$, i.e.,  $\HS_2$. Then, the set $\X_2$ is defined as 
\begin{align}
\begin{split}
\X_2 &= \S_2 \cup \G_2 =\left(\{ \tau \in P_2: \ton = \ttw \}\right) \cup\left(\{\tau \in D : g_1(\tau) = \tau_2\} \cup \{ \tau \in D : g_2(\tau) = \tau_1 \} \right).
\end{split} \label{eqn:X2} 
\end{align}
Note that the set $\S_2$ defines the line $\ton = \ttw$ in $P_2$ and $\G_2$ is given by the points $\{(0,\tb), (\tb, 0)\}$ in $P_2$; see Figure~\ref{fig:A2}.
For $N=3$, the set $\X_3$ is defined as 
\begin{align}
\X_3 = \S_3 \cup \G_3 \label{eqn:X3}
\end{align}
where
\begin{align}
\S_3 = &\{\tau \in P_3 \ : \tau_1 = \tau_2 \}\cup \{\tau \in P_3 \ : \tau_1 = \tau_3 \} \cup \{\tau \in P_3 \ : \tau_2 = \tau_3\}
\end{align}
and
\begin{align}
\begin{split}
\G_3 = &\{\tau \in P_3 \ : g_1(\tau) = \tau_3\} \cup \{\tau \in P_3 \ : g_2(\tau) = \tau_1\} \cup\{\tau \in P_3 \ : g_3(\tau) = \tau_1\}\cup \{\tau \in P_3 \ : g_2(\tau) = \tau_3\} \\ & \qquad \cup \{\tau \in P_3 \ : g_3(\tau) = \tau_2\}  \cup \{\tau \in P_3 \ : g_1(\tau) = \tau_2\}. 
\end{split}
\end{align}
Then, $\X_3$ is defined by the union of $\S_3$, which is the as the union of the planes in $P_3$ given by $\tau_1 = \tau_2$, $\tau_1 = \tau_3$, and $\tau_2 = \tau_3$, and $\G_3$, which is given by $\{(\tb,\tau_2,0), (0,\tau_2,\tb),(\tau_1,\tb,0), \\ (\tau_1,0,\tb), (\tb,0,\tau_3), (0,\tb,\tau_3): \tau \in P_3\}$; see Figure~\ref{fig:X3BadBox}.}
\NotForConf{For this purpose, a Lyapunov-like function be constructed in Section~\ref{sec:lyapunov} to show that a compact set denoted $\A$, defining the desynchronization condition, is asymptotically stable and weakly globally asymptotically stable.}

\NotForConf{\begin{figure}
%Figures In Tech Report Only
\centering
  \subfigure[The sets $\X_2$ \sean{(dashed green)} and $\A$ \sean{(solid blue and red)}. The set $\A$ is defined by the union of $\ell_{1}$ and $\ell_{2}$ (See Section~\ref{sec:AN}).]{
    \psfragfig*[width=.4\textwidth]{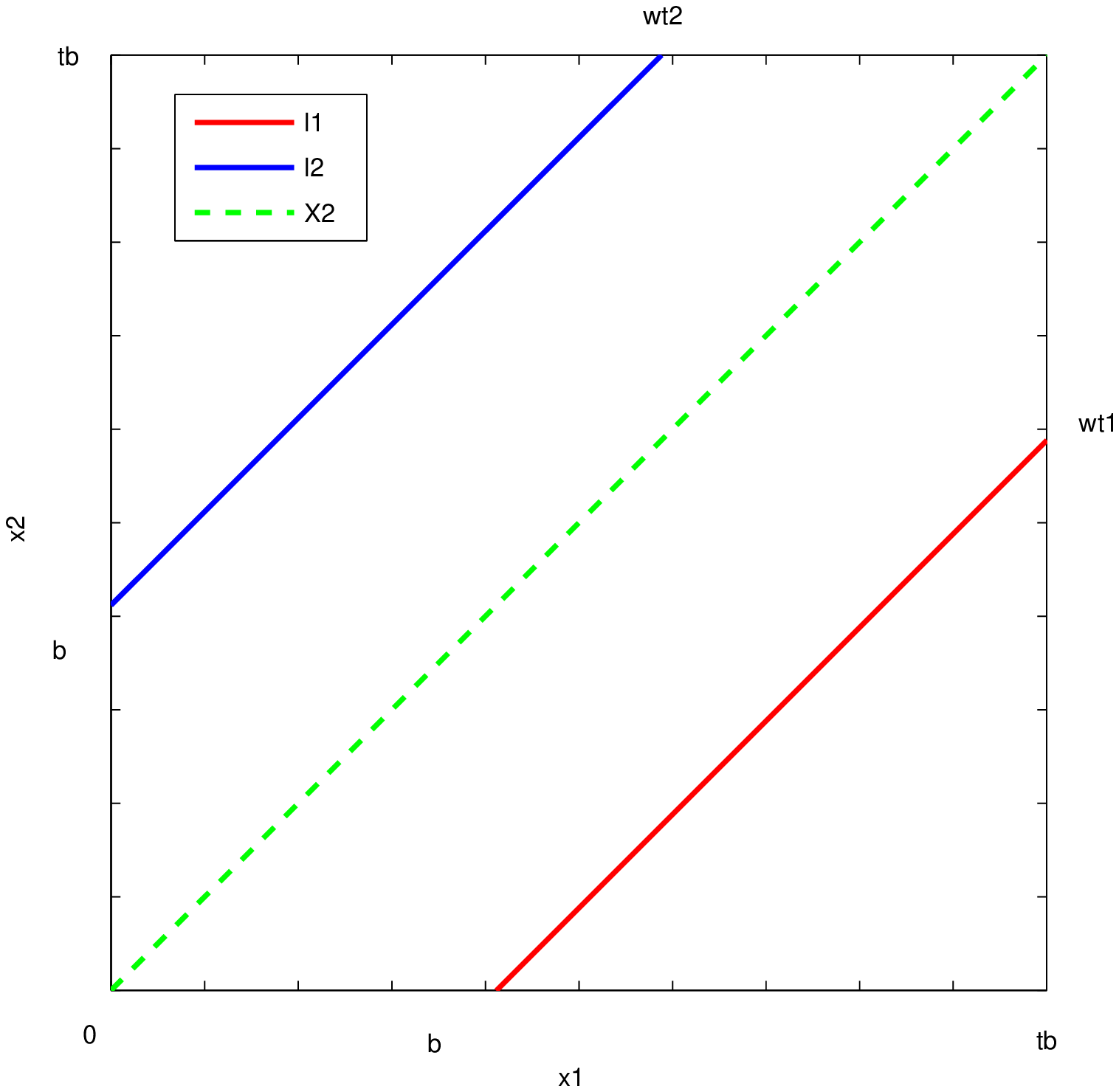}{
         \psfrag{0}[][][1]{$0$}
    \psfrag{x1}[][][1]{$\ton$}
   \psfrag{x2}[][][1][-90]{$\ttw$}
   \psfrag{x0}[][][.8]{\hspace{5mm}$\tau(0,0)$}
   \psfrag{tb}[][][1]{$\tb$}
   \psfrag{b}[][][1]{$B$}
   \psfrag{l1}[][][.6]{\hspace{1mm}$\ell_{2}$}
   \psfrag{l2}[][][.6]{\hspace{.1cm}$\ell_{1}$}
   \psfrag{X2}[][][.6]{\hspace{.1cm}$\X_2$}
   \psfrag{wt1}[][][.9]{$\wt\tau_2$}
   \psfrag{wt2}[][][.9]{$\wt\tau_1$}   
    }
  \label{fig:A2}}\ \
  \subfigure[A simulation (dashed blue) of $\HS_{2}$ showing the attractivity of the set $\A$ (solid black).]{
    \psfragfig*[width=0.4\textwidth]{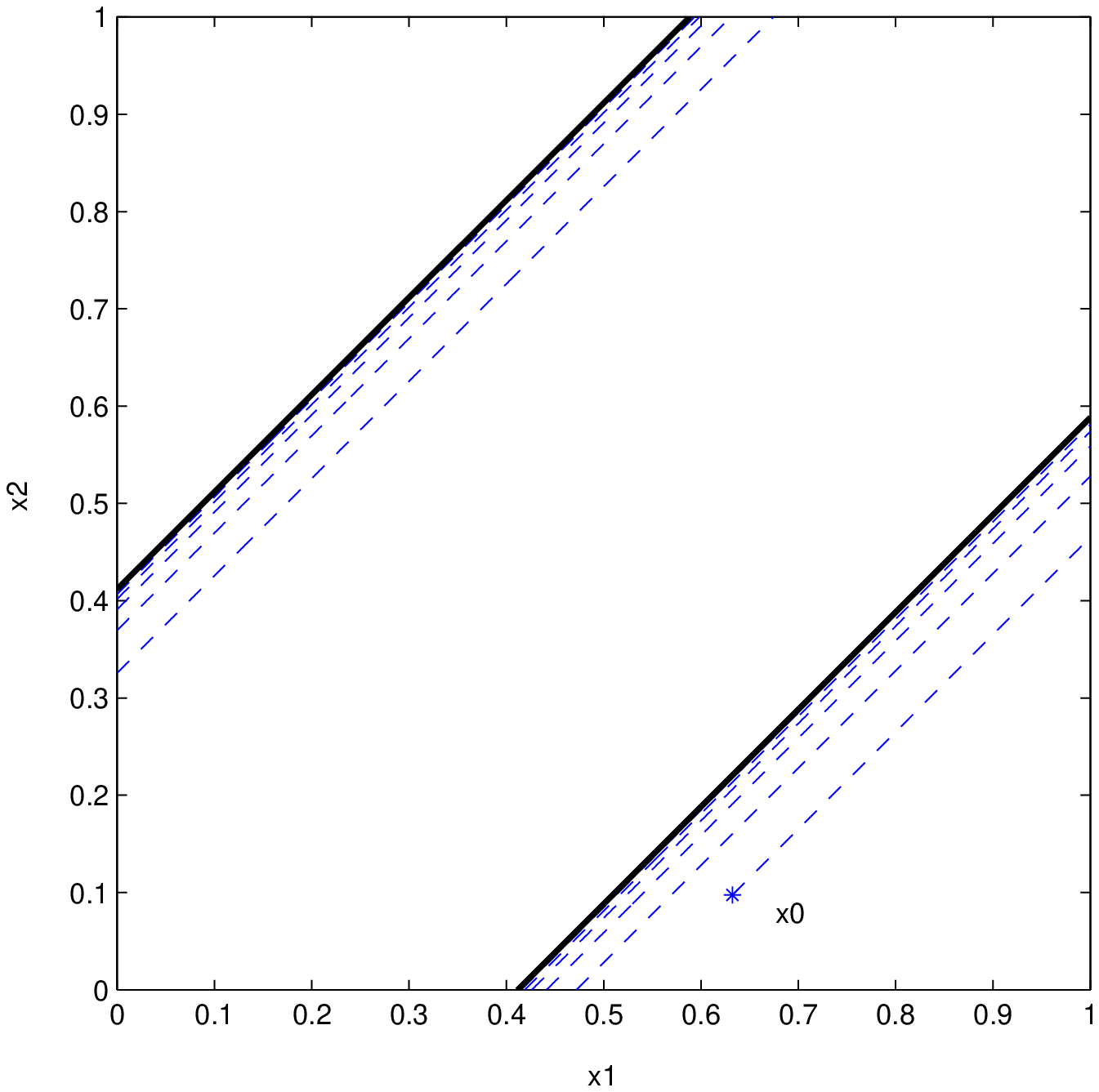}{
    \psfrag{x1}[][][1]{$\ton$}
   \psfrag{x2}[][][1][-90]{$\ttw$}
   \psfrag{x0}[][][.8]{\hspace{5mm}$\tau(0,0)$}
   \psfrag{tb}[][][1]{$\tb$}
   \psfrag{b}[][][1]{$B$}
   \psfrag{l1}[][][.6]{\hspace{1mm}$\ell_{2}$}
   \psfrag{l2}[][][.6]{\hspace{.1cm}$\ell_{1}$}
   \psfrag{X2}[][][.6]{\hspace{.1cm}$\X_2$}
   \psfrag{wt1}[][][.9]{$\wt\tau_2$}
   \psfrag{wt2}[][][.9]{$\wt\tau_1$}   
    }
  \label{fig:A2SetandTraj}}  
  \caption{Sets associated with $\HS_{2}$ and a solution to it from $\tau(0,0) = [0.7, 0.75]^\top$ with $\e = 0.3$ and $\tb = 1$.}
  \vspace{-.3cm}
\end{figure}

\begin{figure}
\centering
\psfragfig*[width=.7\textwidth]{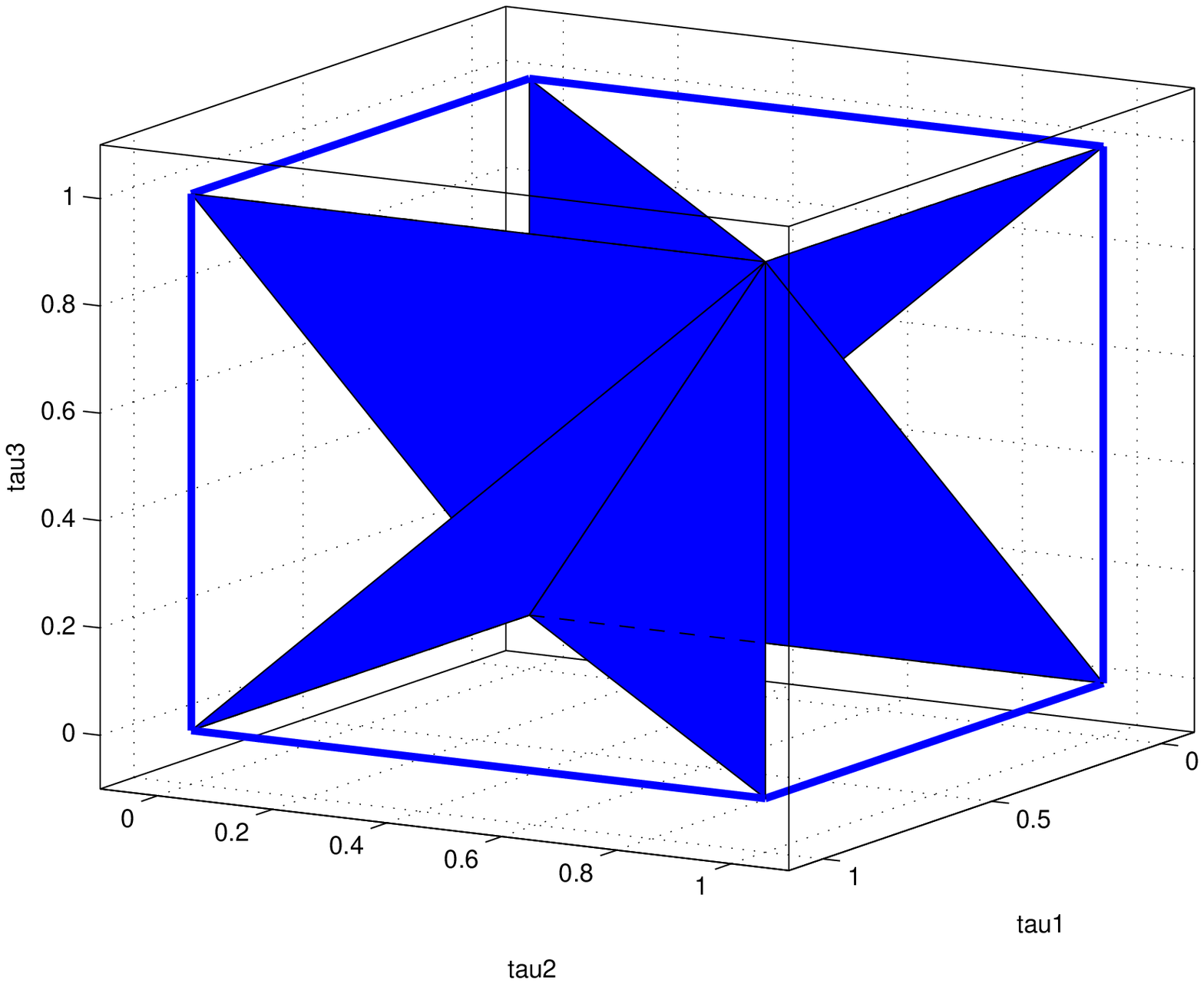}{
\psfrag{tau1}[][][1]{$\tau_1$}
\psfrag{tau2}[][][1]{$\tau_2$}
\psfrag{tau3}[][][1][-90]{\hspace{-.5cm}$\tau_3$}
}
\caption{Points in each blue plane and line belong to the set $\X_3$.}
\label{fig:X3BadBox}
\end{figure}}
\subsection{Construction of the set $\A$ for $\HS_N$}\label{sec:AN}
In this section, we identify the set of points corresponding to the impulse-coupled oscillators being 
desynchronized, namely, we define the {\em desynchronization set}. We define desynchronization as the behavior in which the separation between all of 
the timers' impulses is equal (and nonzero), see Figure~\ref{fig:example}. More specifically desynchronization is defined as follows:
\begin{definition} \label{def:desynch}
A solution $\tau$ to $\HS_N$ is desynchronized if there exists $\Delta > 0$ and a sequence of non-repeated elements $\{i_m\}^N_{m = 1}$ of  $I$ (that is, a reordering of the elements of the set $I = \{1,2, \ldots, N \}$) such that 
$\lim_{j \to \infty} (t_j^{i_m} - t_j^{i_{m+1}} )= \Delta$ for all 
$m \in \{1,2, \ldots, N-1\}$ and
$\lim_{j \to \infty} (t_j^{N} - t_j^{i_{1}}) = \Delta,$
where $\{t_j^{i_m}\}_{j = 0}^\infty$ is the sequence of jump times of the state $\tau_{i_m}$. 
\end{definition}
 
%Figure~\ref{fig:desync} depicts solutions to $\HS_2$ for which the separation 
%between $\tau_1$ and $\tau_2$ is maximized along the system trajectory.
In fact, this separation between impulses leads to an ordered 
sequence of impulse times with equal separation.  
The desynchronization set $\A$ for the hybrid system $\HS_N$ captures such a behavior and is parameterized by $\varepsilon$, the threshold $\tb$, and the number of impulse-coupled oscillators $N$. 

To define this set, first we
provide some basic intuition about the dynamics of $\HS_N$ when desynchronized. The set $\A$ must be forward invariant and such that trajectories staying in it satisfy the property in Definition~\ref{def:desynch}. Due to the definition of the flow map $f$, there exist sets in the form of ``lines" $\ell_k$, each of them in the direction $\one$, which is the direction of the flow map, intersecting the jump set at a point which, for the $k$-th line, we denote as $\wt\tau^{k}$. We define the desynchronization set as the union of sets $\ell_{k}$ collecting points $\tau = \wt\tau^k + \one s \in P_N$ parameterized by $s \in \reals$. 
\NotForConf{Figure~\ref{fig:A2} shows $\ell_1$ and $\ell_2$ (solid blue and red) for the case $N = 2$.}

To identify  $\wt\tau^k$, 
consider a point $\wt\tau^{k} \in D\setminus \X$ with components satisfying $\wt\tau_{1}^{k} = \tb > \wt\tau_{2}^{k} > \wt\tau_{3}^{k} > ... > \wt\tau_{N}^{k}$. \startmodif 
Due to Definition~\ref{def:desynch}, it must be true that the difference between jump times are constant. This means that there must be some correlation between $\Delta$ and the difference between, in this case, $\tau_{1}^{k}$ and $\tau_{2}^{k}$. Moreover, there must be a correlation between $\tau_{1}^{k}$ and all other states at jumps. 
\stopmodif
It follows that this point belongs to $\A$ only if the distance between 
the expiring timer ($\wt\tau_{1}^{k}$) 
and each of its other components ($\wt\tau_{i}^{k}$, $i \in I \setminus \{1\}$) 
is equal to the distance between 
the value after the jump 
of the timer expiring next ($\wt\tau_{2}^{k}\null^{+}$)
and the value after the jump of 
its other components ($\wt\tau_{i}^{k}\null^{+}$, $i \in I \setminus \{2\}$),
respectively. 
\startmodif
This property ensures that, when in the desynchronization set, the relative distance between the leading timer and each of the other timers is equal, before and after jumps.
%This property ensures that the timers maintain a relative distance between each other and the leading timer before and after the jumps when they are in the desynchronization set.
\stopmodif
More precisely,
% for this particular point, 
\begin{align}
\wt\tau_{1}^{k} - \widetilde{\tau}_{i}^{k} = \wt\tau_{2}^{k}\null^{+} - \wt\tau_{\mbox{\scriptsize next}(i)}^{k}\hspace{-0.21in}\null^{+} \qquad
\qquad \forall \ i \in I \setminus \{1\} \label{eqn:ANcondition},
\end{align}
where $\wt\tau^{k}\null^{+} = G(\wt\tau^{k})$ and next$(i) = i + 1$ if $i + 1 \leq N$ and $1$ otherwise.\footnote{Note that $G$ is single valued at each $\wt\tau^k \notin \X$.}
Since $\X$ contains all points such that at least two or more timers are the same, we 
can consider the case when one component of $\wt\tau^k$ is equal to $\tb$ at a time. For 
each such case, we have $(N - 1)!$ possible permutations
 of the other components and $N$ possible timer components equal to $\tb$, leading to $N!$ 
 total possible sets $\ell_{k}$. 

\NotForConf{To illustrate computation of $\wt\tau^k$ in \eqref{eqn:ANcondition} and the construction of $\A$, consider the 
case of $N =2$ and $\wt\tau_1^1 = \tb > \wt\tau_2^1$. For $i = 2$, \eqref{eqn:ANcondition} 
becomes 
$$
\tb -\wt \tau_2^1 = \wt\tau_2^1(\e + 1)
$$
which leads to $\wt\tau_2^1 = \frac{\tb}{\e +2}$. It follows that $\wt\tau^1 = [\tb, \frac{\tb}{\e +2}]^\top$. Similarly for $\wt\tau_2^1 = \tb > \wt\tau_1^1$, we get from \eqref{eqn:ANcondition} the equation $\tb - \wt\tau_1^1 = \wt\tau_2^1(\e+1)$, which implies  $\wt\tau^2 = [\frac{\tb}{\e +2}, \tb]^\top$.
A glimpse at the case for $N=3$ with $\wt\tau_1^1 = \tb > \wt\tau_2^1 > \wt\tau_3^1$ indicates that \eqref{eqn:ANcondition} leads to
\begin{align*}
\tb - \wt\tau_2^1 = \wt\tau_2^1(1+\e) - \wt\tau_3^1(1+\e), \qquad
\tb - \wt\tau_3^1 = \wt\tau_2^1(1+\e) - 0. 
\end{align*}
The solution to these equations is $\wt\tau^1 = [\tb, \tb(\e+2)/(\e^2+3\e+3),\tb/(\e^2+3\e+3)]^\top$. }

For the $N$ case, the algorithm above results in the system of equations $\Gamma \tau_s = b$, where
%\IfConf{
%\begin{figure}[!htb]
%\centering
%\psfrag{tb}[][][.8]{$\tb$}
%\psfrag{t0}[][][.8]{$t_{0}$}
%\psfrag{t1}[][][.8]{$t_{1}$}
%\psfrag{t2}[][][.8]{$t_{2}$}
%\psfrag{t3}[][][.8]{$t_{3}$}
%\psfrag{t4}[][][.8]{$t_{4}$}
%\psfrag{tau1}[][][.8]{$\tau_{1}$}
%\psfrag{tau2}[][][.8]{$\tau_{2}$}
%\psfrag{time}[][][.8]{$t$ [s]}
%\includegraphics[trim=0.1cm 0.1cm 0.5cm 0.8cm,clip,width = .3\textwidth]{Figures/desync_def2}
%\caption{A solution to $\HS_2$ that corresponds to desynchronization.}
%\label{fig:desync}
%\end{figure}}
%{
%\begin{figure}[!htb]
%\centering
%\psfrag{tb}[][][.8]{$\tb$}
%\psfrag{t0}[][][.8]{$t_{0}$}
%\psfrag{t1}[][][.8]{$t_{1}$}
%\psfrag{t2}[][][.8]{$t_{2}$}
%\psfrag{t3}[][][.8]{$t_{3}$}
%\psfrag{t4}[][][.8]{$t_{4}$}
%\psfrag{tau1}[][][.8]{$\tau_{1}$}
%\psfrag{tau2}[][][.8]{$\tau_{2}$}
%\psfrag{time}[][][.8]{$t$ [s]}
%\includegraphics[trim=0.1cm 0cm 0.5cm 0.8cm,clip,width = .6\textwidth]{Figures/desync_def2}
%\caption{A solution to $\HS_2$ that corresponds to desynchronization.}
%\label{fig:desync}
%\end{figure}}
%This leads to the following {system of $N-1$ equations}:
%\begin{equation*}
%\left\{ \begin{array}{ccc} 
%\tb - \tau_{2} & = & (1+\varepsilon)\tau_{2} - (1+\varepsilon)\tau_{3} \\ 
%\tb - \tau_{3} & = & (1+\varepsilon)\tau_{2} - (1+\varepsilon)\tau_{4} \\
%&\vdots & \\
%\tb - \tau_{N} & = & (1+\varepsilon)\tau_{2} - 0 \ .
%\end{array} \right. 
%\end{equation*}
\IfConf{\begin{equation}
\Gamma = \left[ \begin{array}{ccccccc} 
1 & 0 & 0 & 0 & \ldots & 0 \\ 
0 & (2+\varepsilon) &-(1+\varepsilon) & 0 & \ldots & 0  \\
0 & (1+\varepsilon) & 1 & -(1+\e) & \ddots &  \vdots\\
0 & (1+\varepsilon) & 0  & 1 & \ddots & 0 \\
\vdots & \vdots & \vdots  & 0 & \ddots &-(1+\varepsilon) \\
0& (1+\varepsilon) & 0   & 0 & \ldots & 1 \\
\end{array} \right]
\label{eqn:Amatrix} 
\end{equation}}{\begin{align}
\Gamma &= \left[ \begin{array}{ccccccc} 
1 & 0 & 0 & 0 & \ldots & 0 \\ 
0 & (2+\varepsilon) &-(1+\varepsilon) & 0 & \ldots & 0  \\
0 & (1+\varepsilon) & 1 & -(1+\e) & \ddots &  \vdots\\
0 & (1+\varepsilon) & 0  & 1 & \ddots & 0 \\
\vdots & \vdots & \vdots  & 0 & \ddots &-(1+\varepsilon) \\
0& (1+\varepsilon) & 0   & 0 & \ldots & 1 \\
\end{array} \right] 
\label{eqn:Amatrix} 
\end{align}}
and
$
b = \bar{\tau} {\bf 1},
$
where $\tau_s$ is the state $\wt\tau^{k}$ sorted into decreasing order. 
\NotForConf{For example, if $\wt\tau^k$ is such that $\wt\tau^k_2 = \tb > \wt\tau_1^k > \wt\tau_3^k$, then $\tau_s$ is given as $[\wt\tau_2^k,\wt\tau_1^k,\wt\tau_3^k]^\top$.} 
It can be shown that for any $\varepsilon \in (-1,0)$, a solution $\tau_{s}$ exists (see Lemma~\ref{eqn:tauSsolution}).
Then, $\tau_s$ needs to be unsorted and becomes $\wt\tau^k$ in the definition of the set $\ell_k$.

The solution to $\Gamma\tau_s = b$ is the result of a single case of $\tau \in D \setminus \X$. As indicated above, to get a full definition of the set $\A$, the $N!$ sets $\ell_k$ should be computed. 
For arbitrary $N$, the set $\A$ is given as a collection of sets $\ell_{k}$ given by
\begin{equation}
\A = \bigcup_{k =1}^{N!}\ell_{k} \label{eqn:AN},
\end{equation}
where, for each $k \in \{1,2,\dots,N!\}$, $\ell_k := \{\tau  :\tau =  \wt\tau^k + \one s \in P_N, s \in \reals \}.$ 
\NotForConf{For the case $N = 2$, the points $\wt\tau^k$ for $k \in \{1,2\}$ lead to the set $\A$ given by
\begin{align*}
\begin{split}
\A = \ell_{1}\cup\ell_{2} &= \left\{\tau  : \tau = \left[\begin{array}{c} \tb \\ \frac{\tb}{\varepsilon + 2} \end{array}\right] + \one s \in P_2, s\in \reals \right\} \cup \left\{\tau :\tau =  \left[\begin{array}{c}  \frac{\tb}{\varepsilon + 2} \\ \tb \end{array}\right] + \one s \in P_2, s\in \reals \right\} .
\end{split}%\label{eqn:A2}
\end{align*}

\noindent Figure~\ref{fig:A2} shows these sets in the $(\tau_{1},\tau_{2})$-plane (solid blue and red). Figure~\ref{fig:A2SetandTraj} shows a solution to $\HS_{2}$. The initial conditions for the simulation are $\tau(0,0) = (0.75, \ 0.7)$.
%%%%%%%%%%%% A3 %%%%%%%%%%%%%%

Furthermore, for the case $N = 3$ the points $\wt\tau^k$ for $k \in \{1,2,...,6\}$ lead to the set $\A_3$ given by 
\begin{align*}\allowdisplaybreaks
\A_3 &= \ell_1\cup \ell_2 \cup \ell_3 \cup \ell_4 \cup \ell_5\cup \ell_6 \\
&= \left\{\tau  : \tau = \mat{\tb \\ \frac{(\e+2)\tb}{e^{2}+3\e+3} \\ \frac{\tb}{e^{2}+3\e+3}} + \one s \in P_3, s\in \reals \right\} \cup \left\{\tau : \tau =  \mat{\tb \\ \frac{\tb}{e^{2}+3\e+3} \\ \frac{(\e+2)\tb}{e^{2}+3\e+3} } + \one s \in P_3, s\in \reals \right\} \\ 
& \quad \cup \left\{\tau  : \tau = \mat{ \frac{(\e+2)\tb}{e^{2}+3\e+3} \\ \tb \\ \frac{\tb}{e^{2}+3\e+3}} + \one s \in P_3, s\in \reals \right\}
\cup \left\{\tau : \tau =  \mat{ \frac{\tb}{e^{2}+3\e+3} \\ \tb \\ \frac{(\e+2)\tb}{e^{2}+3\e+3} } + \one s \in P_3, s\in \reals \right\} \\
& \quad \cup \left\{\tau  : \tau = \mat{ \frac{(\e+2)\tb}{e^{2}+3\e+3} \\ \frac{\tb}{e^{2}+3\e+3} \\ \tb} + \one s \in P_3, s\in \reals \right\} \cup \left\{\tau : \tau =  \mat{ \frac{\tb}{e^{2}+3\e+3} \\ \frac{(\e+2)\tb}{e^{2}+3\e+3}\\ \tb} + \one s \in P_3, s\in \reals \right\} \\ 
\end{align*}
Figure~\ref{fig:A3} shows these sets in the $(\tau_1,\tau_2,\tau_3)$-plane (solid colored). Figure~\ref{fig:A3SetandTraj} shows two solutions to $\HS_3$. Note how each simulation has jumps that take the trajectory close to different lines. \seannew{This is due to a preservation of order for each $\tau_i$ as seen in Lemma~\ref{lem:ordering}.}
This preservation of order will be used in the Lyapunov stability proof in the next section.}

\NotForConf{
% Figures in Tech Report Only
\begin{figure}[h!]
\centering
\subfigure[{The desynchronization set $\A$ for $N=3$, with $\tb = 1$ and $\e = -0.3$.}]{  \label{fig:A3}
    \psfragfig*[width=0.4\textwidth]{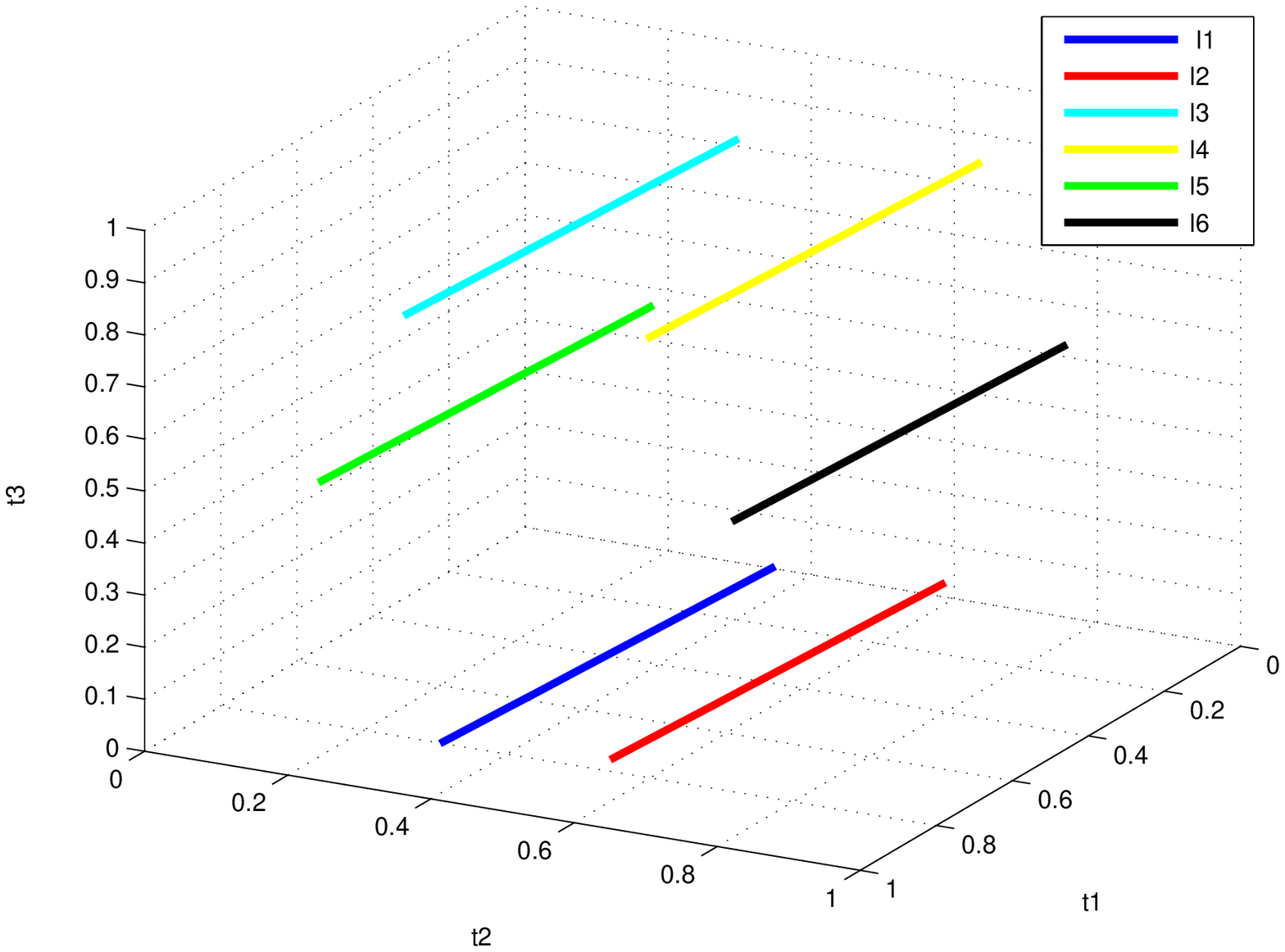}{
       \psfrag{t1}[][][1]{$\ton$}
   \psfrag{t2}[][][1]{$\ttw$}
   \psfrag{t3}[][][1]{$\tau_3$}
   \psfrag{tb}[][][.7]{$\tb$}
   \psfrag{ l1}[][][.4]{$\ell_1$}
   \psfrag{l2}[][][.4]{$\ell_2$}
   \psfrag{l3}[][][.4]{$\ell_3$}
   \psfrag{l4}[][][.4]{$\ell_4$}
   \psfrag{l5}[][][.4]{$\ell_5$}
   \psfrag{l6}[][][.4]{$\ell_6$}
    }}
    \hspace{4mm}
\subfigure[{Two solutions \sean{(magenta and cyan)} to $\HS_3$ such that $\tau(0,0) \notin \X_3$ with $\e = -0.3$ and $\tb = 1$ showing the attractivity to $\A_3 = \cup_{i=1}^{3!} \ell_i$}]{\label{fig:A3SetandTraj}
    \psfragfig*[width=0.4\textwidth,trim = 0mm 0mm 0mm 0mm, clip]{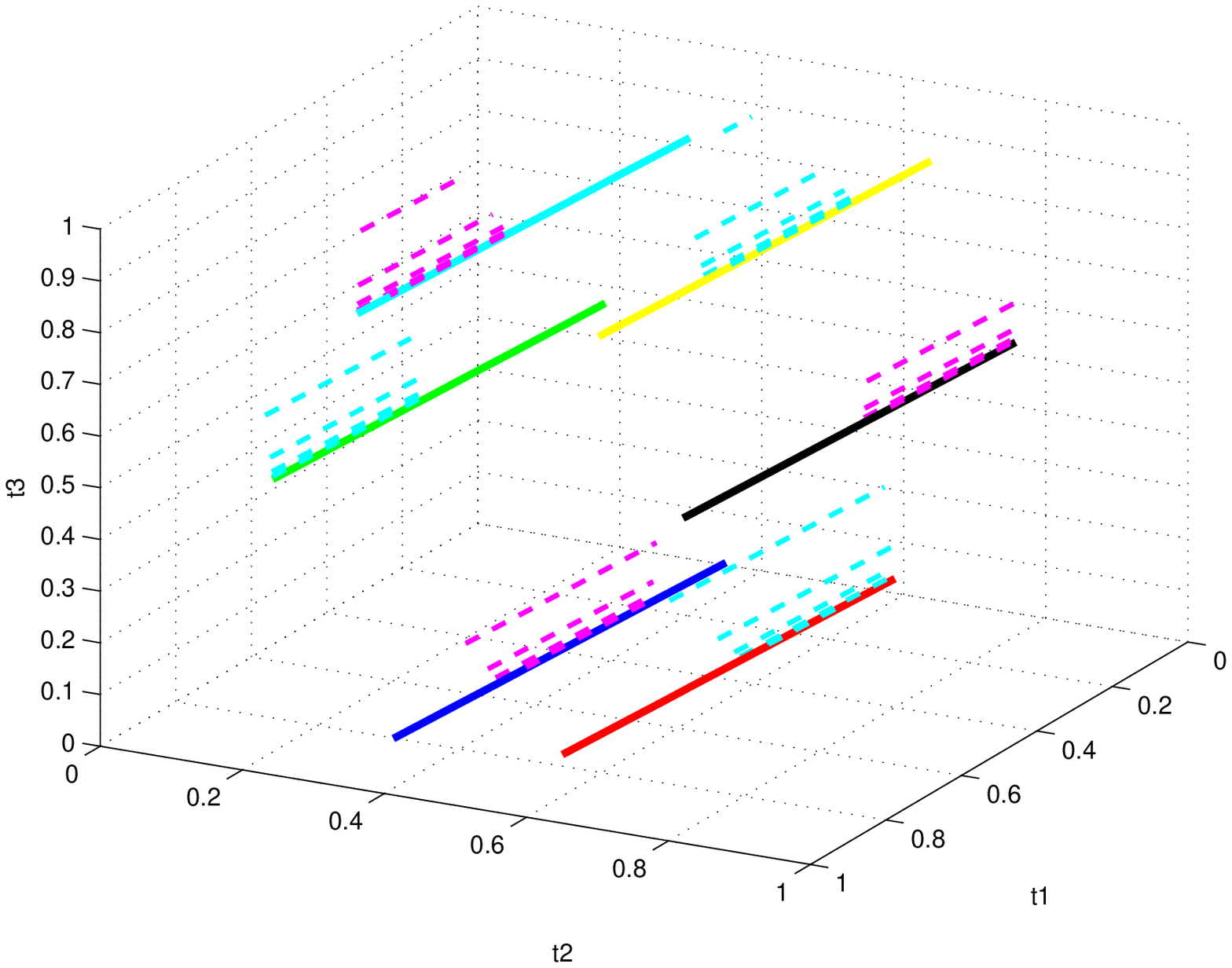}{
       \psfrag{A3}[][][.4]{$\A_3$}
   \psfrag{t1}[][][1]{$\ton$}
   \psfrag{t2}[][][1]{$\ttw$}
   \psfrag{t3}[][][1]{$\tau_3$}
    }}
    \caption{\subref{fig:A3} Set of points $(\ell_1, \ell_2, ..., \ell_6)$ defining $\A_3$ and \subref{fig:A3SetandTraj} two simulations (dashed, one in cyan and the other in magenta) converging to the set $\A_3$ (solid colored). } \label{fig:A3figures}
\end{figure}}

\subsection{Lyapunov Stability} \label{sec:lyapunov}
Lyapunov theory for hybrid systems is employed to show that the
set of points $\A$ is asymptotically stable. 
Our candidate Lyapunov-like function, which is defined below and uses the distance function, is built by observing that there exist points where the distance to $\A$ may increase during flows. This is due to the sets $\ell_{k}$ being a subset $P_{N}$.
To avoid this issue, we define 
$$\wt\A = \bigcup_{k=1}^{N!} \wt\ell_k \supset \A$$
where $\wt\ell_k$ is the extension of $\ell_k$ given by 
\begin{align}
\wt\ell_k = \left\{\tau \in \reals^{N} : \tau = \wt\tau^k + \one s, s \in \reals \right\}. \label{eqn:wtellk}
\end{align}
Then, with this extended version of $\A$, the proposed candidate Lyapunov-like function for asymptotic stability of $\A$ for $\HS_N$ is given by the locally Lipschitz function \startmodif
\begin{equation}\label{eqn:lyapunov}
V(\tau) = \min\{|\tau|_{\wt\ell_1},|\tau|_{\wt\ell_2}, \ldots ,|\tau|_{\wt\ell_k}, \ldots ,|\tau|_{\wt\ell_N!}\} \quad \forall \ \tau \in P_N \setminus \X 
\end{equation} \stopmodif
where, \startmodif for some $k$, $|\tau|_{\wt\ell_k}$ \stopmodif is the distance between the point $\tau$ and the set $\wt\ell_k$.\footnote{The set $\wt\ell_k$ can be described as a straight line in $\reals^n$ passing through a point $\wt\tau^k$ and with slope $\one$. Then, $|\tau|_{\wt\ell_k}$ can be written as the general point-to-line distance $|(\wt\tau^k - \tau) - 1/N ((\wt\tau^k - \tau)^\top\one)\one|$.
} The following theorem establishes asymptotic stability of $\A$ for $\HS_N$.
We show that the change in $V$ during flows is zero and that at jumps we have a strict decrease of $V$; namely, $V(G(\tau)) - V(\tau) = -|\e| V(\tau)$. A key step in the proof is in using \cite[Theorem 8.2]{teel2012hybrid} on a restricted version of $\HS_N$.
%\begin{theorem}\label{thm:stability}
%For every \sean{$N \in \{2, 3, \ldots \}$}, $\tb,\omega > 0$, and $\varepsilon \in (-1,0)$, the hybrid system $\HS_N$ is such that $\A$ is asymptotically stable with basin of attraction given by $P_N \setminus \X$.
%\end{theorem}
\begin{theorem}\label{thm:stability}
For every $N \in \nats, N > 1$, $\tb > 0,\omega > 0$, and $\varepsilon \in (-1,0)$, the hybrid system $\HS_N$ is such that the compact set $\A$ is 
\IfConf{asymptotically stable with basin of attraction given by $\mathcal{B}_{\A} := P_N \setminus \X$. Furthermore, $\A$ is weakly globally asymptotically stable. }{\begin{enumerate}
\item asymptotically stable with basin of attraction given by $\mathcal{B}_{\A} := P_N \setminus \X$.
\item  weakly globally asymptotically stable. 
\end{enumerate}}
\end{theorem}
\begin{proof}
Let the set $\X_{v}$ define the $v$-inflation of $\X$ (defined in Lemma~\ref{lem:defX}), that is, the open set\footnote{The set $\X_{v}$ is open since every point $\tau \in \X_{v}$ is an interior point of $\X_{v}$.} $\X_{v} := \{\tau \in \reals^N : |\tau|_{\X} < v\}$, where $v \in (0,v^{*})$ and $v^{*} = \min_{x\in\X, y\in\wt\A} |x-y|$. Given any $v \in (0,v^{*})$, we now consider a restricted hybrid system $\wt\HS_{N} = (f,\wt{C},G,\wt{D})$, where $\wt{C} := C \setminus \X_{v}$ and $\wt{D} := D \setminus \X_{v}$, which are closed. We establish that $\wt\A$ is an asymptotically stable set for $\wt\HS_N$.

Note that the continuous function $V$, given by \eqref{eqn:lyapunov}, is defined as the minimum distance from $\tau$ to $\wt\A$, where $\wt\A$ is the union of $N!$ sets $\wt\ell_{k}$ in \eqref{eqn:wtellk}. To determine the change of $V$ during flows\footnote{
Its derivative can be computed using Clarke's generalized gradient \cite{Clarke90}.}, we consider the relationship between the flow map and the sets $\wt\ell_{k}$. The inner product between a vector pointing in the direction of the set $\wt\ell_{k}$ and the flow map on $\wt{C}$ satisfies $$\one^{\top}f(\tau) = \one^{\top}(\omega\one) = \omega N=|\one||\omega \one| = |\one||f(\tau)|\cos \theta$$, which is only true if $\theta$ is zero. Therefore, the direction of the flow map and of the vector defining $\wt\ell_{k}$ are parallel, implying that the distance to the set $\wt\A$ is constant during flows.

The change in $V$ during jumps is given by $V(G(\tau)) - V(\tau)$ for $\tau \in \wt D \setminus \wt\A$. Due to the fact that we can rearrange the components of $\tau \in P_N \setminus \X$, without loss of generality, we consider a single jump condition, namely, we consider $\tau$ such that $\tb = \tau_{1} > \tau_{2} > \ldots > \tau_{N-1} > \tau_{N}$. Using the formulation in Section~\ref{sec:AN} and Lemma~\ref{eqn:tauSsolution}, the elements of the vector $\wt\tau^{k}$ associated with $\wt\ell_{k}$ for this case of $\tau$ are given by $\wt\tau_{i}^{k} = \frac{\sum_{p=0}^{N-i}(\e+1)^p}{\sum_{p=0}^{N-1}(\e+1)^p}\tb$, which by Lemma~\ref{lem:consum1} is equal to $\frac{(\e+1)^{N - i +1} - 1}{(\e+1)^{N} - 1}\tb$.
After the jump, $G(\tau)$ is single valued and is such that its elements are ordered as follows: $g_{2}(\tau) > g_{3}(\tau) > \ldots > g_{N}(\tau) > g_{1}(\tau) = 0.$ Specifically, the jump map is $G(\tau) = [0,(1+\e)\tau_{2},\ldots, (1+\e)\tau_{N}]^{\top}$. Then, the 
formulation in Section~\ref{sec:AN} and {Lemma~\ref{eqn:tauSsolution}} leads to a case of $\wt\tau^{k}$ denoted as $\wt\tau^{k'}$.
By {Lemma~\ref{lem:consum1}}, the elements of the vector $\wt\tau^{k'}$ are given by $\wt\tau^{k'}_{1} = \frac{\e}{(\e+1)^{N} - 1}\tb$ and 
$\wt\tau^{k'}_{i} = \frac{(\e+1)^{N - i +2} - 1}{(\e+1)^{N} - 1}\tb$ 
for $i > 1$.
Due to the ordering of $\tau$ and $G(\tau)$, $\wt\tau^{k'}$ is a one-element shifted (to the right) version of $\wt\tau^{k}$. 

From the definition of $\wt\tau^{k}$ above, $V$ at $\tau$ reduces to $$V(\tau) = |\tau|_{\wt\ell_{k}} = \left|(\wt\tau^{k} - \tau) - \frac{1}{N}((\wt\tau^{k} - \tau)^{\top}\one)\one \right|$$ for some $k$. Note that 
$$(\wt\tau^{k} - \tau)^{\top}\one = \sum_{i = 1}^{N} \wt\tau^{k}_{i} - \sum_{i=1}^{N} \tau_{i}$$ reduces to  $\sum_{i = 2}^{N} \wt\tau^{k}_{i} - \sum_{i=2}^{N} \tau_{i}$ 
since $\tau_{1} = \wt\tau_{1}^{k} = \tb$. Using {Lemmas~\ref{lem:consum1} and \ref{lem:consum2}},
it follows that 
$$\sum_{i = 2}^{N} \wt\tau^{k}_{i}= \frac{\sum_{i=2}^N \sum_{p=0}^{N-i}(\e+1)^p}{\sum_{p=0}^{N-1}(\e+1)^p}\tb = \frac{((\e +1)^{N} - 1) - N\e}{\e((\e+1)^{N} - 1)}\tb.$$
Then, the first element of the vector inside 
the norm in the expression of $V(\tau)$ is given as 
\begin{align*}
&(\wt\tau^{k}_{1} - \tau_{1}) - \frac{1}{N}\left(\frac{((\e +1)^{N} - 1 )- N\e}{\e((\e+1)^{N} - 1)}\tb - \sum_{i=2}^{N} \tau_{i}\right) 
\\ & \qquad\qquad\qquad\qquad= -\frac{((\e +1)^{N} - 1) - N\e}{\e N((\e+1)^{N} - 1)}\tb + \frac{1}{N}\sum_{i=2}^{N} \tau_{i},\end{align*} while
the elements with $m \in \{2,3,\ldots,N\}$ 
are given by 
\IfConf{
\begin{align*}
&(\wt\tau_{m}^{k} -\tau_{m}) - \frac{1}{N}\left(\frac{((\e +1)^{N} - 1) - N\e}{\e((\e+1)^{N} - 1)}\tb - \sum_{i=2}^{N} \tau_{i}\right) \\
&= \left(\frac{(\e+1)^{N - m +1} - 1}{(\e+1)^{N} - 1}\tb -\tau_{m}\right) 
\\ &\qquad \qquad \qquad \qquad - \frac{1}{N}\left(\frac{((\e +1)^{N} - 1) - N\e}{\e((\e+1)^{N} - 1)}\tb - \sum_{i=2}^{N} \tau_{i}\right) \\
&= \frac{\e N(\e+1)^{N - m +1} - ((\e +1)^{N} - 1)}{\e N((\e+1)^{N} - 1)}\tb -\frac{N-1}{N}\tau_{m} \\ & \qquad \qquad \qquad \qquad \qquad \qquad \qquad \qquad \qquad + \frac{1}{N}\sum_{i=2,i\neq m}^{N} \tau_{i}.
\end{align*}
}{
\begin{eqnarray*}(\wt\tau_{m}^{k} -\tau_{m}) - \frac{1}{N}\left(\frac{((\e +1)^{N} - 1) - N\e}{\e((\e+1)^{N} - 1)}\tb - \sum_{i=2}^{N} \tau_{i}\right) 
&=& \left(\frac{(\e+1)^{N - m +1} - 1}{(\e+1)^{N} - 1}\tb -\tau_{m}\right)- \\ && \qquad \qquad \qquad \qquad  \frac{1}{N}\left(\frac{((\e +1)^{N} - 1) - N\e}{\e((\e+1)^{N} - 1)}\tb - \sum_{i=2}^{N} \tau_{i}\right) \\
&=& \frac{\e N(\e+1)^{N - m +1} - ((\e +1)^{N} - 1)}{\e N((\e+1)^{N} - 1)}\tb \\ && \qquad \qquad \qquad \qquad \qquad \qquad -\frac{N-1}{N}\tau_{m} + \frac{1}{N}\sum_{i=2,i\neq m}^{N} \tau_{i}.
\end{eqnarray*}
}

After the jump at $\tau$, since $G(\tau)$ is single valued, $V(G(\tau))$ is given by $$|G(\tau)|_{\wt\ell_{k'}} = \left|(\wt\tau^{k'} - G(\tau)) - \frac{1}{N}((\wt\tau^{k'} - G(\tau))^{\top}\one)\one \right|.$$ 
Note that $(\wt\tau^{k'} - G(\tau))^{\top}\one = \sum_{i = 1}^{N} \wt\tau^{k'}_{i} - \sum_{i = 1}^{N} g_{i}(\tau)$ reduces to $\sum_{i = 1}^{N} \wt\tau^{k'}_{i} - \sum_{i = 2}^{N} (1+ \e)\tau_i$, since $g_{1}(\tau) = 0$ and $g_{i}(\tau) = (1+\e)\tau_{i}$  for $i > 1$. Using {Lemmas~\ref{lem:consum1} and \ref{lem:consum2}}, it 
follows that 
\begin{align*}
\sum_{i = 1}^{N} \wt\tau^{k'}_{i} &= \frac{\sum_{i=1}^N \sum_{p=0}^{N-i}(\e+1)^p}{\sum_{p=0}^{N-1}(\e+1)^p}\tb \\ &= \frac{(\e+1)((\e+1)^{N} - 1) - N\e}{\e((\e+1)^{N} - 1)}\tb
\end{align*}
which leads to $$(\wt\tau^{k'} - G(\tau))^{\top}\one = \frac{(\e+1)((\e+1)^{N} - 1) - N\e}{\e((\e+1)^{N} - 1)}\tb  - \sum_{i = 2}^{N} (1+\e)\tau_{i}.$$ The first element inside 
the norm in $V(G(\tau))$ is given by 
\begin{align*}
&(\wt\tau_{1}^{k'} - g_1(\tau)) - \frac{1}{N}\left(\frac{(\e+1)((\e+1)^{N} - 1) - N\e}{\e((\e+1)^{N} - 1)}\tb \right. \\ & \qquad \qquad \qquad \qquad \qquad \qquad \qquad \qquad \qquad \left. - \sum_{i = 2}^{N} (1+\e)\tau_{i}\right) \\
& = \frac{\e}{(\e+1)^{N} - 1}\tb - \frac{(\e+1)((\e+1)^{N} - 1) - N\e}{\e N((\e+1)^{N} - 1)}\tb  \\ 
& \qquad \qquad \qquad \qquad \qquad \qquad \qquad \qquad \qquad + \frac{1}{N}\sum_{i = 2}^{N} (1+\e)\tau_{i} \\
 &= (1+\e)\left(-\frac{((\e+1)^{N} - 1) - N\e}{\e N((\e+1)^{N} - 1)}\tb + \frac{1}{N}\sum_{i = 2}^{N} \tau_{i}\right).
\end{align*} For each element $m > 1$,
it follows that 

\IfConf{\begin{align*}
&(\wt\tau^{k'}_{m} - g_{m}(\tau)) - \frac{1}{N}\left(\frac{(\e+1)((\e+1)^{N} - 1) - N\e}{\e((\e+1)^{N} - 1)}\tb \right. \\ &\left. \qquad\qquad\qquad\qquad\qquad\qquad\qquad\qquad\qquad - \sum_{i = 2}^{N} (1+\e)\tau_{i}\right)  \\ &  =
\frac{(\e+1)^{N - m +2} - 1}{(\e+1)^{N} - 1}\tb - (1+\e)\frac{N-1}{N}\tau_{m} \\ & \qquad  - \frac{(\e+1)((\e+1)^{N} - 1) - N\e}{\e N((\e+1)^{N} - 1)}\tb  + \frac{1}{N}\sum_{i = 2, i \neq m}^{N} (1+\e)\tau_{i} \\ &  
% \frac{\e N(\e+1)^{N - m +2} - (\e+1)((\e+1)^{N} - 1)}{\e N((\e+1)^{N} - 1)}\tb - (1+\e)\frac{N-1}{N}\tau_{m} + \frac{1}{N}\sum_{i = 2, i \neq m}^{N} (1+\e)\tau_{i} 
= (1+\e) \left(\frac{\e N(\e+1)^{N - m +1} - ((\e+1)^{N} - 1)}{\e N((\e+1)^{N} - 1)}\tb \right. \\ 
& \left.  \qquad\qquad\qquad\qquad\qquad\qquad - \frac{N-1}{N}\tau_{m} + \frac{1}{N}\sum_{i = 2, i \neq m}^{N} \tau_{i} \right).\end{align*} }
{\begin{eqnarray*}
(\wt\tau^{k'}_{m} - g_{m}(\tau)) &-& \frac{1}{N}\left(\frac{(\e+1)((\e+1)^{N} - 1) - N\e}{\e((\e+1)^{N} - 1)}\tb -  \sum_{i = 2}^{N} (1+\e)\tau_{i}\right)  \\ 
&=& \frac{(\e+1)^{N - m +2} - 1}{(\e+1)^{N} - 1}\tb - (1+\e)\frac{N-1}{N}\tau_{m}  - \frac{(\e+1)((\e+1)^{N} - 1) - N\e}{\e N((\e+1)^{N} - 1)}\tb  + \frac{1}{N}\sum_{i = 2, i \neq m}^{N} (1+\e)\tau_{i} \\   
% \frac{\e N(\e+1)^{N - m +2} - (\e+1)((\e+1)^{N} - 1)}{\e N((\e+1)^{N} - 1)}\tb - (1+\e)\frac{N-1}{N}\tau_{m} + \frac{1}{N}\sum_{i = 2, i \neq m}^{N} (1+\e)\tau_{i} 
&=& (1+\e) \left(\frac{\e N(\e+1)^{N - m +1} - ((\e+1)^{N} - 1)}{\e N((\e+1)^{N} - 1)}\tb - \frac{N-1}{N}\tau_{m} + \frac{1}{N}\sum_{i = 2, i \neq m}^{N} \tau_{i} \right).\end{eqnarray*} }

Combining the expressions for each of the elements inside the norm of $V(G(\tau))$, it follows that $V(G(\tau)) = (1+\e) V(\tau)$. 
 
Then, the change during jumps is given by 
$V(G(\tau)) - V(\tau) = \e V(\tau)$ where 
$\e \in (-1,0)$. With the property of $V$ during 
flows established above, the change of
$V$ along solutions 
is bounded during flows and jumps by the nonpositive functions 
$u_{\wt{C}}$ and $u_{\wt{D}}$, respectively, defined as follows:  
$u_{\wt{C}} (z)= 0$ for each 
$z \in \wt{C}$ and $u_{\wt{C}} (z)= -\infty$ 
otherwise; $u_{\wt{D}}(z) = \e V(z)$ 
for each $z \in \wt{D}$ and $u_{\wt{D}}(z) = -\infty$ otherwise. Using Lemma~\ref{lem:BasicConds}, the fact that $\wt{C}$ and $\wt{D}$ are closed, and the fact that every maximal solution to $\wt\HS$ is bounded 
and complete, by 
\cite[Theorem 8.2]{teel2012hybrid}, every maximal solution to $\wt\HS_N$ approaches the 
largest weakly invariant subset of 
$L_{V}(r') \cap \wt{C} \cap [L_{u_{\wt{C}}}(0) \cup (L_{u_{\wt{D}}}(0) \cap G(L_{u_{\wt{C}}}(0)))] = L_{V}(r') \cap \wt{C}$ 
for $r' \in V(\wt{C})$. Since every maximal solution jumps an infinite number of times, the 
largest invariant set is given for $r' = 0$ due to the fact that $V(G(\tau)) - V(\tau) = \e V(\tau) < 0$ if $r' > 0$. 
\startmodif Then, the largest invariant set is given by $L_{V}(0) \cap \wt{C} = \wt\A \cap \wt{C}$ which is identically equal to $\A$\stopmodif. Hence, 
the set $\A$ is attractive. 
\startmodif Stability is guaranteed from the fact that $V$ is nonincreasing during flows and strictly decreasing during jumps. 
Then, the set $\wt\A$ is asymptotically stable for the hybrid system $\wt\HS_{N}$. 
We have that $\A$ is (strongly) forward invariant and from Theorem 3.4 we know that $\A$ is uniformly attractive from a neighborhood of itself. Then by Proposition 7.5 in \cite{teel2012hybrid}, it follows that $\A$ is asymptotically stable. 
\stopmodif

Note that the set of solutions to $\wt\HS_{N}$ coincides with the set of solutions to $\HS_{N}$ from $P_{N} \setminus \X_{v}$. Therefore, the set $\A$ 
is asymptotically stable for $\HS_N$ with basin of attraction ${\cal B}_{\A} = P_{N} \setminus \X_{v}$. \seannew{Since $v$ is arbitrary, it follows that the basin of attraction is equal to \startmodif $P_{N} \stopmodif \setminus \X$.}

Note that the jump map $G$, at points $\tau \in \X$, is set valued by definition of $g_i$ in \eqref{eqn:gi}. From these points there exist solutions to $\HS_N$ that jump out of $\X$. 
In fact, consider the case $\tau \in \X$. We have that $\tau_{i} = \tau_r$ for some $i,r \in I$. Then, after the jump it follows that $g_{i}(\tau) \in \{0, (1+\e)\tb\}$ and $g_{r}(\tau) \in \{0, (1+\e)\tb\},$ and there exist $g_{i}$ and $g_{r}$ such that $g_{i} = g_{r}$ or $g_{i} \neq g_{r}$. 
Since for every point in $\X$ there exists a solution that converges to $\A$ and also a solution that stays in $\X$, $\X$ is weakly forward invariant.\footnote{For example, consider the case $N = 2$. 
If $\tau(0,0) = [\tb,\tb]^\top \in D$, then there are nonunique solutions due to the jump map 
begin set valued. It follows that after the jump, each $\tau_i$ can be mapped to any point in 
$\{0,\tau_i(1+\e)\}$, which leads to any of the following four options of the states $(\tau_1,\tau_2)$ 
after such a jump: $(0,0),(0,\tb(1+\e)),(\tb(1+\e),0)$ or $(\tb(1+\e),\tb(1+\e))$. If 
the state is mapped 
to either $(0,0)$ 
or $(\tb(1+\e),\tb(1+\e))$, then it remains in $\X_2$. Conversely, if any of the other options are chosen, then 
$(\ton,\ttw)$ leaves $\X_2$ and converges 
to $\A$ asymptotically.} 

\end{proof}
\vspace{-.5cm}
\subsection{Characterization of Time of Convergence}\label{sec:timetoconverge}
In this section, we characterize the time to converge to a neighborhood of $\A$. The proposed (upper bound) of the time to converge depends on the initial distance to the set $\wt\A$ and the parameters of the hybrid system $(\varepsilon,\tb)$. 
%The following result characterizes this time.
\begin{theorem}\label{thm:timetoconverge}
For every $N \in \nats$, $N > 1$, and every $c_1, c_2$ such that $\overline{c} > c_2 > c_1 > 0$ with $\overline{c} = \max_{x \in \X} |x|_{\widetilde\A}$, every maximal solution 
to $\HS_N$ with initial condition $\tau(0,0) \in (P_{N} \setminus \X) \cap \wt L_{V}(c_2)$ is 
such that
\IfConf{$
\tau(t,j) \in \wt L_{V}(c_1)$ for each $(t,j) \in \dom \tau, t+j \geq  M,$
where
$M = \left(\frac{\tb}{\omega}+1\right)\frac{\log \frac{c_{2}}{c_{1}}}{\log \frac{1}{1+\varepsilon}}$
and $\wt L_{V}(\mu) := \{\tau \in C\cup D : V(\tau) \leq \mu \}$. }{
$$
\tau(t,j) \in \wt L_{V}(c_1) \ \ \forall (t,j) \in \dom \tau, t+j \geq  M, $$
where
$$M = \left(\frac{\tb}{\omega}+1\right)\frac{\log \frac{c_{2}}{c_{1}}}{ \log \frac{1}{1+\varepsilon}}$$
and $\wt L_{V}(\mu) := \{\tau \in C\cup D : V(\tau) \leq \mu \}$. }
\end{theorem}
\IfConf{\begin{figure}
% In Journal Version
\centering
\psfrag{09}[][][.35]{-$0.9$}
\psfrag{08}[][][.35]{-$0.8$}
\psfrag{07}[][][.35]{-$0.7$}
\psfrag{06}[][][.35]{-$0.6$}
\psfrag{05}[][][.35]{-$0.5$}
\psfrag{04}[][][.35]{-$0.4$}
\psfrag{03}[][][.35]{-$0.3$}
\psfrag{02}[][][.35]{-$0.2$}
\psfrag{01}[][][.35]{-$0.1$}
\psfrag{Nstar}[][][1][-90]{$\frac{T^{*}}{\tb+1}$}
\psfrag{eps}[][][1]{$\varepsilon$}
\psfrag{data1}[][][.55]{\hspace{8mm}$c_{1} = 0.5$}
\psfrag{data2}[][][.55]{\hspace{8mm}$c_{1} = 0.3$}
\psfrag{data3}[][][.55]{\hspace{8mm}$c_{1} = 0.1$}
\psfrag{data4}[][][.55]{\hspace{10mm}$c_{1} = 0.05$}
\includegraphics[width=.49\textwidth,trim = 6mm 4mm 14mm 4mm, clip]{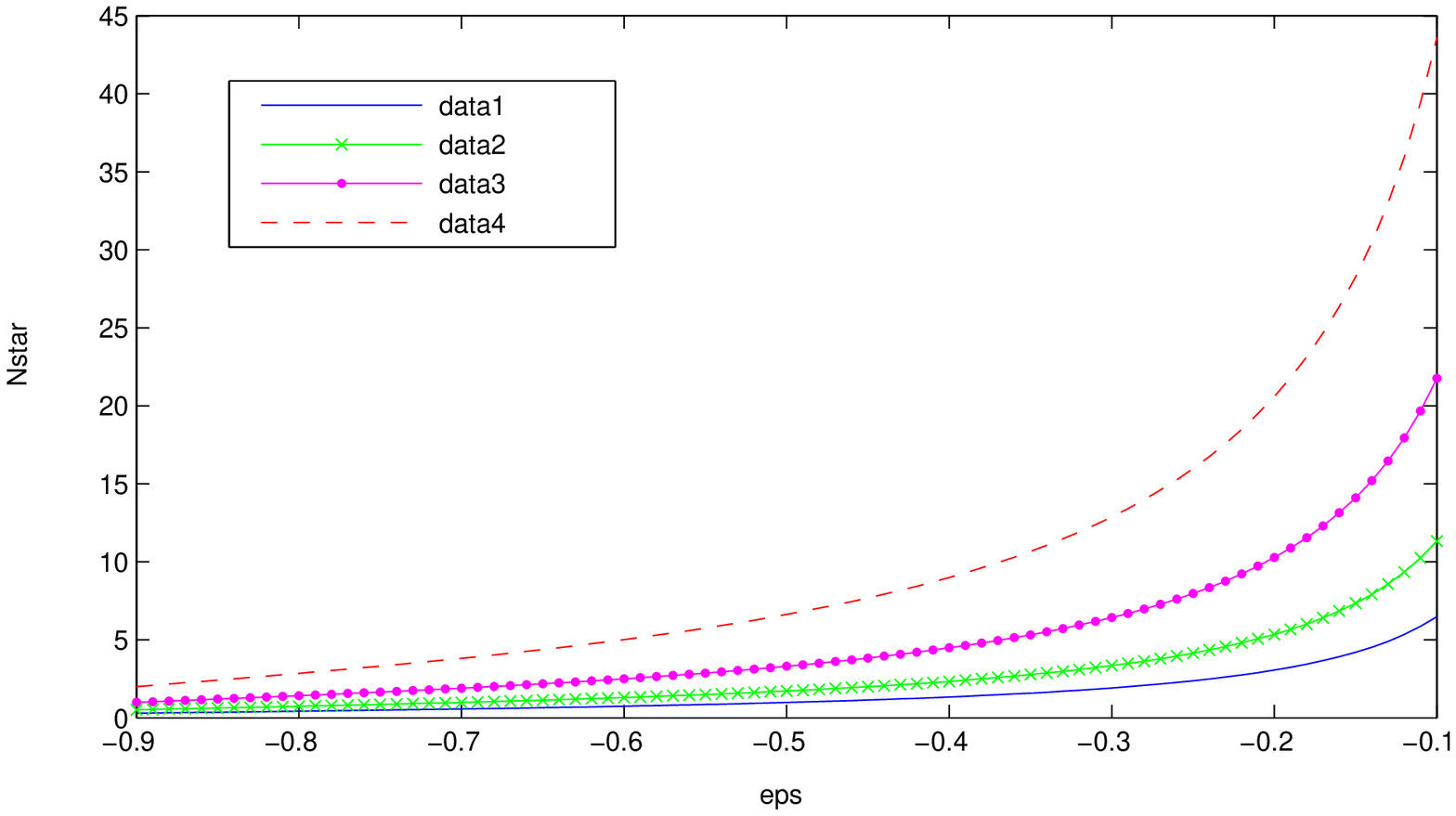}
\caption{Time to converge (over $\bar{\tau}+1$) as a function of $\varepsilon \in [-0.9,-0.1]$, with $c_{2} = 0.99\tb$ and $c_{1} \in \{0.5\tb, 0.3\tb, 0.1\tb, 0.05\tb\}$}
\label{fig:timetoconverge}
\end{figure}}{
\begin{figure}
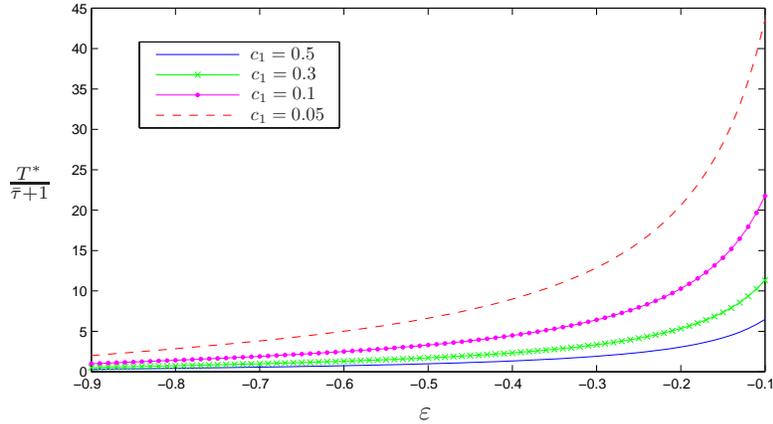

% In Tech Report
\centering
\psfragfig*[width=.7\textwidth,]{Figures/NstarE}{
\psfrag{09}[][][.5]{-$0.9$}
\psfrag{08}[][][.5]{-$0.8$}
\psfrag{07}[][][.5]{-$0.7$}
\psfrag{06}[][][.5]{-$0.6$}
\psfrag{05}[][][.5]{-$0.5$}
\psfrag{04}[][][.5]{-$0.4$}
\psfrag{03}[][][.5]{-$0.3$}
\psfrag{02}[][][.5]{-$0.2$}
\psfrag{01}[][][.5]{-$0.1$}
\psfrag{Nstar}[][][1][-90]{$\frac{T^{*}}{\tb+1}$}
\psfrag{eps}[][][1]{$\varepsilon$}
\psfrag{data1}[][][.7]{\hspace{7mm}$c_{1} = 0.5$}
\psfrag{data2}[][][.7]{\hspace{7mm}$c_{1} = 0.3$}
\psfrag{data3}[][][.7]{\hspace{7mm}$c_{1} = 0.1$}
\psfrag{data4}[][][.7]{\hspace{8.5mm}$c_{1} = 0.05$}
}
\caption{Time to converge (over $\bar{\tau}+1$) as a function of $\varepsilon \in [-0.9,-0.1]$, with $c_{2} = 0.99\tb$ and $c_{1} \in \{0.5\tb, 0.3\tb, 0.1\tb, 0.05\tb\}$}
\label{fig:timetoconverge}
\vspace{-.3cm}
\end{figure}}
%}
%{\begin{figure}[h!]
%\centering
%\psfrag{09}[][][.5]{-$0.9$}
%\psfrag{08}[][][.5]{-$0.8$}
%\psfrag{07}[][][.5]{-$0.7$}
%\psfrag{06}[][][.5]{-$0.6$}
%\psfrag{05}[][][.5]{-$0.5$}
%\psfrag{04}[][][.5]{-$0.4$}
%\psfrag{03}[][][.5]{-$0.3$}
%\psfrag{02}[][][.5]{-$0.2$}
%\psfrag{01}[][][.5]{-$0.1$}
%\psfrag{Nstar}[][][1]{$\frac{T^{*}}{\tb+1}$}
%\psfrag{eps}[][][1]{$\varepsilon$}
%\psfrag{c15}[][][.7]{\hspace{-3mm}$c_{1} = .5$}
%\psfrag{c13}[][][.7]{\hspace{-3mm}$c_{1} = .3$}
%\psfrag{c11}[][][.7]{\hspace{-3mm}$c_{1} = .1$}
%\psfrag{c105}[][][.7]{\hspace{-3mm}$c_{1} = .05$}
%\includegraphics[width=.8\textwidth,]{Figures/NstarE}
%\caption{Time to converge (over $\bar{\tau}+1$) as a function of $\varepsilon \in [-0.9,-0.1]$, with $c_{2} = 0.99\tb$ and $c_{1} \in \{0.5\tb, 0.3\tb, 0.1\tb, 0.05\tb\}$}
%\label{fig:timetoconverge}
%\end{figure}}

\begin{proof}
Let $\tau_0 = \tau(0,0)$ and pick a maximal solution $\tau$ to $\HS_{N}$ from $\tau_{0}$. At every 
jump time $(t_{j},j) \in \dom \tau$, define $\bar{g}_1 = \tau(t_1,1)$, $\bar{g}_2 = \tau(t_2,2), \ldots, \bar{g}_J = \tau(t_J,J)$, for some $J \in \nats$. From Theorem~\ref{thm:stability}, we have that there is no change in the Lyapunov 
function during flows. Furthermore, we have that for each $\tau \in D \setminus \A$ the difference $V(G(\tau)) - V(\tau) = \varepsilon V(\tau)$ with $\varepsilon \in (-1,0)$. Since, for every $j,$ $\tau(t_{j},j) \in D$, we have
\IfConf{$V(\bar g_1) - V(\tau_0) = \varepsilon V(\tau_0),$}{
$$V(\bar g_1) - V(\tau_0) = \varepsilon V(\tau_0),$$}
which implies
\IfConf{$V(\bar g_1) =(1+ \varepsilon) V(\tau_0) .$}{
$$V(\bar g_1) =(1+ \varepsilon) V(\tau_0) .$$}
At the next jump, we have
\IfConf{$V(\bar g_2) =(1+ \varepsilon) V(\bar g_1) = (1+ \varepsilon)^2 V(\tau_0).$}{
$$V(\bar g_2) =(1+ \varepsilon) V(\bar g_1) = (1+ \varepsilon)^2 V(\tau_0).$$}
Proceeding in this way, after $J$ jumps we have
\IfConf{$V(\bar g_J) =(1+ \varepsilon) V(g_{J-1}) = (1+ \varepsilon)^J V(\tau_0).$}{
$$V(\bar g_J) =(1+ \varepsilon) V(g_{J-1}) = (1+ \varepsilon)^J V(\tau_0).$$}
From $V(\bar g_{J}) = (1+\varepsilon)^{J}V(\tau_{0})$, we want to find $J$ so that $V(\bar g_{J}) \leq c_{1}$ when $V(\tau_{0}) \leq c_{2}$. Considering the worst cast for $V(\tau_{0})$, we want $(1+\varepsilon)^{J}c_{2} \leq c_{1}$, which implies $\frac{c_{2}}{c_{1}} \leq \left(\frac{1}{1+\varepsilon} \right)^{J}$, and therefore
$J = \left\lceil\frac{\log \frac{c_{2}}{c_{1}}}{ \log \frac{1}{1+\varepsilon}}\right\rceil > 0.$
For each $j$, the time between jumps satisfies
$t_1 - t_0 \leq \frac{\tb}{\omega}, 
t_2 - t_1 \leq \frac{\tb}{\omega}, 
\ldots, 
t_j - t_{j-1} \leq \frac{\tb}{\omega}.$
Then, we have that after $J$ jumps,
$
\sum_{j=1}^{J} t_{j} - t_{j-1} \leq J\frac{\tb}{\omega}.
$
With $t_{0} = 0$, the expression reduces to 
$t_{J} \leq J\frac{\tb}{\omega}  = \left\lceil\frac{\log \frac{c_{2}}{c_{1}}}{ \log \frac{1}{1+\varepsilon}}\right\rceil\frac{\tb}{\omega}.$
Then, after $t+j \geq t_{J} +J$, the solution is at least $c_{1}$ close to the set $\wt\A$. Defining $M = t_{J} + J$ we then have 
\IfConf{$M =  \left(\frac{\tb}{\omega}+1\right)\frac{\log \frac{c_{2}}{c_{1}}}{ \log \frac{1}{1+\varepsilon}}.$}{
$$M =  \left(\frac{\tb}{\omega}+1\right)\left\lceil\frac{\log \frac{c_{2}}{c_{1}}}{ \log \frac{1}{1+\varepsilon}}\right\rceil.$$} 
\end{proof}

Figure~\ref{fig:timetoconverge} shows the time to converge (divided by $\frac{\tb}{\omega}+1$) versus $\varepsilon$ with constant $c_{2} = 0.99\tb$ and varying values of $c_{1}$. As the figure indicates, the time to converge decreases as $|\e|$ increases, which confirms the intuition that the larger the jump the faster oscillators desynchronize.

\subsection{Robustness Analysis}\label{sec:robustness}
\sean{Lemma~\ref{lem:BasicConds} establishes that the hybrid model of $N$ impulse-coupled oscillators satisfies 
the hybrid basic conditions. In light of this property, the asymptotic stability property of $\A$ for $\HS_N$ is preserved under certain perturbations; i.e., asymptotic stability is robust \cite{teel2012hybrid}. In the next sections, we consider a perturbed version of $\HS_N$ and present robust stability results. In particular, we consider generic perturbations to $\HS_N$, and two different cases of   perturbations only on the timer rates to allow for heterogeneous timers.}

\subsubsection{Robustness to Generic Perturbations}
We start by revisiting the definition of perturbed hybrid systems in \cite{teel2012hybrid}.
\NotForConf{\begin{definition}[perturbed hybrid system {\cite[Definition 6.27]{teel2012hybrid}}]
Given a hybrid system $\HS$ and a function $\rho: \reals^N \to \reals_{\geq 0}$, the $\rho$-perturbation of $\HS$, denoted $\HS_\rho$, is the hybrid system 
$$
\left\{\begin{array}{cc}x \in C_\rho & \quad \dot{x} \in F_\rho(x) \\
x \in D_\rho & \quad x^+ \in G_\rho(x) \\
 \end{array} \right.
$$
where
\IfConf{$C_\rho = \{x \in \reals^n : (x + \rho(x)\ball) \cap C \neq \emptyset \}$,
$F_\rho(x) = \overline{\mathrm{con}}F((x + \rho(x)\ball)\cap C) + \rho(x) \ball$ for all $x \in \reals^n$,
$D_\rho = \{x \in \reals^n : (x + \rho(x) \ball) \cap D \neq \emptyset \}$,
and $G_\rho(x) = \{v \in \reals^n : v \in g + \rho(g)\ball, g \in G((x + \rho(x)\ball) \cap D) \}$ for all $x \in \reals^n$.
}{\begin{align*}
C_\rho &= \{x \in \reals^n : (x + \rho(x)\ball) \cap C \neq \emptyset \}, \\
F_\rho(x) &= \overline{\mathrm{con}}F((x + \rho(x)\ball)\cap C) + \rho(x) \ball \quad \forall x \in \reals^n, \\
D_\rho &= \{x \in \reals^n : (x + \rho(x) \ball) \cap D \neq \emptyset \}, \\
G_\rho(x) &= \{v \in \reals^n : v \in g + \rho(g)\ball, g \in G((x + \rho(x)\ball) \cap D) \} \quad \forall x \in \reals^n.
\end{align*}}
\end{definition}}
Using this definition, we can deduce a generic perturbed hybrid system modeling $N$ impulse-coupled oscillators. Then, for the hybrid system $\HS_N$, we denote $\HS_{N,\rho}$ as the $\rho$-perturbation of $\HS_N$. Given the perturbation function $\rho : \reals^N \to \reals_{\geq 0}$, the perturbed flow map is given by 
\IfConf{$F_\rho(\tau) = \omega\one + \rho(\tau)\ball$ for all $\tau \in C_\rho,$}{\begin{align*}
F_\rho(\tau) &= \omega\one + \rho(\tau)\ball \qquad \forall \ \tau \in C_\rho, 
\end{align*}}
where the perturbed flow set $C_\rho$ is given by 
\IfConf{$
C_\rho = \{\tau \in \reals^N : (\tau + \rho(\tau)\ball) \cap P_N \neq \emptyset\}. 
$}{\begin{align*}
C_\rho &= \{\tau \in \reals^N : (\tau + \rho(\tau)\ball) \cap P_N \neq \emptyset\}. 
\end{align*}}
For example, if $N = 2$ and $\rho(\tau) = \bar\rho > 0$ for all $\tau \in \reals^N$, which would correspond to constant perturbations on the lower value and threshold, then $C_\rho = C + \rho \ball$.
The perturbed jump map and jump set are defined as
\IfConf{$D_\rho = \{\tau \in \reals^N : (\tau + \rho(\tau)\ball )\cap D \neq \emptyset\}$,}{\begin{align*}
D_\rho &= \{\tau \in \reals^N : (\tau + \rho(\tau)\ball )\cap D \neq \emptyset\}, 
\end{align*}}
\IfConf{$G_\rho 
= [g_{1,\rho}(\tau), \ldots, g_{N,\rho}(\tau)]^\top$,}{\begin{align*}
G_\rho 
&= [g_{1,\rho}(\tau), \ldots, g_{N,\rho}(\tau)]^\top, 
\end{align*}}
where $g_{i,\rho}$ is the $i$-th component of $G_\rho$.
The following result establishes that the hybrid system $\HS_N$ is robust to small perturbations.
\begin{theorem}(robustness of asymptotic stability)\label{thm:robustofAS}
If $\rho : \reals^N \to \reals_{\geq 0}$ is continuous and positive on $\reals^{N} \setminus \A$, then $\A$ is semiglobally practically robustly $\cal{KL}$ asymptotically stable with basin of attraction $B_{\A} = P_N \setminus \X$, i.e., for every compact set $K \subset B_{\A}$ and every $\alpha > 0$, there exists $\delta \in (0,1)$ such that every maximal solution $\tau$ to $\HS_{N,\delta\rho}$ from $K$ satisfies 
$|\tau(t,j)|_{\A} \leq \beta(|\tau(0,0)|_{\A},t+j) + \alpha$ for all $(t,j) \in \dom \tau$.
\end{theorem}
\begin{proof}
From Lemma~\ref{lem:BasicConds}, the hybrid system $\HS_N$ satisfies the hybrid basic conditions. Therefore, by \cite[Theorem 6.8]{teel2012hybrid} $\HS_N$ is nominally well-posed and, moreover, by \cite[Proposition 6.28]{teel2012hybrid} is well-posed. From the proof of Theorem~\ref{thm:stability}, we know that the set $\A$ is an asymptotically stable compact set for the hybrid system $\HS_N$ with basin of attraction $B_{\A}$. Since by Lemma~\ref{lem:solutions2HS}, every maximal solution is complete, then \cite[Theorem 7.20]{teel2012hybrid} implies that $\A$ is semiglobally practically robustly $\cal{KL}$ asymptotically stable. 
\end{proof}
Section~\ref{sec:jumpperturbs1} showcases \IfConf{an example simulation }{several simulations} of $\HS_{N}$ with $\rho$-perturbations on the jump map.
\subsubsection{Robustness to Heterogeneous Timer Rates}
We consider the case when the continuous dynamic rates are perturbed in the form of
\IfConf{$
\frac{d}{dt}|\tau(t,j)|_{\wt\A} = c(t,j)
$}{$$
\frac{d}{dt}|\tau(t,j)|_{\wt\A} = c(t,j)
$$}
for a given solution $\tau$. 
For example, consider the perturbation of the flow map given by 
\begin{eqnarray}
f(\tau) = \omega \one+ \Delta\omega \label{eqn:fdelta}
\end{eqnarray}
where $\Delta\omega \in \reals^n$ is a constant defining a perturbation 
from the natural frequencies of the impulse-coupled oscillators. 
Then \startmodif for some $k$, \stopmodif during flows, along a solution $\tau$ such that over $[t_{j},t_{j+1}]\times\{j\}$ satisfies \startmodif $V(\tau(t,j)) = |\tau(t,j)|_{\wt\ell_{k}}$, \stopmodif
it follows that $c$ reduces to \startmodif
$c(t,j) = \left(\frac{r_{\ell_{k}}^{\top}(\tau(t,j))(\frac{1}{N}{\bf \underline{1} - \bf I})}{|\tau(t,j)|_{\ell_{k}}}\right) \Delta\omega.$\stopmodif\footnote{ \startmodif
Let $r_{\ell_{k}}(\tau)$ be the vector defined by the minimum distance from $\tau$ to the line $\ell_{k}$.
Then, it follows that $V(\tau) = (r_{\ell_{k}}^{\top}(\tau)r_{\ell_{k}}(\tau))^{\frac{1}{2}}$. 
To determine its change during flows, note that on $C\setminus (\X \cup \A)$ the gradient 
is given by 
$
\nabla V(\tau) = \frac{\partial}{\partial \tau} \left( r_{\ell_{k}}^{\top}(\tau)r_{\ell_{k}}(\tau) \right)^{\frac{1}{2}}
= \frac{\left(r_{\ell_{k}}^{\top}(\tau) \frac{\partial}{\partial \tau} r_{\ell_{k}}(\tau) \right)}{|\tau|_{\ell_{k}}} 
$
where
each $j$-th entry of $\frac{\partial}{\partial \tau} r_{\ell_{k}}(\tau)$ is given by
$
\frac{\partial}{\partial \tau} r^{j}_{\ell_{k}}(\tau) = \frac{\partial}{\partial \tau}\left((\wt\tau_j\null\hspace{-.08cm}^{k} - \tau_{j}) - \frac{1}{N}\sum_{i=1}^N(\wt\tau_i\null\hspace{-.08cm}^k - \tau_i)^\top \right)
= \left[\frac{1}{N}, \frac{1}{N},\ldots, \frac{1}{N}, -1 + \frac{1}{N}, \frac{1}{N}, \ldots, \frac{1}{N} \right]
$
-- the term $-1 + \frac{1}{N}$ corresponds to the $j$-th element of the vector. 
It follows that
$\frac{\partial}{\partial \tau} r_{\ell_{k}}(\tau) = \frac{1}{N}{\bf \underline{1}} - {\bf I}$. Then,
for each $\tau \in C \setminus \X$, 
$
\langle \nabla V(\tau), f(\tau) \rangle = \left(\frac{r_{\ell_{k}}^{\top}(\tau)(\frac{1}{N}{\bf \underline{1} - \bf I})}{|\tau|_{\ell_{k}}}\right)f(\tau)
$.\stopmodif
}
Furthermore, the norm of the hybrid arc $c$ can be bounded by a constant $\bar{c}$ given by 
\begin{eqnarray}
\bar{c} = \left|\left(\frac{1}{N}\underline{\bf 1} - {\bf I}\right)\Delta\omega\right| \label{eqn:ctbound}.
\end{eqnarray}
Building from this example, the following result provides properties of the distance to $\wt\A$ from solutions $\tau$ 
to $\HS_N$ under generic perturbations on $f$ (not necessarily as in \eqref{eqn:fdelta}).

\begin{theorem}\label{thm:vanishingC}
Suppose that the perturbation on the flow map of $\HS_N$ is such that a perturbed solution $\tau$ satisfies, for each $j$ such that $\{t : (t,j) \in \dom \tau\}$ has more than one point, 
$\frac{d}{dt} |\tau(t,j)|_{\wt\A} = c(t,j)$ for all $t \in \{t : (t,j) \in \dom \tau\}$ \seannew{and $\tau(t,j) \in P_N \setminus \X$ for all $(t,j) \in \dom \tau$}, for some hybrid arc $c$ with $\dom c = \dom \tau$.
Then, the following hold:
\begin{itemize}
\item The asymptotic value of $|\tau(t,j)|_{\wt\A}$ satisfies 
\begin{align}
\lim_{t+j \to \infty}|\tau(t,j)|_{\wt\A} \leq \lim_{t+j \to \infty} \sum_{i = 0}^{j} (1+\e)^{j-i} \int_{t_i}^{t_{i+1}} c(t,j) dt
\end{align}
\item If there exists $\bar{c} > 0$ such that $|c(t,j)| \leq \bar{c}$ for each $(t,j) \in \dom \tau$ then 
\begin{align}
\lim_{t+j \to \infty}|\tau(t,j)|_{\wt\A} \leq \frac{\bar{c}\tb}{|\e|\omega}. \label{thmeqn:distupperbound}
\end{align}
\item If $\wt{j} : \reals_{\geq 0} \to \nats$ is a function that \seannew{chooses the appropriate minimum $j$ such that $(t,j) \in \dom \tau$} for each time $t$ and $t \mapsto c(t,\wt{j}(t))$ is absolutely integrable, i.e., $\exists B$ such that 
\IfConf{$\int_0^\infty |c(t,\wt{j}(t))| dt \leq B$,}{\begin{align} 
\int_0^\infty |c(t,\wt{j}(t))| dt \leq B, \label{eqn:cint}
\end{align}}
then
\begin{align}
\lim_{t + j \to \infty} |\tau(t,j)|_{\wt\A} \leq \frac{B}{\e}. \label{eqn:limboundbeps}
\end{align}
%\item If the function $c$ is such that $\int_0^\infty c(s,j) ds = 0$ then $\lim_{t+j \to \infty} |\tau(t,j)|_{\A} = 0$. 
\end{itemize}
%
%where the function $c$ vanishes as $t + j \to \infty$, i.e. 
%$$
%\lim_{t+j \to \infty} c(t,j) = 0
%$$
%then solutions to $\HS_N$ with initial condition $\tau(0,0) \in P_N \setminus \X$ and $(t,j) \in \dom \tau$ are given by
%$$\lim_{t+j \to \infty} |\tau(t,j)|_{\A} = 0.$$ 
\end{theorem}
\begin{proof}
Consider a maximal solution $\tau$ to $\HS_N$ with initial condition $\tau(0,0) \in P_N \setminus \X$.
This proof uses the  
function $V$ from the proof of Theorem~\ref{thm:stability}. With $V$ equal to the distance from $\tau$ to the set $\wt\A$, then, for each $\tau \in D \setminus \X$, we have that 
$V(G(\tau)) - V(\tau) = \e V(\tau)$.
Using the fact that $V(\tau) = |\tau|_{\wt\A}$ and the fact that, $G$ along the solution is single valued, it follows that $|\tau|_{\wt\A}$ after a jump can be equivalently written as
\IfConf{$
|\tau(t_j,j+1)|_{\wt\A} = (1+\e)|\tau(t_j,j)|_{\wt\A}.
$}{$$
|\tau(t_j,j+1)|_{\wt\A} = (1+\e)|\tau(t_j,j)|_{\wt\A}.
$$}
By assumption, in between jumps, the distance to the set $\wt\A$ is such that
$\frac{d}{dt}|\tau(t,j)|_{\wt\A} = c(t,j)$,
which implies that at $t_{j+1}$ the distance to the desynchronization set is given by 
$$|\tau(t_{j+1},j)|_{\wt\A} = \int_{t_{j}}^{t_{j+1}} c(s,j) ds + |\tau(t_{j},j)|_{\wt\A}.$$
It follows that  
\IfConf{
\begin{eqnarray*}|\tau(t_{1},0)|_{\wt\A} &=& \int_{0}^{t_{1}} c(s,0) ds + |\tau(0,0)|_{\wt\A},\\
|\tau(t_{1},1)|_{\wt\A} &=& (1+\e)\left(\int_{0}^{t_{1}} c(s,0) ds + |\tau(0,0)|_{\wt\A}\right) \\ &=& (1+\e)\int_{0}^{t_{1}} c(s,0) ds + (1+\e)|\tau(0,0)|_{\wt\A}, \\
|\tau(t_{2},1)|_{\wt\A} &=& \int_{t_{1}}^{t_{2}} c(s,1) ds + (1+\e)\int_{0}^{t_{1}} c(s,0) ds \\ && \qquad \qquad \qquad \qquad \qquad  + (1+\e)|\tau(0,0)|_{\wt\A}, \\
|\tau(t_{2},2)|_{\wt\A} &=& (1+\e) \left(\int_{t_{1}}^{t_{2}} c(s,1) ds  \right. \\ &&  \left.+ (1+\e)\int_{0}^{t_{1}} c(s,0) ds + (1+\e)|\tau(0,0)|_{\wt\A} \right)
\end{eqnarray*}
}{\begin{align*}
|\tau(t_{1},0)|_{\wt\A} & = \int_{0}^{t_{1}} c(s,0) ds + |\tau(0,0)|_{\wt\A} \\
|\tau(t_{1},1)|_{\wt\A} & = (1+\e)\left(\int_{0}^{t_{1}} c(s,0) ds + |\tau(0,0)|_{\wt\A}\right) = (1+\e)\int_{0}^{t_{1}} c(s,0) ds + (1+\e)|\tau(0,0)|_{\wt\A} \\
|\tau(t_{2},1)|_{\wt\A} & =\int_{t_{1}}^{t_{2}} c(s,1) ds + (1+\e)\int_{0}^{t_{1}} c(s,0) ds + (1+\e)|\tau(0,0)|_{\wt\A} \\
|\tau(t_{2},2)|_{\wt\A} & =(1+\e) \left(\int_{t_{1}}^{t_{2}} c(s,1) ds + (1+\e)\int_{0}^{t_{1}} c(s,0) ds + (1+\e)|\tau(0,0)|_{\wt\A} \right). 
%\IfConf{.}{\\
%& =(1+\e) \int_{t_{1}}^{t_{2}} c(t,1) dt + (1+\e)^{2}\int_{0}^{t_{1}} c(t,0) dt + (1+\e)^{2}|\tau(0,0)|_{\wt\A}  \\
%|\tau(t_{3},2)|_{\wt\A} & = \int_{t_{2}}^{t_{3}} c(t,2) dt + (1+\e) \int_{t_{1}}^{t_{2}} c(t,1) dt + (1+\e)^{2}\int_{0}^{t_{1}} c(t,0) dt + (1+\e)^{2}|\tau(0,0)|_{\wt\A}\\
%|\tau(t_{3},3)|_{\wt\A} & = (1+\e)\left(\int_{t_{2}}^{t_{3}} c(t,2) dt + (1+\e) \int_{t_{1}}^{t_{2}} c(t,1) dt + (1+\e)^{2}\int_{0}^{t_{1}} c(t,0) dt  +  (1+\e)^{2}|\tau(0,0)|_{\wt\A}  \right) \\
%& = (1+\e)\int_{t_{2}}^{t_{3}} c(t,2) dt + (1+\e)^{2} \int_{t_{1}}^{t_{2}} c(t,1) dt + (1+\e)^{3}\int_{0}^{t_{1}} c(t,0) dt  + (1+\e)^{3}|\tau(0,0)|_{\wt\A}. }
\end{align*}}
Then, proceeding in this way, we obtain
\begin{align*}
|\tau(t_{j},j)|_{\wt\A} &= (1+\e)^{j}|\tau(0,0)|_{\wt\A} \\ & \qquad  \qquad \qquad +\sum_{i = 0}^{j-1}(1+\e)^{j-i}\int_{t_{i}}^{t_{i+1}} c(s,i)ds. 
\end{align*}
For the case of generic $t_{j+1} \geq t \geq t_j$, we have that 
\begin{align*}
|\tau(t,j)|_{\wt\A} = (1+\e)^{j}|\tau(0,0)|_{\wt\A} + \sum_{i = 0}^{j}(1+\e)^{j-i}\int_{t_{i}}^t c(s,i)ds.
\end{align*}
%Taking the limit as $t + j \to \infty$, it follows that
%$\lim_{t + j \to \infty} |\tau(t,j)|_{\wt\A} = \lim_{t + j \to \infty} (1+\e)^{j}|\tau(0,0)|_{\wt\A} + \sum_{i = 0}^{j}(1+\e)^{j-i}\int_{t_{i}}^{t} c(s,i)ds$.
Since, we know that as either $t$ or $j$ goes to infinity, $j$ or $t$ go to infinity as well, respectively. The expression reduces to
\IfConf{$\lim_{t + j \to \infty} |\tau(t,j)|_{\wt\A} = \lim_{j \to \infty} (1+\e)^{j}|\tau(0,0)|_{\wt\A} + \lim_{t + j \to \infty} \sum_{i = 0}^{j}(1+\e)^{j-i}\int_{t_{i}}^{t} c(s,i)ds 
= \lim_{t + j \to \infty} \sum_{i = 0}^{j}(1+\e)^{j-i}\int_{t_{i}}^{t} c(s,i)ds.$}{\begin{align}
\lim_{t + j \to \infty} |\tau(t,j)|_{\wt\A} &= \lim_{j \to \infty} (1+\e)^{j}|\tau(0,0)|_{\wt\A} + \lim_{t + j \to \infty} \sum_{i = 0}^{j}(1+\e)^{j-i}\int_{t_{i}}^{t} c(s,i)ds 
= \lim_{t + j \to \infty} \sum_{i = 0}^{j}(1+\e)^{j-i}\int_{t_{i}}^{t} c(s,i)ds. \label{eqn:limctj}
\end{align}}
If $c(t,j) \leq \bar{c}$, it follows that
\IfConf{$\lim_{t + j \to \infty} |\tau(t,j)|_{\wt\A} = \lim_{t + j \to \infty} \sum_{i = 0}^{j}(1+\e)^{j-i}\int_{t_{i}}^{t} c(s,i)ds \leq \frac{\bar{c}\tb}{|\e|\omega}$.}{
\begin{align*}
\lim_{t + j \to \infty} |\tau(t,j)|_{\wt\A} &= \lim_{t + j \to \infty} \sum_{i = 0}^{j}(1+\e)^{j-i}\int_{t_{i}}^{t} c(s,i)ds \NotForConf{ \\
&\leq \lim_{t + j \to \infty} \sum_{i = 0}^{j}(1+\e)^{j-i}\int_{t_{i}}^{t_{i+1}} \bar{c} dt \\
%&= \lim_{t + j \to \infty} \sum_{i = 0}^{j}(1+\e)^{j-i}\bar{c}(t_{i+1} - t_{i}) \\
&\leq \frac{c\tb}{\omega} \lim_{t + j \to \infty} \sum_{i = 0}^{j}(1+\e)^{j-i} \\
& = \frac{\bar{c}\tb}{\omega} \lim_{t + j \to \infty} \frac{(1+\e)^{j} - 1}{(1+\e) - 1} \\
%& = \frac{\bar{c}\tb}{\omega} \frac{- 1}{\e} \\
&} \leq \frac{\bar{c}\tb}{|\e|\omega}.
\end{align*}}

Lastly, since this hybrid system has the property that for any maximal solution $\tau$ with  $(t,j) \in \dom \tau$, if $t$ approaches $\infty$  then the parameter $j$ also approaches $\infty$, the expression given by $\lim_{t + j \to \infty} |\tau(t,j)|_{\wt\A}$ can be simplified. To do this, we know that the series
$\sum_{i=0}^j (1+\e)^{j -i} = \frac{(1+\e)^{j+1}-1}{\e}$
approaches $\frac{1}{|\e|}$ as $j \to \infty$. Since $1+\e >0$ for $\e \in (-1,0)$, the series is absolutely convergent and its partial sum 
$s_j = \sum_{i=0}^j (1+\e)^{j -i}$
is such that $\{s_j\}^\infty_{j=m}$ is a nondecreasing sequence (for each $m$). This implies that 
$s_j \leq 1/|\e|$ for all $j$ and for each $m$.
Then, it follows that
$(1+\e)^{j-i} \leq \frac{1}{|\e|}$ for every $j,i \in \nats$.
Since the expression is a function of $j$ only and, for complete solutions, $t$ is such that as $t \to \infty$, then $j \to \infty$, we obtain
\IfConf{$\lim_{t+j\to \infty} \sum_{i=0}^j (1+\e)^{j-i} \int_{t_i}^{t} c(s,i) ds= \lim_{j\to \infty} \sum_{i=0}^j (1+\e)^{j-i} \int_{t_i}^{t} c(s,i) ds  \leq  \frac{1}{|\e|}  \int_{0}^{\infty} |c(s,j(s))| ds. $
}{\begin{align*}
\lim_{t+j\to \infty} \sum_{i=0}^j (1+\e)^{j-i} \int_{t_i}^{t} c(s,i) ds&= \lim_{j\to \infty} \sum_{i=0}^j (1+\e)^{j-i} \int_{t_i}^{t} c(s,i) ds  \\ 
&\leq  \lim_{j\to \infty} \sum_{i=0}^j (1+\e)^{j-i} \int_{t_i}^{t} |c(s,i)| ds \\ 
&\leq \left( \sum_{i=0}^\infty (1+\e)^{j-i} \right)\int_{0}^{\infty} |c(s,i)| ds \\ 
&\leq  \frac{1}{|\e|}  \int_{0}^{\infty} |c(s,\wt{j}(s))| ds.
\end{align*}}
\end{proof}

%\newpage \noindent 

%\NotForConf{
%\begin{figure}[h!]
%\centering
%\psfrag{09}[][][.5]{-$0.9$}
%\psfrag{08}[][][.5]{-$0.8$}
%\psfrag{07}[][][.5]{-$0.7$}
%\psfrag{06}[][][.5]{-$0.6$}
%\psfrag{05}[][][.5]{-$0.5$}
%\psfrag{04}[][][.5]{-$0.4$}
%\psfrag{03}[][][.5]{-$0.3$}
%\psfrag{02}[][][.5]{-$0.2$}
%\psfrag{01}[][][.5]{-$0.1$}
%\psfrag{Nstar}[][][1]{$N^{*}$}
%\psfrag{eps}[][][1]{$\varepsilon$}
%\psfrag{c15}[][][.5]{$c_{1} = .5$}
%\psfrag{c13}[][][.5]{$c_{1} = .3$}
%\psfrag{c11}[][][.5]{$c_{1} = .1$}
%\psfrag{c105}[][][.5]{$c_{1} = .05$}
%\includegraphics[width=.8\textwidth]{Figures/Nstare}
%\caption{Time to converge as a function of $\varepsilon \in [-.9,-.1]$, with $c_{2} = .99\tb$ and $c_{1} \in \{0.5\tb, 0.3\tb, 0.1\tb, 0.05\tb\}$}
%\label{fig:timetoconverge}
%\end{figure}
%}
\vspace{-.2cm}
\section{Numerical Analysis}\label{sec:numerics}
This section presents numerical results obtained from simulating $\HS_N$. First, we present results for the nominal case of $\HS_N$ given by \eqref{eqn:HS}. Then, we present results for $\HS_N$ under different types of perturbations. The Hybrid Equations (HyEQ) Toolbox in \cite{Sanfelice.ea.13.HSCC} was used to compute the trajectories.
\subsection{Nominal Case}
The possible solutions to the hybrid system $\HS_N$ fall into four categories: always desynchronized, asymptotically desynchronized, never desynchronized, and initially synchronized. \IfConf{{Due to space constraints, in this article we present numerical results for the case of asymptotically desynchronized solutions. For more information regarding each case, see \cite{Phillips2013TechReport}.}}{The following simulation results show the evolution of solutions for each category.} The parameters used in these simulations are $\tb = 1$ and $\varepsilon = -0.2$.
\NotForConf{
% In Tech Report Only
\subsubsection{Always desynchronized ($N \in \{2,3\}$)} 
\begin{figure}[]
\centering
\subfigure[Solutions to $\HS_{2}$ with $\tau(0,0) \in \A$]{
\psfragfig*[width = .4\textwidth,trim = 15mm 5mm 15mm 10mm, clip]{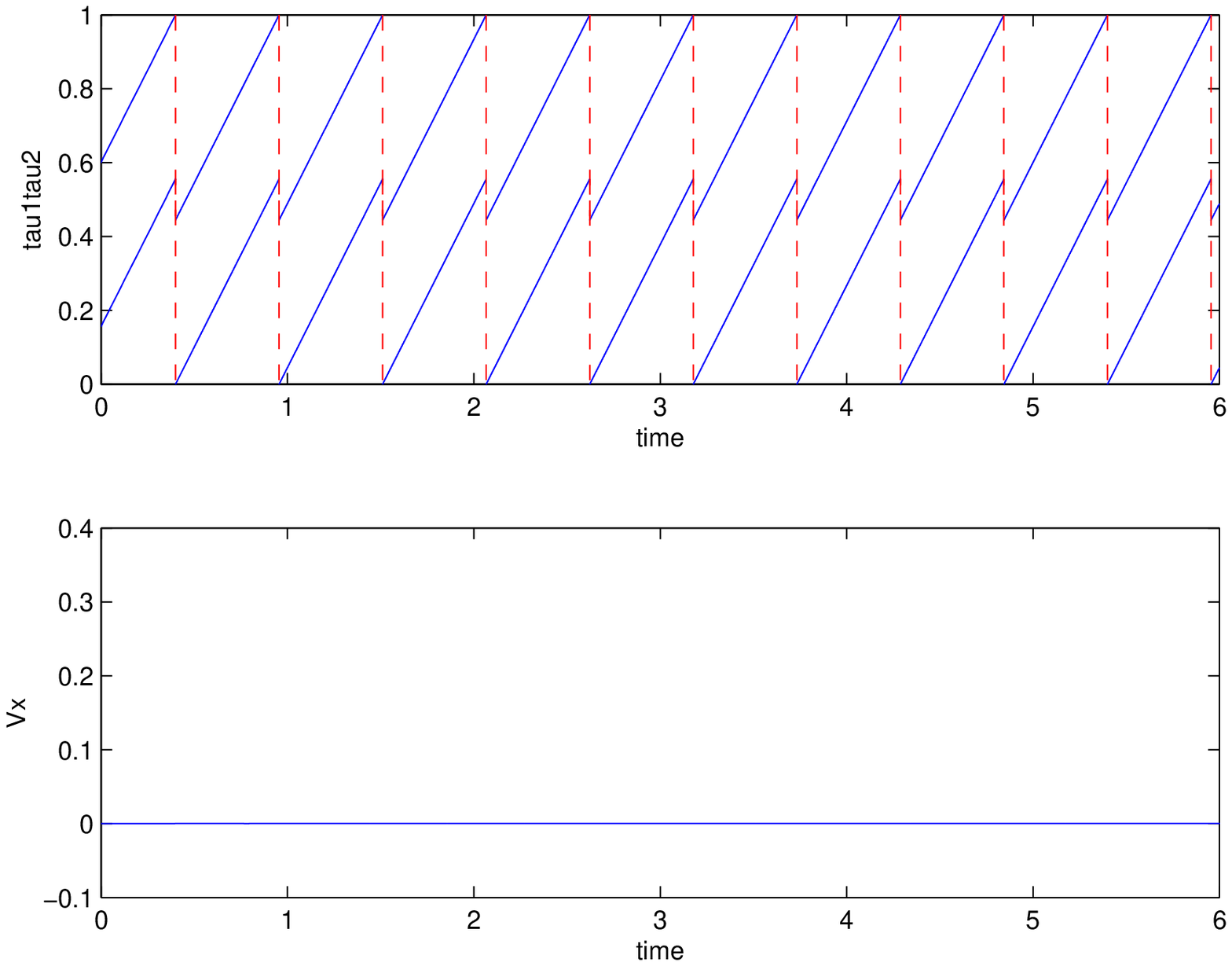}{
\psfrag{tau1tau2}[][][.7]{$\ton,\ttw$}
\psfrag{time}[][][.7]{$t$ [seconds]}
\psfrag{flows [t]}[][][.8]{}
\psfrag{Vx}[][][.7]{$V(\tau)$}
}}
\ \subfigure[Solutions to $\HS_{3}$ with $\tau(0,0) \in \A_3$]{
\psfragfig*[width = .4\textwidth,trim = 15mm 5mm 15mm 10mm, clip]{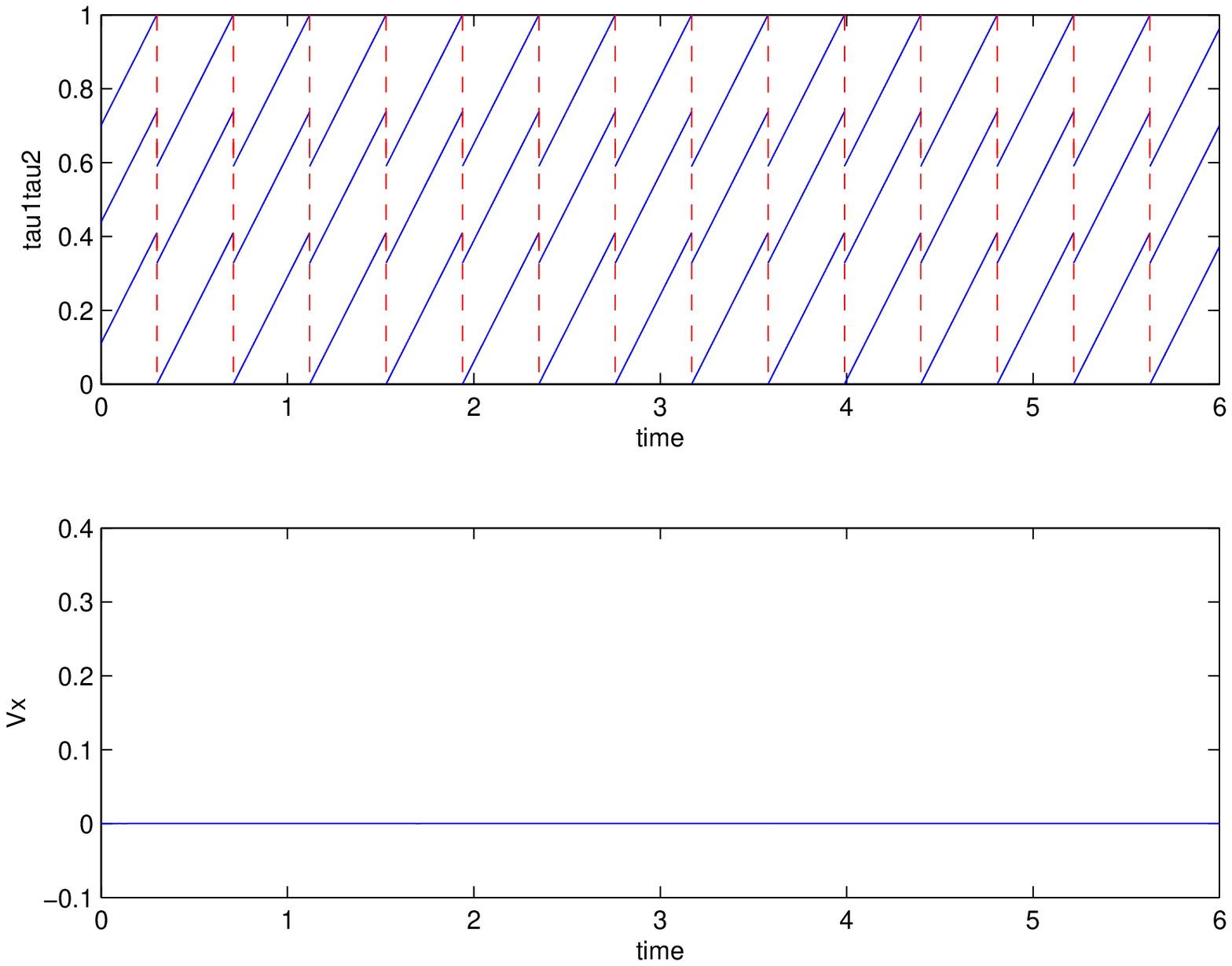}{
\psfrag{time}[][][.7]{$t$ [seconds]}
\psfrag{flows [t]}[][][.8]{}
\psfrag{Vx}[][][.7]{$V(\tau)$}
\psfrag{tau1tau2}[][][.7]{$\ton, \ttw, \tau_3$}
}}
\caption{Solutions to $\HS_N$ with $N \in \{2,3\}$ that are initially in the set $\A$.}
\label{fig:desynchsim}
\vspace{-.2cm}
\end{figure}
A solution to $\HS_N$ that has initial condition $\tau(0,0) \in \A$ stays desynchronized. Figure~\ref{fig:desynchsim} shows the evolution of such a solution for systems $\HS_{2}$ and $\HS_{3}$. Furthermore, as also shown in the figures, for these same solutions, the Lyapunov function is initially zero and stays equal to zero as hybrid time goes on.

}

\NotForConf{\subsubsection{Asymptotically desynchronized ($N \in \{2,3,7,10\}$)}} A solution of $\HS_N$ that starts in $P_N \setminus (\X \cup \A)$ asymptotically converges to $\A$, as Theorem~\ref{thm:timetoconverge} indicates. \IfConf{Figure~\ref{fig:H2notinX2} and Figure~\ref{fig:H3notinX3}}{Figure~\ref{fig:asympdesynch}} show solutions to both $\HS_{2}$ and $\HS_{3}$ converging to their respective desynchronization sets.

\IfConf{
%In Journal Version
\begin{figure}[]
\centering
\setlength{\unitlength}{0.019cm}
\psfrag{tau1tau2}[][][.6]{$\ton,\ttw$}
\psfrag{time}[][][.6]{$t$ [sec]}
\psfrag{flows [t]}[][][.8]{}
\subfigure[Solutions to $\HS_{2}$ with $c_{2} = 0.24$ and $\tau(0,0) = {[0,0.1]^{\top}} \in P_2 \setminus \X_2$.]{
\label{fig:H2notinX2}
\psfrag{Vx}[][][.6]{\hspace{-.5cm}$V(\tau)$}
\includegraphics[width = .23\textwidth,trim = 15mm 5mm 15mm 10mm, clip]{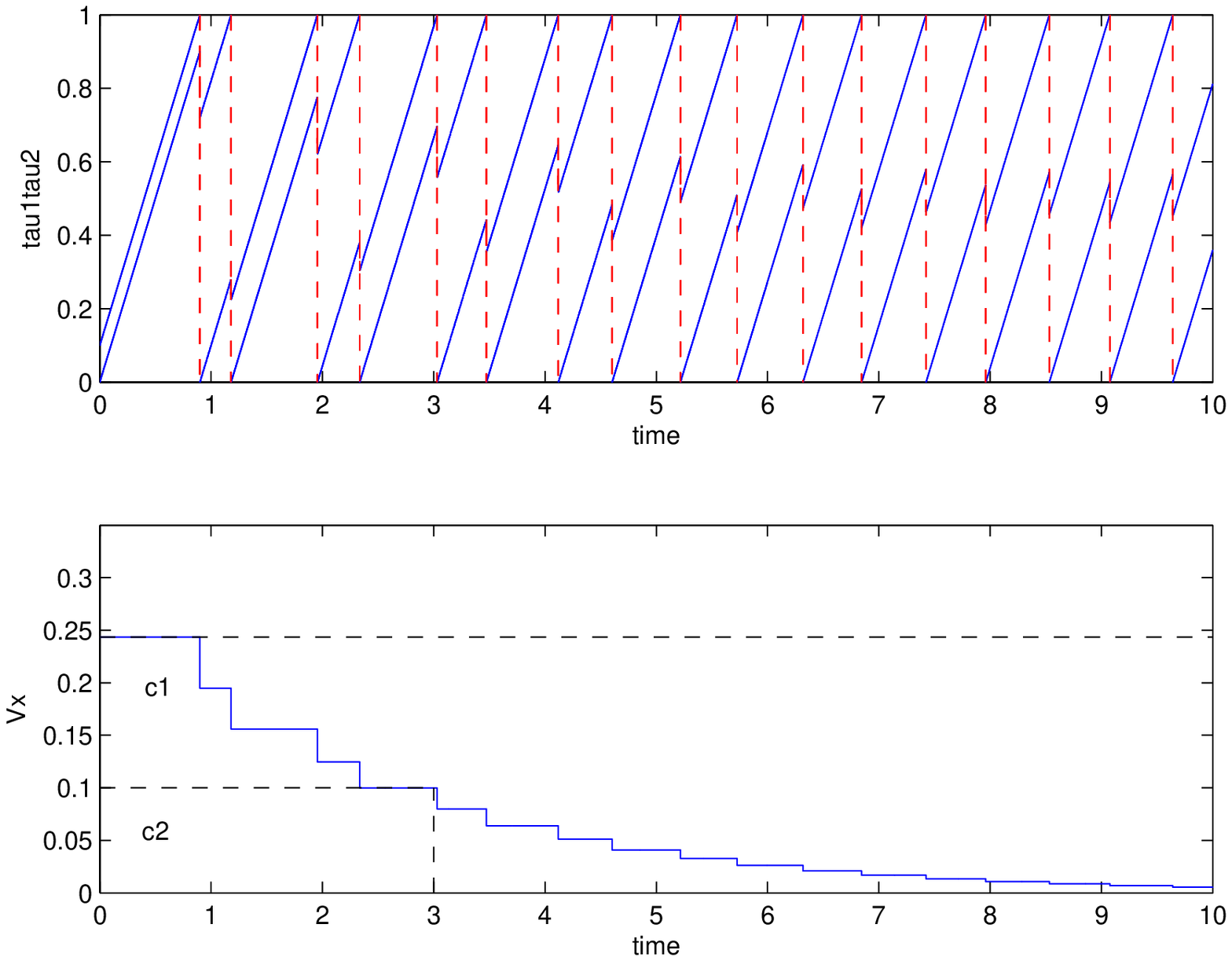}
\put(-195,37){{\scalebox{.6}{$c_{1}$}}}
\put(-195,63){{\scalebox{.6}{$c_{2}$}}}
}
\ \subfigure[Solutions to $\HS_{3}$ with $c_{2} = 0.32$ and $\tau(0,0) = {[0, 0.1,0.2]}^{\top} \in P_3 \setminus \X_3$.]{
\label{fig:H3notinX3}
\psfrag{Vx}[][][.6]{\hspace{-.5cm}$V(\tau)$}
\psfrag{tau1tau2}[][][.6]{$\ton,\ttw,\tau_3$}
\includegraphics[width = .23\textwidth,trim = 15mm 5mm 15mm 10mm, clip]{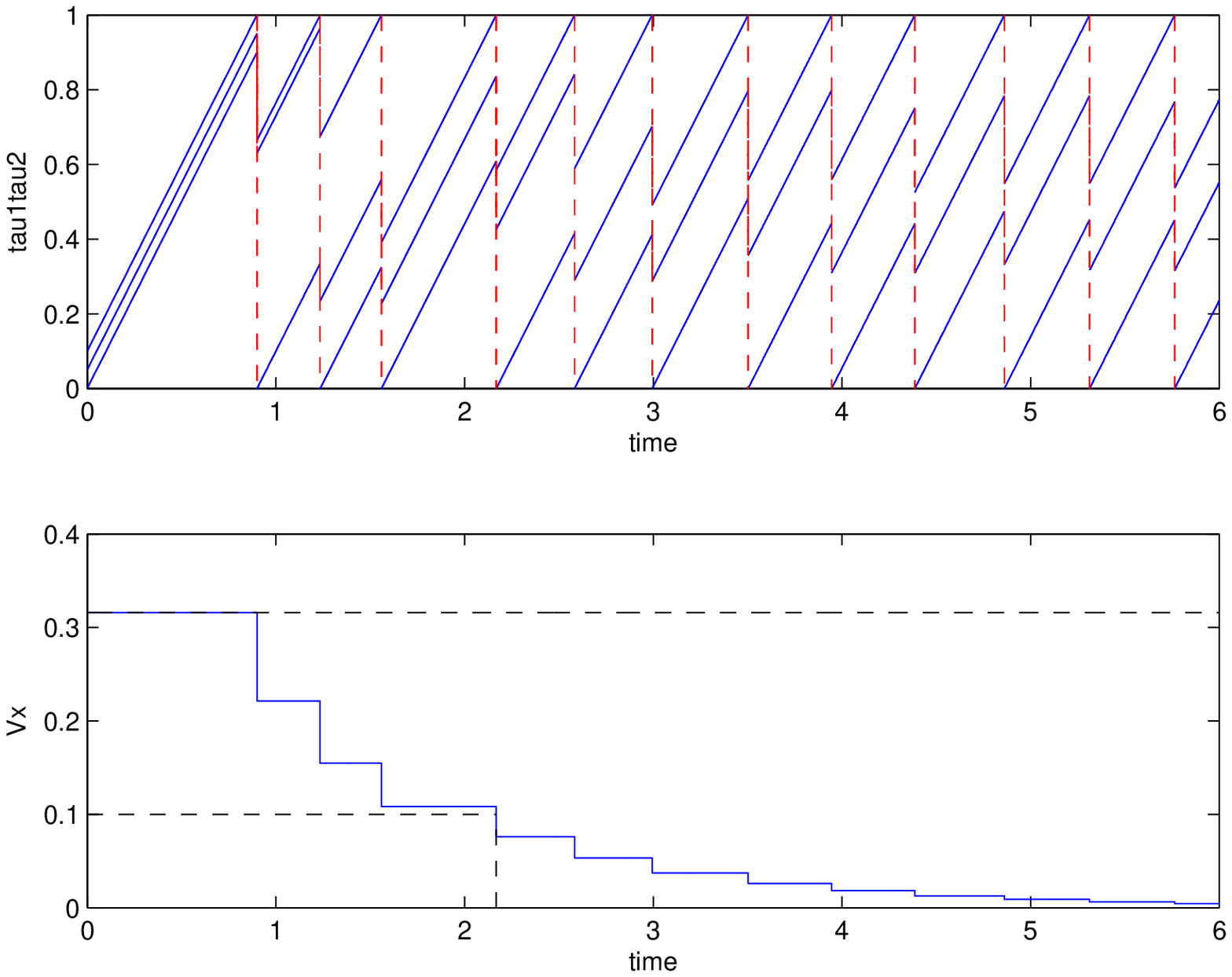}
\put(-195,35){{\scalebox{.6}{$c_{1}$}}}
\put(-195,70){{\scalebox{.6}{$c_{2}$}}}
}\\
\psfrag{tau1tau2}[][][.6]{$\tau_1,\tau_{2}$}
\psfrag{time}[][][.6]{$t$ [sec]}
\psfrag{flows [t]}[][][.8]{}
\subfigure[Solutions to $\HS_{7}$.]{
\psfrag{Vx}[][][.6]{}
\label{fig:H7notinX7}\includegraphics[width = .23\textwidth,trim = 15mm 80mm 15mm 10mm, clip]{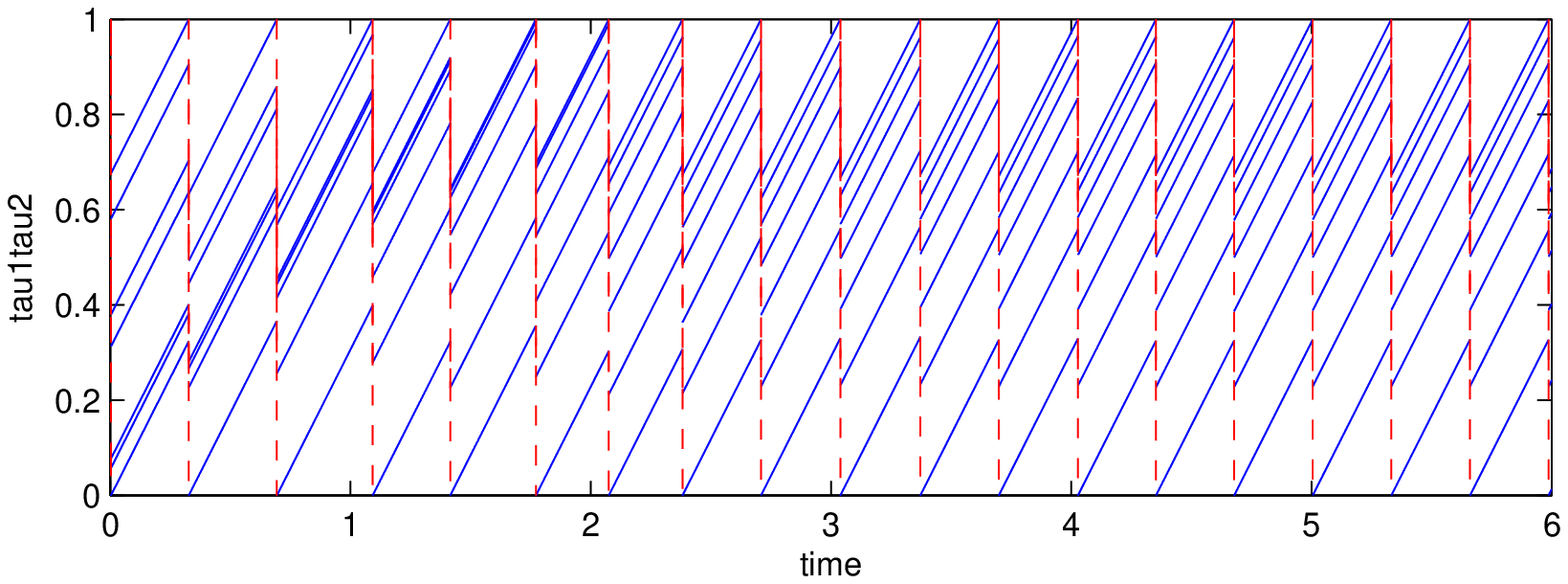}}
\subfigure[Solutions to $\HS_{10}$.]{\label{fig:H10notinX10}
\psfrag{Vx}[][][.6]{}
\psfrag{tau1tau2}[][][.6]{$\tau_1,\tau_{2}$}
\includegraphics[width = .23\textwidth,trim = 15mm 80mm 15mm 10mm, clip]{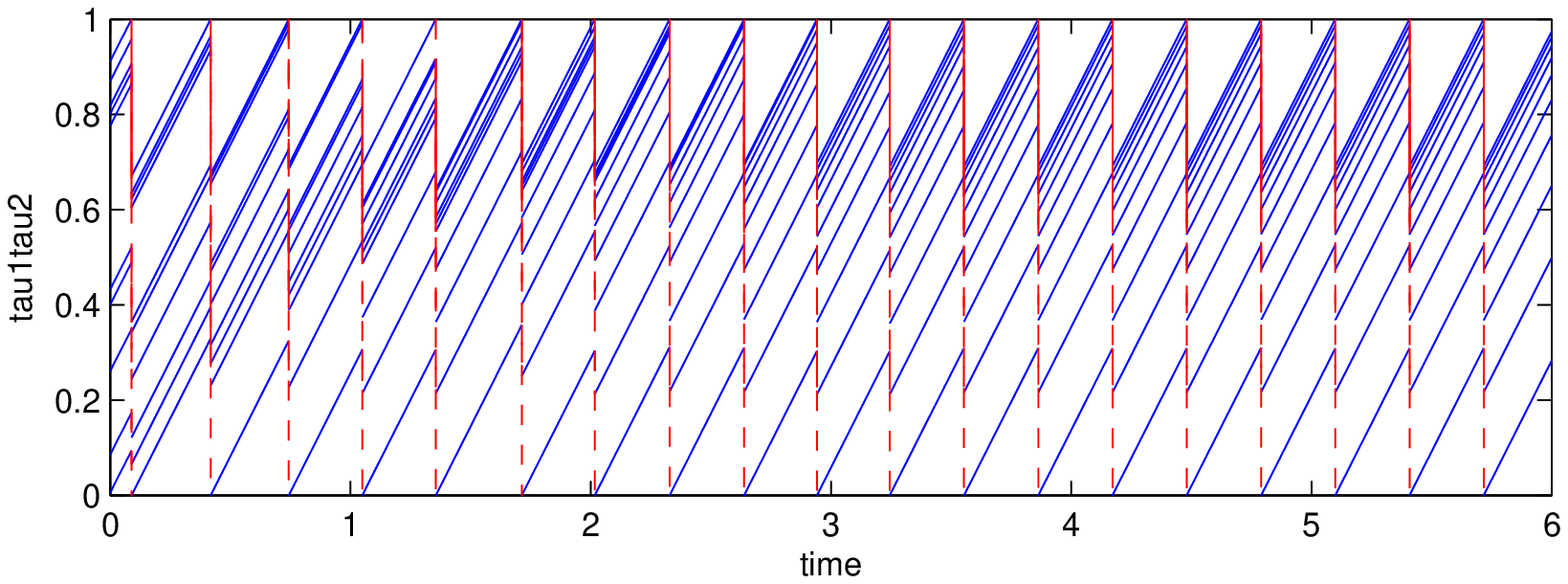}}
\caption{Solutions to $\HS_N$ that asymptotically converge to the set $\A$ for $N \in \{2,3,7,10\}$.}
\label{fig:asympdesynch}
\end{figure}}
%%%%%%%%%
{ % In Tech Report
\begin{figure}[]
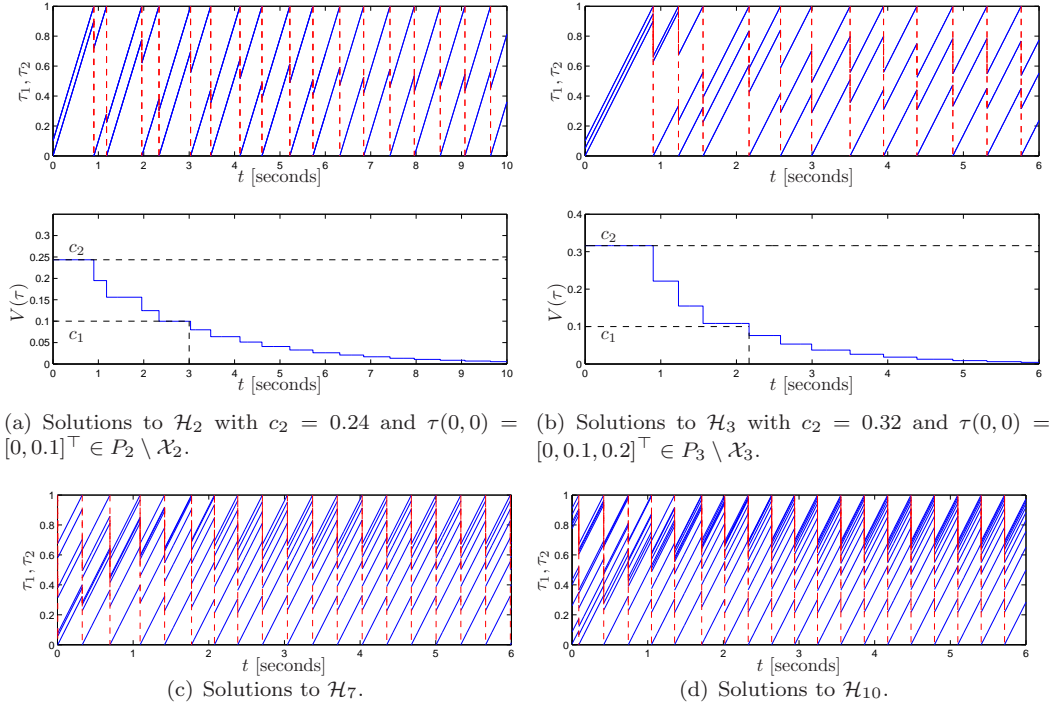

\centering
%\setlength{\unitlength}{0.023cm}
%\psfrag{tau1tau2}[][][.7]{$\ton,\ttw$}
%\psfrag{time}[][][.7]{$t$ [seconds]}
%\psfrag{flows [t]}[][][.8]{}
%\psfrag{c1}[][][.7]{$$}
%\psfrag{c2}[][][.7]{$$}
\subfigure[Solutions to $\HS_{2}$ with $c_{2} = 0.24$ and $\tau(0,0) = {[0,0.1]^{\top}} \in P_2 \setminus \X_2$.]{
\label{fig:H2notinX2}
\psfragfig*[width = .4\textwidth,trim = 15mm 5mm 15mm 10mm, clip]{Figures/H2notinX2}{
\psfrag{Vx}[][][.7]{\hspace{-.5cm}$V(\tau)$}
\psfrag{tau1tau2}[][][.7]{$\ton,\ttw$}
\psfrag{time}[][][.7]{$t$ [seconds]}
\psfrag{flows [t]}[][][.8]{}
\psfrag{c1}[][][.7]{}
\psfrag{c2}[][][.7]{}
}
\put(-170,22){{\scalebox{.7}{$c_{1}$}}}
\put(-170,55){{\scalebox{.7}{$c_{2}$}}}
}
\ \subfigure[Solutions to $\HS_{3}$ with $c_{2} = 0.32$ and $\tau(0,0) = {[0, 0.1,0.2]}^{\top} \in P_3 \setminus \X_3$.]{
\label{fig:H3notinX3}
\psfrag{tau1tau2}[][][.7]{$\ton,\ttw,\tau_3$}
\psfragfig*[width = .4\textwidth,trim = 15mm 5mm 15mm 10mm, clip]{Figures/H3notinA3}{
\psfrag{Vx}[][][.7]{\hspace{-.5cm}$V(\tau)$}
\psfrag{tau1tau2}[][][.7]{$\ton,\ttw$}
\psfrag{time}[][][.7]{$t$ [seconds]}
\psfrag{flows [t]}[][][.8]{}
\psfrag{c1}[][][.7]{$$}
\psfrag{c2}[][][.7]{$$}
}
\put(-170,22){{\scalebox{.7}{$c_{1}$}}}
\put(-170,60){{\scalebox{.7}{$c_{2}$}}}
}\\
\psfrag{tau1tau2}[][][.7]{$\tau_1,\tau_{2}$}
\psfrag{time}[][][.7]{$t$ [seconds]}
\psfrag{flows [t]}[][][.8]{}
\subfigure[Solutions to $\HS_{7}$.]{
\label{fig:H7notinX7}
\psfragfig*[width = .4\textwidth,trim = 15mm 80mm 15mm 10mm, clip]{Figures/H7notinX7}{
\psfrag{Vx}[][][.7]{}
\psfrag{tau1tau2}[][][.7]{$\tau_1,\tau_{2}$}
\psfrag{time}[][][.7]{$t$ [seconds]}
\psfrag{flows [t]}[][][.8]{}
}}
\subfigure[Solutions to $\HS_{10}$.]{\label{fig:H10notinX10}
\psfragfig*[width = .4\textwidth,trim = 15mm 80mm 15mm 10mm, clip]{Figures/H10notinX10}}{
\psfrag{tau1tau2}[][][.7]{$\tau_1,\tau_{2}$}
\psfrag{time}[][][.7]{$t$ [seconds]}
\psfrag{flows [t]}[][][.8]{}
\psfrag{Vx}[][][.7]{}}
\caption{Solutions to $\HS_N$ that asymptotically converge to the set $\A$ for $N \in \{2,3,7,10\}$.}
\label{fig:asympdesynch}
\end{figure}
}

For $\HS_2$, if $\tau(0,0) = [0,0.1]^\top$, then the initial sublevel set is $\wt L_{V}(c_{2})$ with $c_{2} = 0.24$. Using Theorem~\ref{thm:timetoconverge}, the time to converge to the sublevel set $\wt L_{V}(c_{1})$ with $c_{1} = 0.1$ leads to $M = 7.84$. Figure~\ref{fig:H2notinX2} shows a solution to the system for 10 seconds of flow time. From the figure, it can be seen that $V(\tau(t,j)) \approx 0.1$ at $(t,j) = (3,4)$. Then, the property guaranteed by Theorem~\ref{thm:timetoconverge}, namely, $V(\tau(t,j)) \leq c_{1}$ for each $(t,j)$ such that $t + j \geq M$, is satisfied. 
Figure~\ref{fig:H3notinX3}, shows a solution and the distance of this solution to $\A$. Notice that the initial sub level set is $\wt{L}_{V}(c_{2})$ with $c_{2} =0.32$. From Theorem~\ref{thm:timetoconverge} it follows that the time to converge to $\wt{L}_{V}(c_{1})$ with $c_{1} = 0.1$ is given by $M =10.14$, which is actually already satisfied at $(t,j) = (2.2,4)$.
\IfConf{Figure~\ref{fig:H7notinX7} and Figure~\ref{fig:H10notinX10}}{Figure~\ref{fig:asympdesynch}} show solutions to $\HS_N$ that asymptotically desynchronize
for $N \in \{7,10\}$.
%\begin{figure}
%\centering
%\psfrag{tau1tau2}[][][.6]{$\tau_i$}
%\psfrag{time}[][][.6]{$t$ [sec]}
%\psfrag{flows [t]}[][][.8]{}
%\subfigure[Solutions to $\HS_{7}$.]{
%\psfrag{Vx}[][][.6]{}
%\label{fig:H2inX2short}\includegraphics[width = .46\textwidth,trim = 15mm 80mm 15mm 10mm, clip]{Figures/H7notinX7}}
%\subfigure[Solutions to $\HS_{10}$.]{
%\psfrag{Vx}[][][.6]{}
%\psfrag{tau1tau2}[][][.6]{$\tau_i$}
%\includegraphics[width = .46\textwidth,trim = 15mm 80mm 15mm 10mm, clip]{Figures/H10notinX10}}
%\caption{Solutions to $\HS_N$ for $N \in \{7,10\}$ with $\tau(0,0) \in P_N \setminus \X$.}
%\label{fig:7and10}
%\end{figure}

\NotForConf{ %In Tech Report
\subsubsection{Always Synchronized} When the impulse-coupled oscillators start from an initial condition $\tau(0,0) \in \X$, a solution remains in $\X$.
Figure~\ref{fig:neverdesynch} shows solutions to $\HS_{2}$ and $\HS_{3}$ that never desynchronize.
\begin{figure}
\centering
\subfigure[Solutions to $\HS_{2}$ with $\tau(0,0) \in \X_{2}$.]{
\label{fig:H2inX2}\psfragfig*[width = .4\textwidth,trim = 15mm 5mm 15mm 10mm, clip]{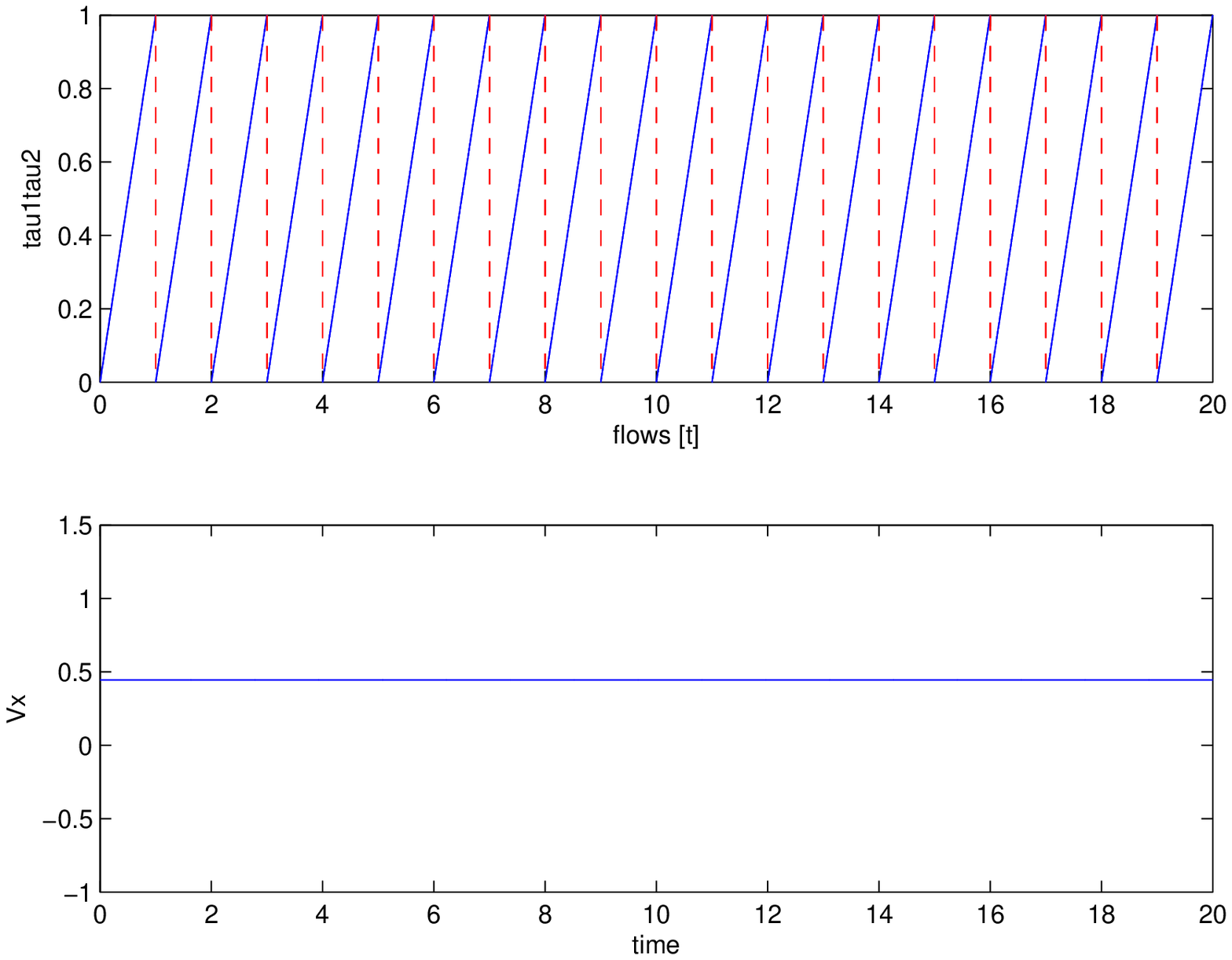}{
\psfrag{Vx}[][][.7]{$V(\tau)$}
\psfrag{tau1tau2}[][][.7]{$\ton,\ttw$}
\psfrag{time}[][][.7]{$t$ [seconds]}
\psfrag{flows [t]}[][][.8]{}
}}
\subfigure[Solutions to $\HS_{3}$ with $\tau(0,0) \in \X_{3}$.]{
\psfragfig*[width = .4\textwidth,trim = 15mm 5mm 15mm 10mm, clip]{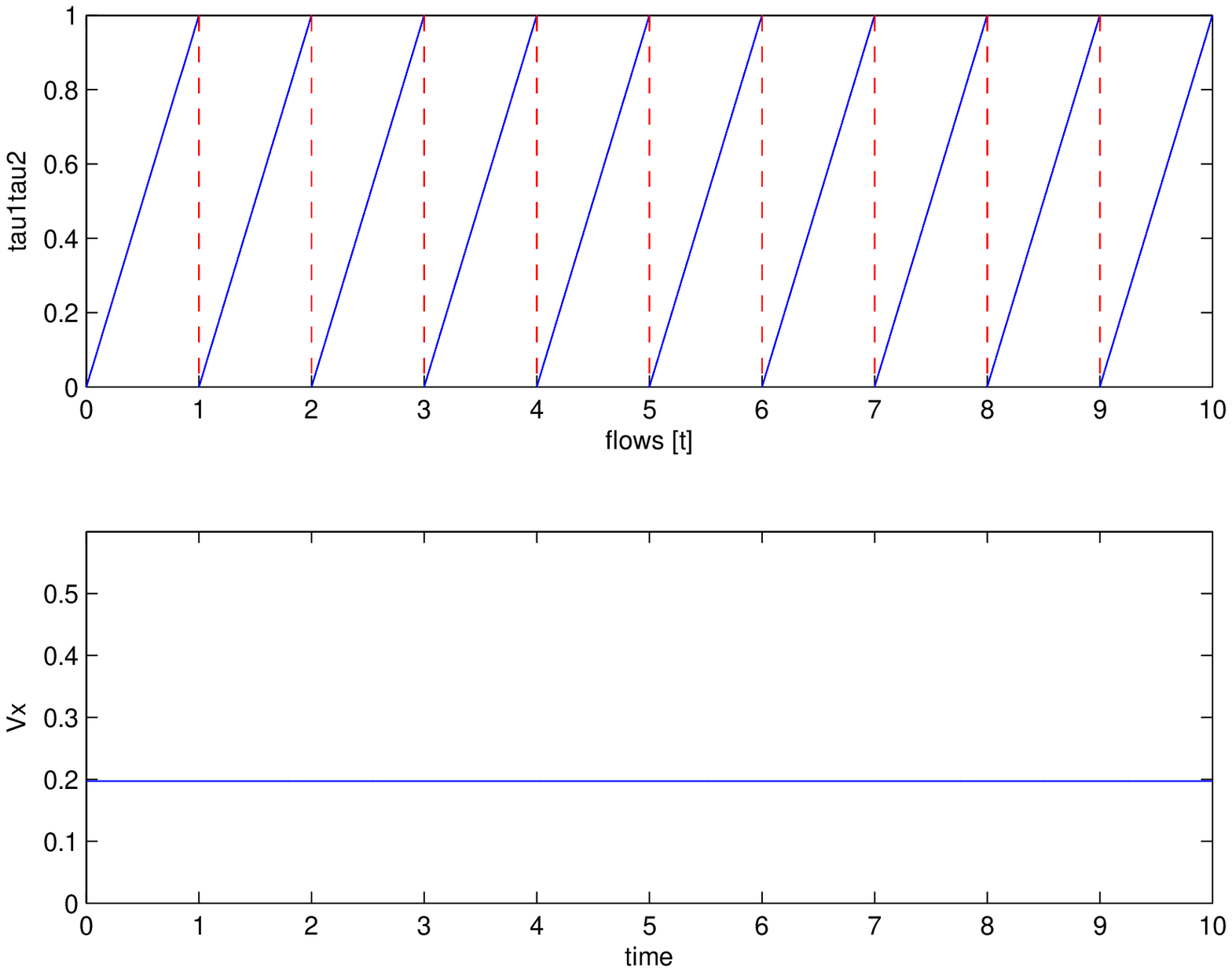}{\psfrag{Vx}[][][.7]{$V(\tau)$}
\psfrag{tau1tau2}[][][.7]{$\ton,\ttw,\tau_3$}
\psfrag{time}[][][.7]{$t$ [seconds]}
\psfrag{flows [t]}[][][.8]{}
}}
\caption{Solutions to $\HS_{N}$ that never converge to the set $\A$ for $N = \{2,3\}$.}
\label{fig:neverdesynch}
\end{figure}

It can be seen that (since $\tau(t,j) \in \X$ for all $(t,j) \in \dom \tau$) $V$ remains constant.}

\NotForConf{ %In Tech Report
\subsubsection{Initially Synchronized}
As mentioned in the proof of Theorem~\ref{thm:stability}, there exist solutions that are initialized in $\X$ and eventually become desynchronized. This is due to the set-valuedness of the jump map at such points. Figure~\ref{fig:initSyncSol} shows two different solutions to $\HS_2$ and $\HS_3$ from the same initial conditions $\tau(0,0) = [0,0,0]^\top$. Furthermore, notice that, for each $(t,j)$, the that Lyapunov function along solutions does not decrease to zero until all states are non-equal. Recall that from the analysis in Section~\ref{sec:lyapunov}, when states are equal, the issued solutions are outside of the basin of attraction. 
\begin{figure}
\centering
\subfigure[Solutions to $\HS_{2}$ with $\tau(0,0) \in \X_{2}$. Notice that the solution jumps out of $\X_2$ at $(t,j) = (3,3)$ and the function $V$ begins to decrease after that jump.]{
\psfragfig*[width = .4\textwidth,trim = 15mm 5mm 15mm 10mm, clip]{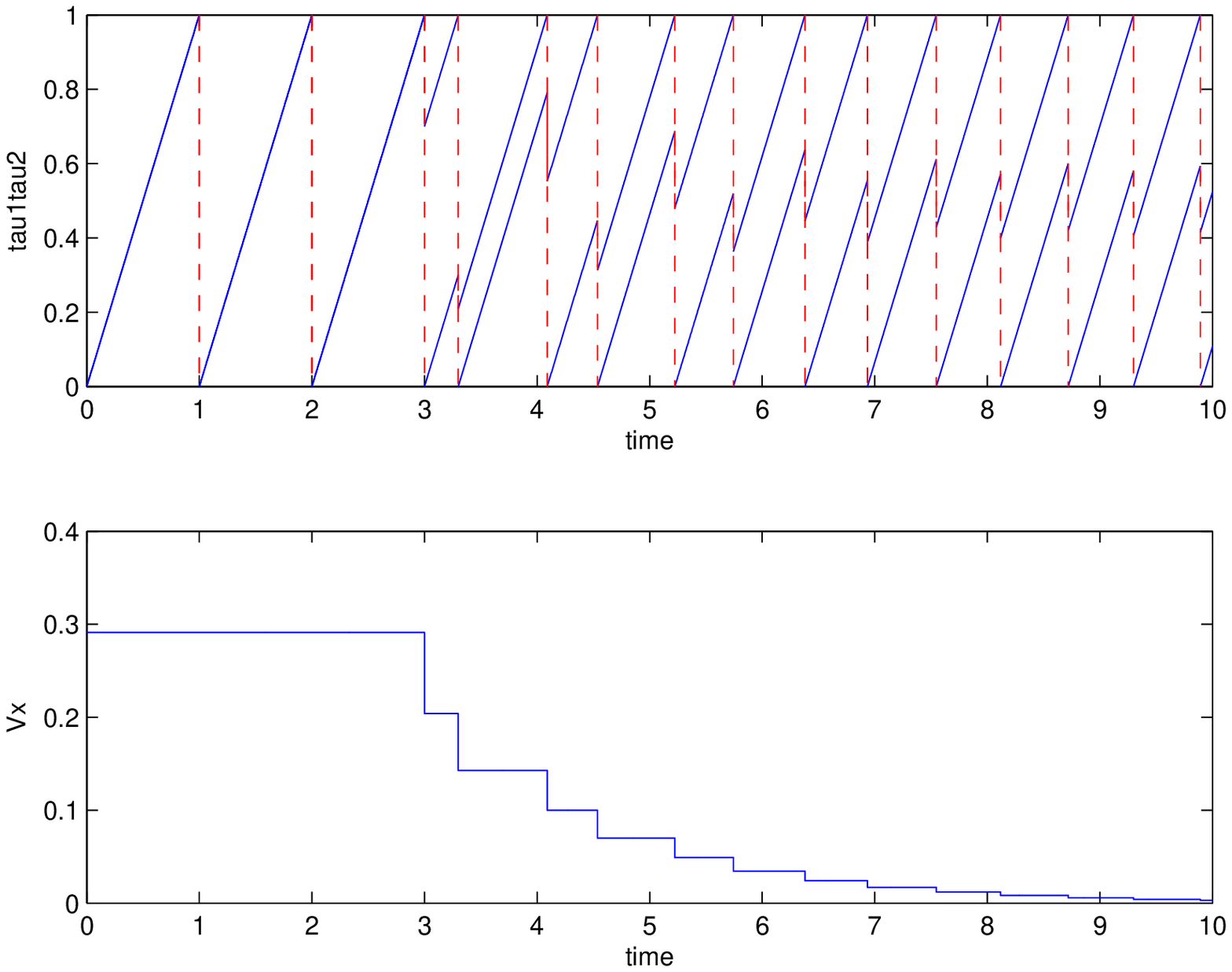}{
\psfrag{Vx}[][][.7]{$V(\tau)$}
\psfrag{tau1tau2}[][][.7]{$\ton,\ttw$}
\psfrag{time}[][][.7]{$t$ [seconds]}
\psfrag{flows [t]}[][][.8]{}
}}
\hspace{0.2cm}
\subfigure[Solutions to $\HS_{3}$ with $\tau(0,0) \in \X_{3}$. At hybrid time $(t,j) = (1,0)$ the timer state $\tau_1$ jumps away from the other two and begin to desynchronize. At approximately $(t,j) = (4.5,8)$, all of the states are not equal and $V$ begins to decrease.]{
\psfragfig*[width = .4\textwidth,trim = 15mm 5mm 15mm 10mm, clip]{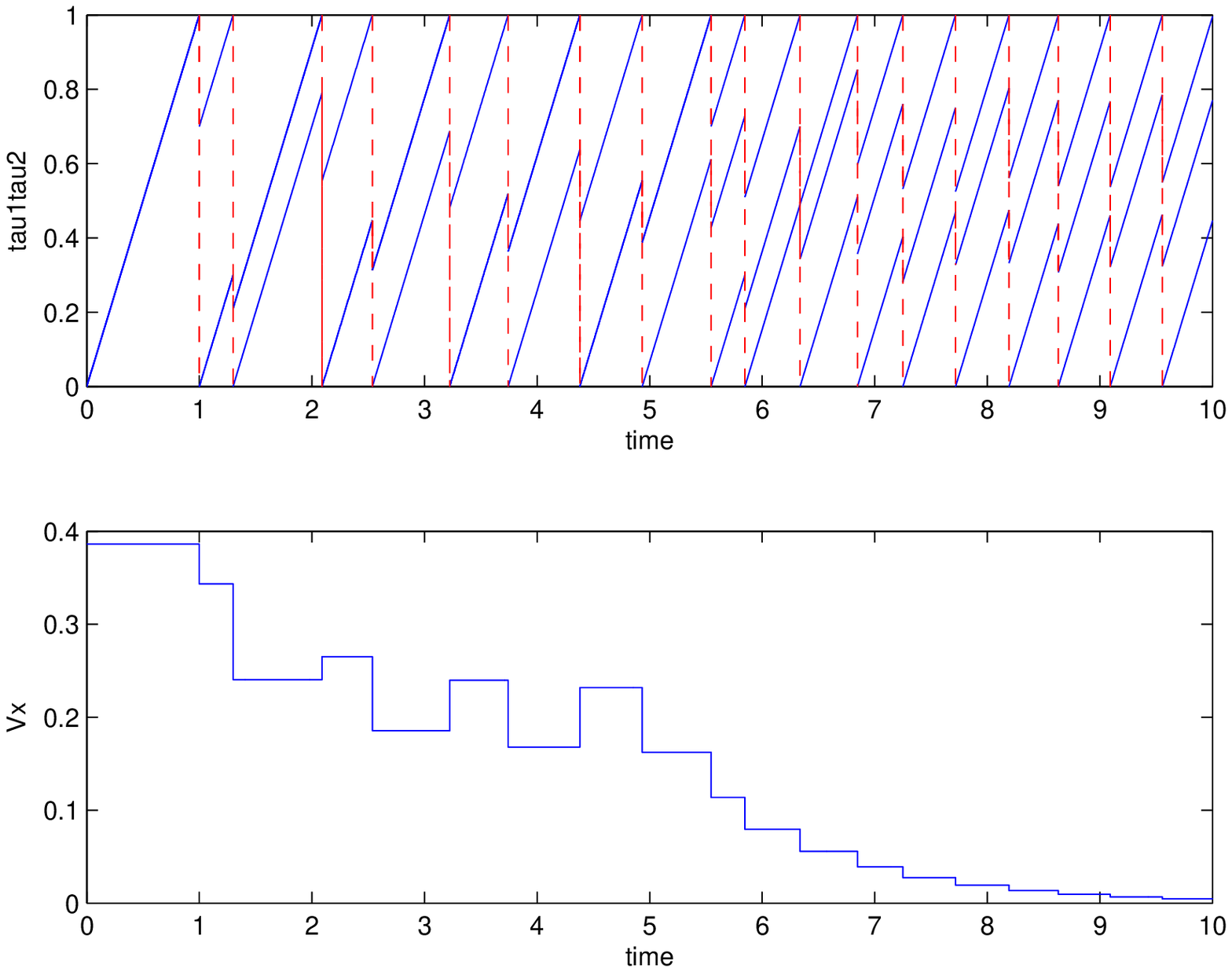}{
\psfrag{Vx}[][][.7]{$V(\tau)$}
\psfrag{tau1tau2}[][][.7]{$\ton,\ttw,\tau_3$}
\psfrag{time}[][][.7]{$t$ [seconds]}
\psfrag{flows [t]}[][][.8]{}}
}
\caption{Solutions to $\HS_{N}$ for $N \in \{2,3\}$ that initially evolve in $\X$ and eventually become desynchronized due to the set-valuedness of the jump map.}
\label{fig:initSyncSol}
\end{figure}
}

\subsection{Perturbed Case}
In this section, we present numerical results to validate the statements in Section~\ref{sec:robustness}.
\subsubsection{Simulations of $\HS_N$ with perturbed jumps}\label{sec:jumpperturbs1}
\IfConf{}{In this section, we consider a class of perturbations on the jump map and jump set. }
%\begin{enumerate} 
\NotForConf{

$\bullet$ {\bf Perturbation of the threshold in the jump set:} We replace the jump set $D$ by  $D_{\rho} := \{\tau : \exists i \in I \ s.t. \ \tau_i = \tb + \rho_i\}$ where $\rho_i \in [0, \bar{\rho}_i]$, $\bar{\rho}_i > 0$ for each $i \in I$. To avoid maximal solutions that are not complete, the flow set $C$ is replaced by $C_\rho := [0,\tb+\rho_1]\times[0,\tb+\rho_2] \times \ldots \times [0,\tb+\rho_N]$. Furthermore, the components of the jump map are also replaced by 
\begin{equation}
g_{\rho_i} (\tau) = \left\{ \begin{array}{ll} 0 \qquad \qquad \ \ \ \ &\mbox{if } \tau_{i} = \tb +\rho_i, \tau_r < \tb +\rho_j \ \ \forall j \in I\setminus \{i\} \\ 
\{0 , \tau_{i}(1+\varepsilon) \} \ &\mbox{if } \tau_{i} = \tb+\rho_i \ \exists j \in I \setminus \{i\} \  \mbox{s.t.} \  \tau_r = \tb+\rho_j \\ 
(1+\varepsilon)\tau_{i} \qquad \  &\mbox{if } \tau_i < \tb +\rho_i \ \exists j \in I \setminus \{i\} \  \mbox{s.t.} \  \tau_r = \tb+\rho_j\end{array} \right. \label{eqn:gi_perturbed} .
\end{equation}
}
\IfConf{ %In Journal Version
\begin{figure}
\centering
\setlength{\unitlength}{0.0165cm}
\subfigure[Initial condition $\tau(0,0) = {[0.1,0.2]^\top}$.]{
\label{fig:ResetBumpPertub_di_equal_t1t2plot_di=[0.1,0.1]}
\includegraphics[width = .2\textwidth,trim = 25mm 0mm 20mm 10mm, clip]{Figures/ResetBumpPertub_di_equal_t1t2plot_di=[0.1,0.1].eps}
\put(-120,4){\scalebox{.7}{ $\tau_1$}}
\put(-240,105){\scalebox{.7}{ $\tau_2$}}
}\hspace{.2cm}
\subfigure[Distance to the set $\wt\A$ for 10 solutions to $\HS_2$ with initial conditions randomly chosen from $C$.]{
\label{ResetBumpPertub_di_equal_timeplot_di=[0.1,0.1]}
\includegraphics[width = .22\textwidth,trim = 20mm 00mm 15mm 6mm, clip]{Figures/ResetBumpPertub_di_equal_timeplot_di=[0.1,0.1].eps}
\put(-146,0){\scalebox{.7}{$t$ [seconds]}}
\put(-265,70){\scalebox{.7}{$|\tau|_{\wt\A}$}}
}
\caption{Solutions to the hybrid system with perturbed jump conditions. 
Figures (a) and (b) have the perturbation given in Section~\ref{sec:jumpperturbs1} with $\wt\rho_1 = \wt\rho_2 = 0.1$.}
\label{figs:di_equal_di=0.2}
\end{figure}}
{ %In Tech Report
\begin{figure}
\centering
\setlength{\unitlength}{0.023cm}
\subfigure[Solution to $\HS_2$ on the $(\tau_1,\tau_2)$-plane with initial condition $\tau(0,0) = {[1.6,2.1]^\top}$.]{\label{figs:di_equal_di=0.2:tau1tau2}
\psfragfig*[width = .4\textwidth]{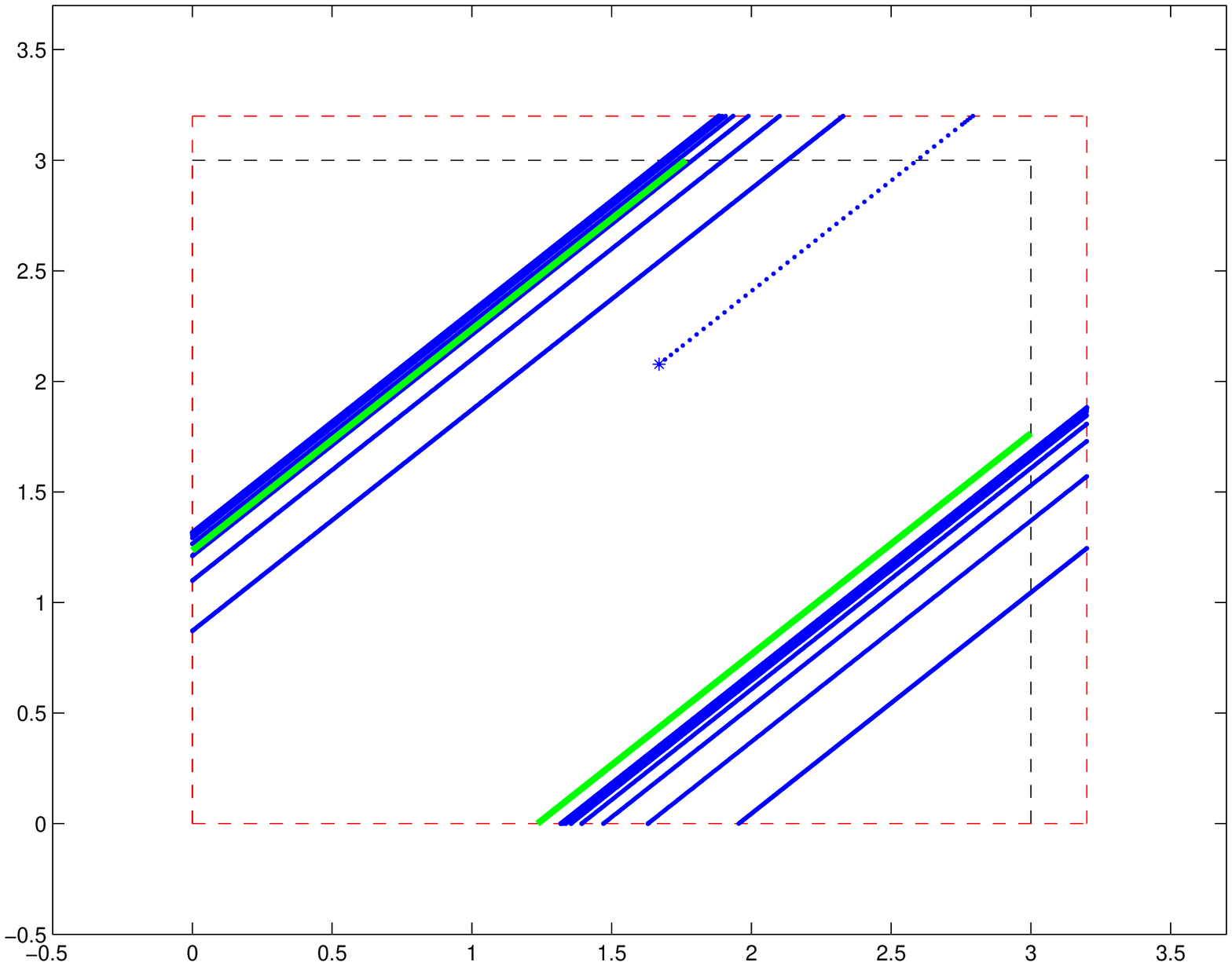}
\put(-140,6){\scalebox{.7}{ $\tau_1$}}
\put(-270,105){\scalebox{.7}{ $\tau_2$}}
%\put(-120,4){\small $\tau_1$}
%\put(-220,85){\small $\tau_2$}
}\hspace{.2cm}
\subfigure[Distance to the set $\wt\A$ for 10 solutions with initial conditions randomly chosen from ${[0,\tb + \rho_1] \times [0,\tb + \rho_2]}$. The solutions have a distance that converges to a steady state value of approximately $0.08$ at approximately $28$ seconds of flow time.]{\label{figs:di_equal_di=0.2:timetraj}
\psfragfig*[width = .4\textwidth]{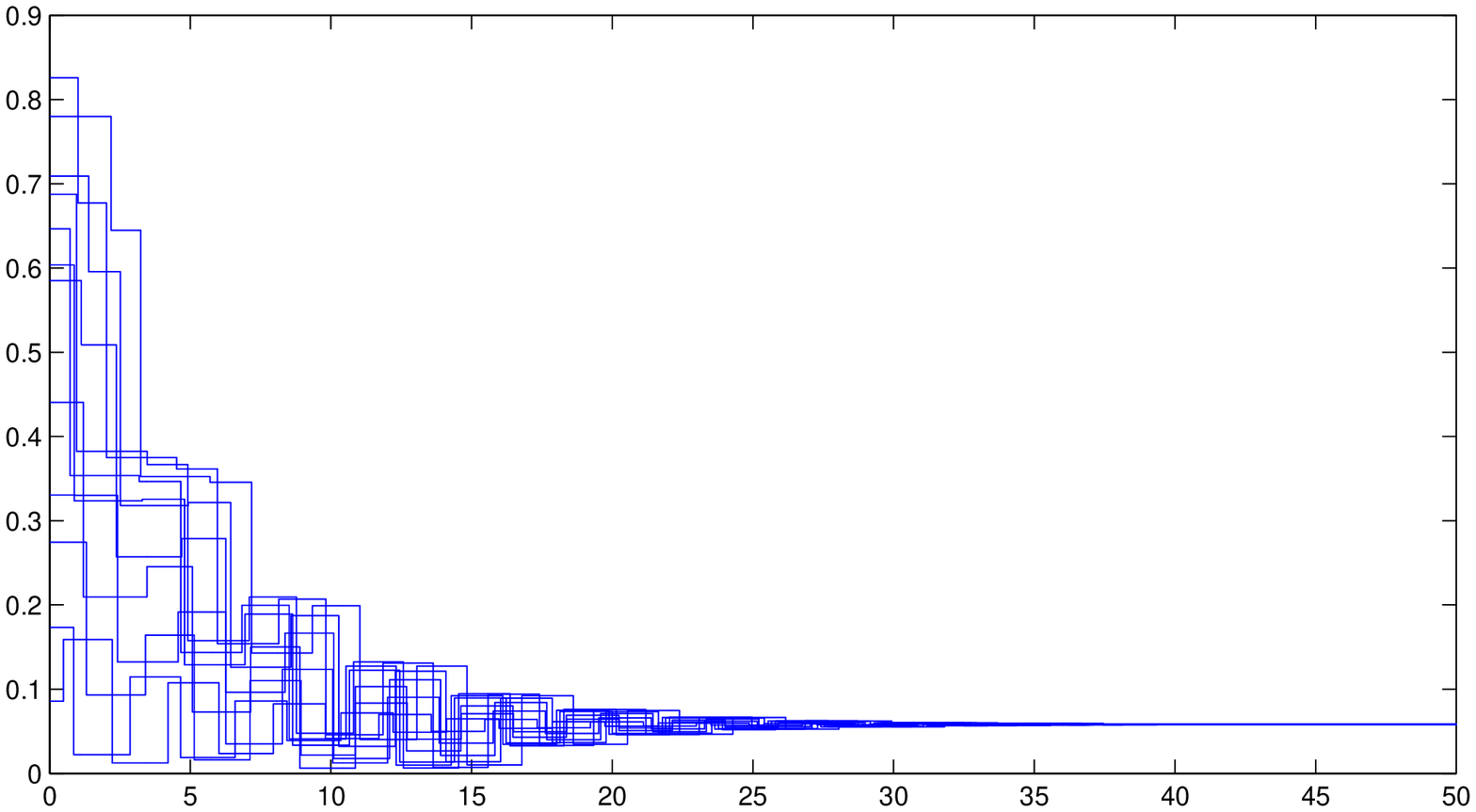}
\put(-160,0){\scalebox{.7}{ $t$ [seconds]}}
\put(-285,70){\scalebox{.7}{ $|\tau|_{\wt\A}$}}
}
\caption{Solutions to the hybrid system with perturbed threshold, namely, with $D_\rho = \{\tau : \exists i \in \{1,2\} \mbox{ s.t. } \tau_i = \tb + \rho_i\}$ for $\rho_1 = \rho_2 = 0.2$.}
\label{figs:di_equal_di=0.2}
\end{figure}}
\NotForConf{
This case of perturbations is an example of Theorem~\ref{thm:robustofAS} with $\rho$ affecting only the jump map. The trajectories of the perturbed version of $\HS_N$ will converge to a region around the set $\wt\A$. 
Simulations are presented in Figures~\ref{figs:di_equal_di=0.2} and \ref{figs:di_notequal_di} for $N = 2$, $\omega = 1$, $\tb = 3$, and $\e = -0.3$. 

Figure~\ref{figs:di_equal_di=0.2} shows numerical results for the case when each $\rho_i$ are equal, i.e., $\rho_1 = \rho_2 = 0.02$. Figure~\ref{figs:di_equal_di=0.2:tau1tau2} shows a solution (solid blue) to the perturbed $\HS_2$ with initial condition $\tau(0,0) = [1.6,2.1]^\top$ (blue asterisk) on the $(\tau_1,\tau_2)$-plane with $C$ (black dashed line), the perturbed flow set $C_\rho$ (red dashed line), and the desynchronization set $\A$ (solid green line). From this figure, notice that the solution extends beyond the set $C$ and resets at $\tau_i = 3+0.2$. The solution converges to a region near the desynchronization set, as Theorem~\ref{thm:robustofAS} guarantees.
To further clarify the response of $\HS_2$ to this type of perturbation, Figure~\ref{figs:di_equal_di=0.2:timetraj} shows the distance to the set $\wt\A$ for 10 solutions with randomly chosen initial conditions $\tau(0,0) \in C_\rho$. Notice that for the initial conditions chosen, all solutions converge to a distance of approximately $0.08$ by $t \approx 28$ seconds. 

Figure~\ref{figs:di_notequal_di} shows the numerical results for the case when each $\rho_i$ are not equal, i.e., $\rho_1 \neq \rho_2$.  Figure~\ref{fig:di_NotEqual_timetraj_di=[0.5,0.4]} shows 10 solutions from  random initial conditions $\tau(0,0) \in C_\rho$ with $\rho_1 = 0.5$ and $\rho_2 = 0.4$. For this case, the solutions converge to a region near $\wt\A$, in that, $|\tau(t,j)|_{\wt\A} \leq 0.22$ after approximately $0.28$ seconds of flow time.
Figure~\ref{di_NotEqual_timetraj_di=[0.02,0.01]} shows 15 solutions when $\rho_1 = 0.02$ and $\rho_2 = 0.01$. For this set of simulations, the solutions converge to a distance of approximately $0.04$ around $\wt\A$ after approximately 26 seconds of flow time. These simulations validate Theorem~\ref{thm:robustofAS} with $\rho$ affecting only the jump map, verifying that the smaller the size of the perturbation the smaller the steady-state value of the distance to $\wt\A$.
}
\NotForConf{ %In Tech Report Version
\begin{figure}
\setlength{\unitlength}{0.023cm}
\centering
\subfigure[Distance to the set $\wt\A$ for 10 solutions with random initial conditions ${\tau(0,0) \in [0,\tb + \rho_1] \times [0,\tb + \rho_2]}$ with $\rho_1 = 0.5$ and $\rho_2 = 0.4$.]{\label{fig:di_NotEqual_timetraj_di=[0.5,0.4]}
\includegraphics[width = .4\textwidth]{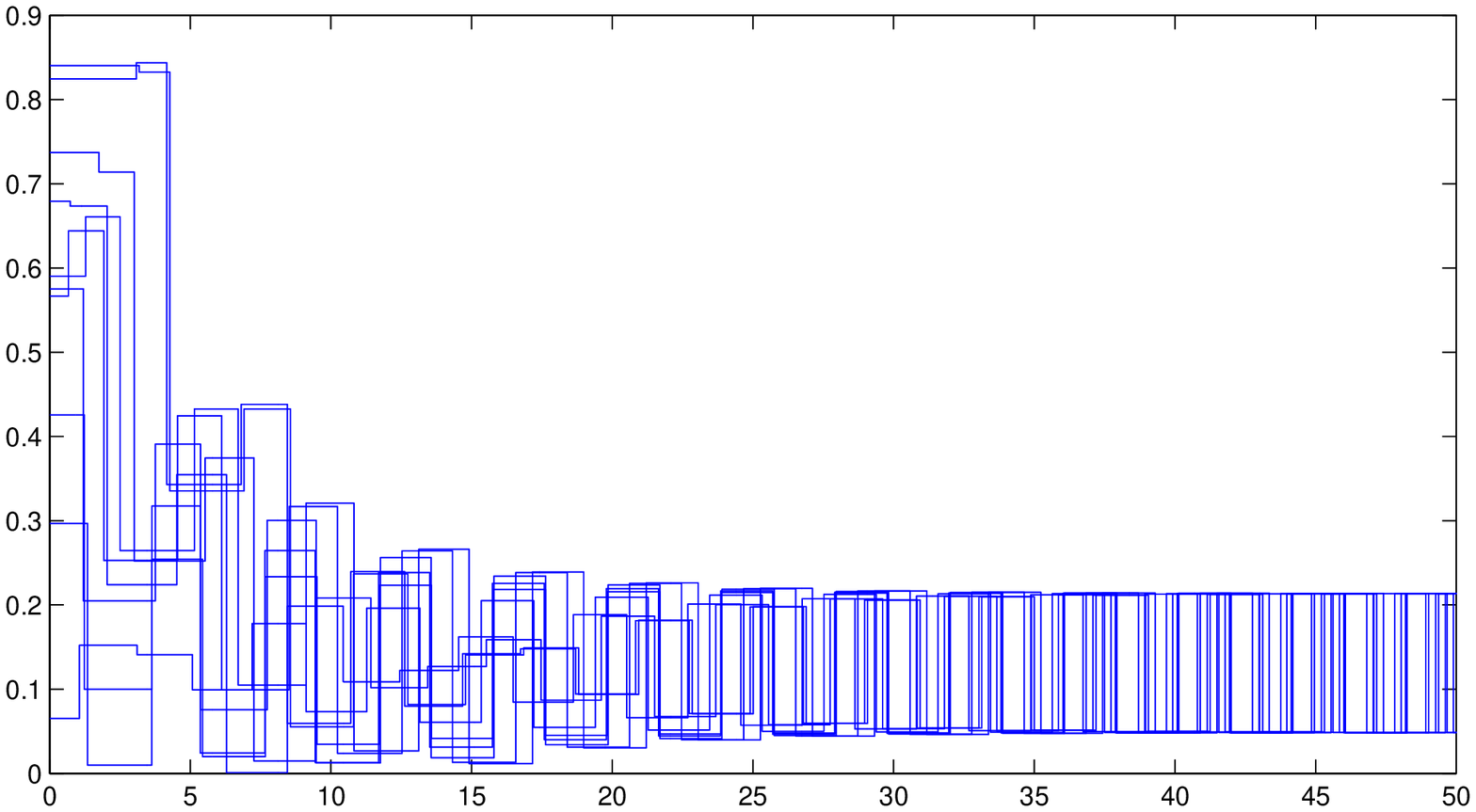}
\put(-160,0){\scalebox{.7}{ $t$ [seconds]}}
\put(-285,70){\scalebox{.7}{ $|\tau|_{\wt\A}$}}
} \hspace{.2cm}
\subfigure[Distance to the set $\wt\A$ for 15 solutions with random initial conditions ${\tau(0,0) \in [0,\tb + \rho_1] \times [0,\tb + \rho_2]}$ with $\rho_1 = 0.02$ and $\rho_2 = 0.01$.]{\label{di_NotEqual_timetraj_di=[0.02,0.01]}
\includegraphics[width = .4\textwidth]{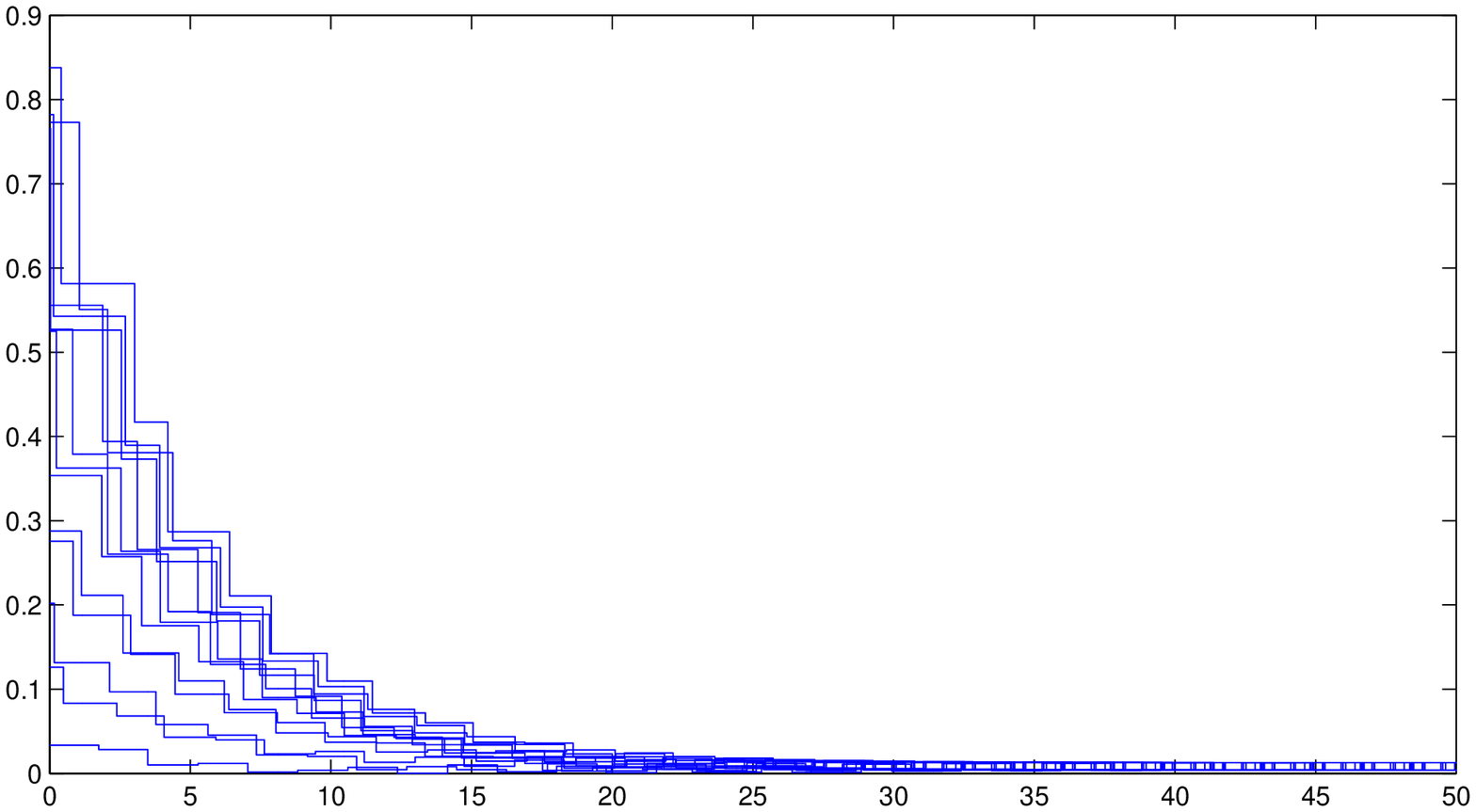}
\put(-160,0){\scalebox{.7}{ $t$ [seconds]}}
\put(-285,70){\scalebox{.7}{ $|\tau|_{\wt\A}$}}
}
\caption{Numerical simulations of the perturbed version of $\HS_2$ with  jump set given by $D_\rho = \{\tau : \exists i \in \{1,2\} \mbox{ s.t. } \tau_i = \tb + \rho_i\}$ for different values of $\rho_i$.}
\label{figs:di_notequal_di}
\end{figure}}
 
\NotForConf{
$\bullet$ {\bf Perturbations on the reset component of the jump map:}
Under the effect of the perturbations considered in this case, instead of 
reseting $\tau_i$ to zero, the perturbed jump resets $\tau_i$ to a value $\rho_i \in \reals_{\geq 0}$, for each $i \in I$. 
The perturbed hybrid system has the following data:
\IfConf{$f(\tau) = \omega \one$ for all $\tau \in C_\rho := C$;}{\begin{align*}
f(\tau) = \omega \one \qquad \forall \tau \in C_\rho := C
\end{align*}}
and 
\IfConf{$G_{\rho}(\tau) = [g_{\rho_1}(\tau), \ldots, g_{\rho_1}(\tau)]^\top$ for all $\tau \in D_\rho =  D$}{\begin{align*}
G_{\rho}(\tau) = [g_{\rho_1}(\tau), \ldots, g_{\rho_1}(\tau)]^\top \qquad \forall \tau \in D_\rho =  D
\end{align*}}
where, for each $i \in I$, the perturbed jump map is given by 
\begin{equation}
g_i (\tau) = \left\{ \begin{array}{l} \rho_i \qquad \qquad \ \ \ \ \mbox{if } \tau_{i} = \tb , \tau_r < \tb \ \ \forall j \in I\setminus \{i\} \\ 
\{\rho_i , \tau_{i}(1+\varepsilon) \} \ \mbox{if } \tau_{i} = \tb\ \exists j \in I \setminus \{i\} \  \mbox{s.t.} \  \tau_r = \tb \\ 
(1+\varepsilon)\tau_{i} \qquad \  \mbox{if } \tau_i < \tb \ \exists j \in I \setminus \{i\} \  \mbox{s.t.} \  \tau_r = \tb\end{array} \right. .\label{eqn:gi_resetperturbation}
\end{equation}

This case of perturbations exemplifies Theorem~\ref{thm:robustofAS} 
with $\rho$ affecting only the jump map of $\HS_N$. Figures~\ref{figs:ResetPerturb_di_equal_di=0.2} 
and \ref{figs:ResetPertub_di_notequal_di} show several 
simulations to this perturbation of $\HS_N$. All of the 
simulations in this section use parameters 
$\omega = 1$, $\tb = 3$, $\e = -0.3$, and $N = 2$. 

The first case of the perturbed jump map $G_\rho$ considered is for $\rho_1 = \rho_2 = 0.02$. Figure~\ref{ResetPerturb_di_equal_t1t2plot_di=0.2} shows a solution to the perturbed $\HS_2$ from the initial condition $\tau(0,0) = [2.4,2.3]^\top$ on the $(\tau_1,\tau_2)$-plane. Notice that for $\tau \in D$ such that $\tau_i = \tb$ the jump map resets $\tau_i$ to $\rho_i$ (red dashed line) and not to $0$ as in the unperturbed case. 
The solution for this case approaches a region around $\wt\A$, as Theorem~\ref{thm:robustofAS} guarantees. Figure~\ref{ResetPerturb_di_equal_timeplot_di=0.2} shows the distance to the set $\wt\A$ over time for 10 solutions of the perturbed system $\HS_2$ with initial conditions $\tau(0,0) \in P_2\setminus \X_{2}$. This figure shows that solutions approach a distance of about $0.12$ after 25 seconds. 

\begin{figure}
\setlength{\unitlength}{0.023cm}
\centering
\subfigure[Solution to $\HS_2$ on the $(\tau_1,\tau_2)$-plane with initial condition $\tau(0,0) = {[2.4,2.3]^\top}$.]{
\label{ResetPerturb_di_equal_t1t2plot_di=0.2}
\includegraphics[width = .4\textwidth]{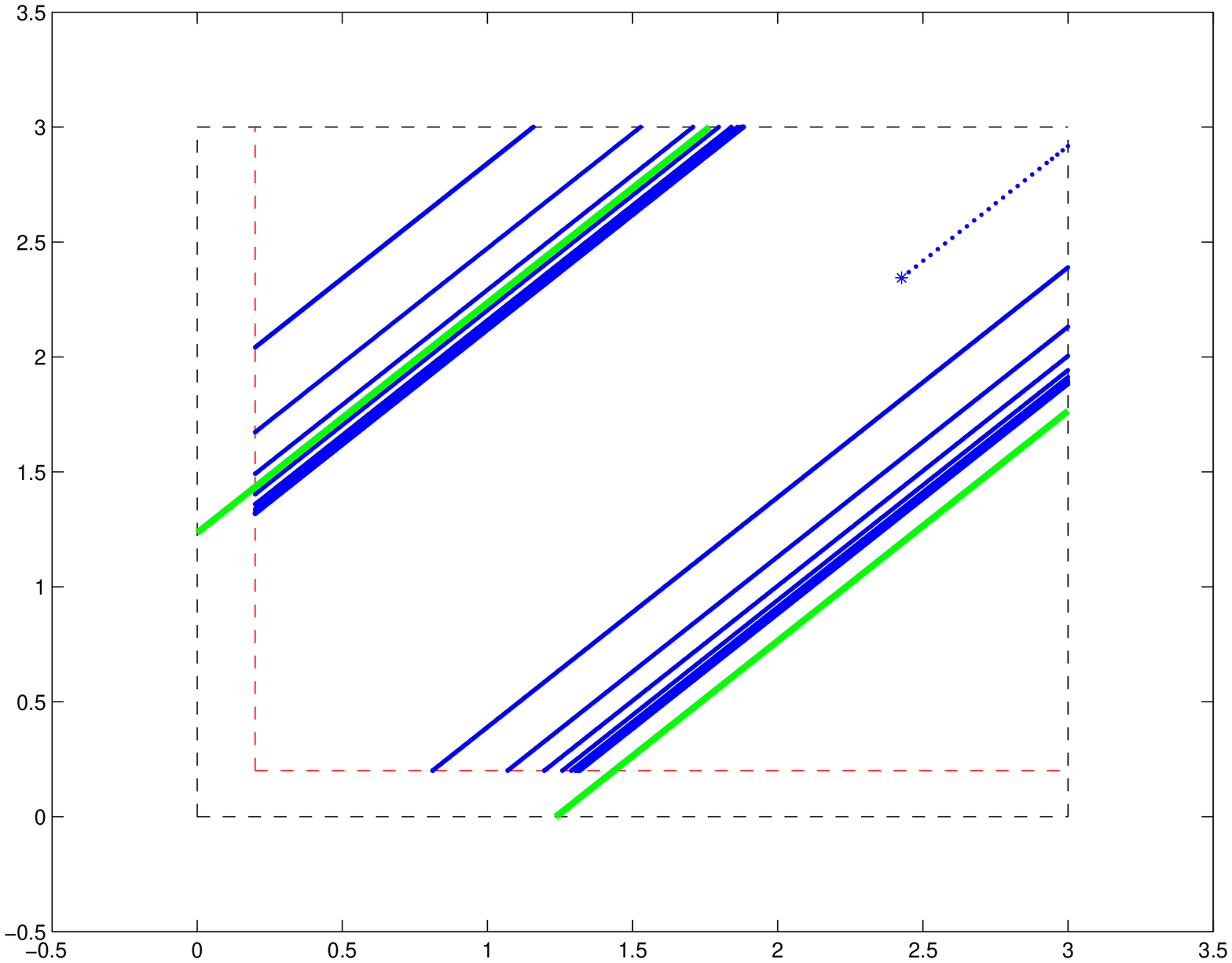}
\put(-140,6){\scalebox{.7}{ $\tau_1$}}
\put(-270,105){\scalebox{.7}{ $\tau_2$}}
}\hspace{.2cm}
\subfigure[Distance to the set $\wt\A$ for 10 solutions to $\HS_2$ with initial conditions randomly chosen from ${C}$. Most of the solutions have a distance that converges to a steady state value of approximately 0.12 at about $25$ seconds]{\label{ResetPerturb_di_equal_timeplot_di=0.2}
\includegraphics[width = .4\textwidth]{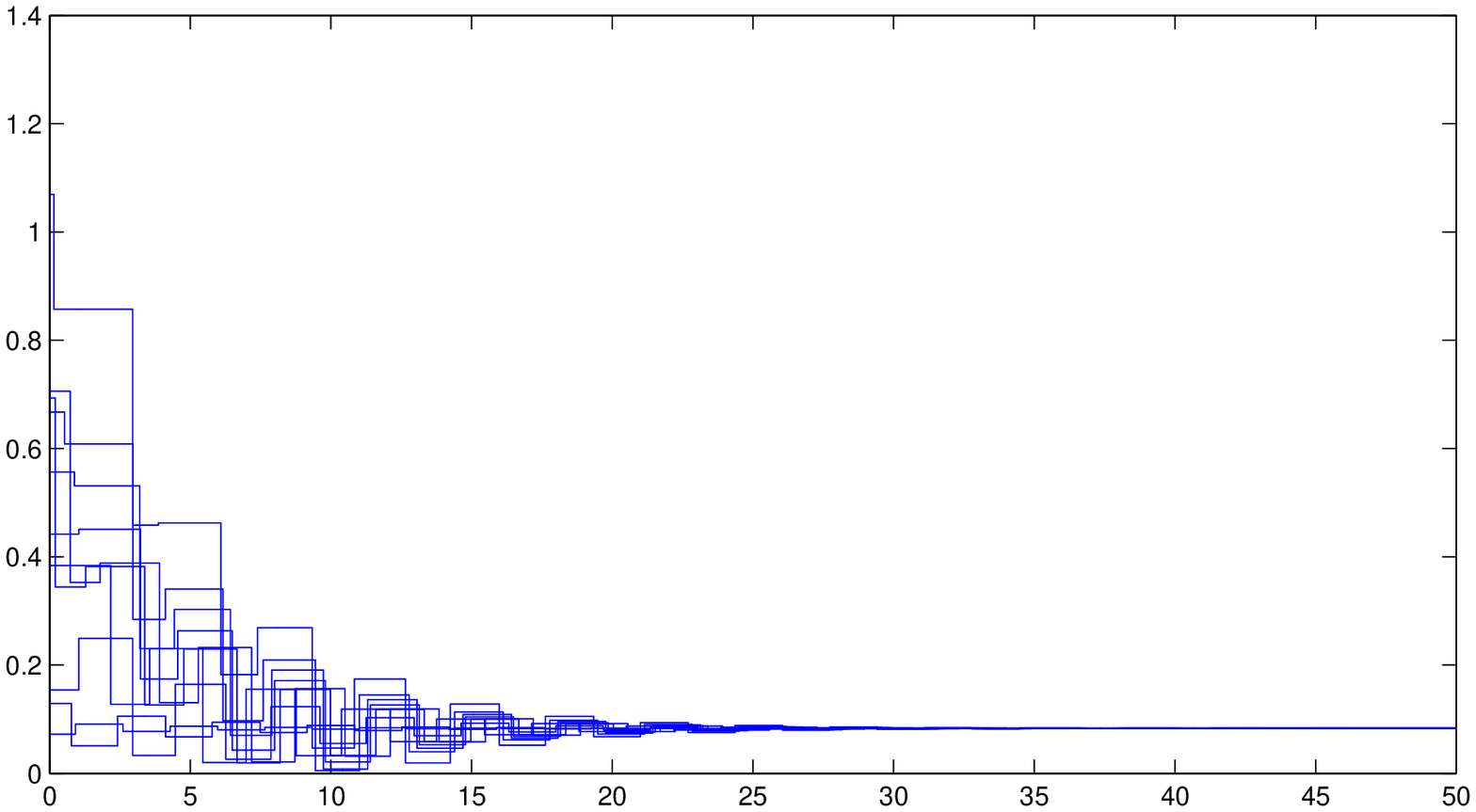}
\put(-160,0){\scalebox{.7}{ $t$ [seconds]}}
\put(-285,70){\scalebox{.7}{ $|\tau|_{\wt\A}$}}
}
\caption{Solutions to the hybrid system $\HS_2$ with the perturbed jump in \eqref{eqn:gi_resetperturbation} map with $\rho_1 = \rho_2 = 0.2$.}
\label{figs:ResetPerturb_di_equal_di=0.2}
\end{figure}

Now, consider the case where $\rho_1 \neq \rho_2$. Figure~\ref{figs:ResetPertub_di_notequal_di} shows the distance to $\wt\A$ for two sets of solutions with different values for $\rho_1$ and $\rho_2$. More specifically, Figure~\ref{ResetPertub_di_NotEqual_timeplot_di=[0.15,0.25]} shows the case of $\rho_1 = 0.15$ and $\rho_2 = 0.25$. For this case, it can be seen that the solutions converge after $\approx 28$ seconds of flow time and, after that time, satisfy $|\tau(t,j)|_{\wt\A} \leq 0.25$.  Figure~\ref{ResetPertub_di_NotEqual_timeplot_di=[0.01,0.02]} shows the case of $\rho_1 = 0.02$ and $\rho_2 = 0.01$. For this case, this figure shows that, after $\approx 28$ seconds of flow time, the solutions satisfy $|\tau(t,j)|_{\wt\A} \leq 0.04$. These simulations validate Theorem~\ref{thm:robustofAS} with $\rho$ affecting only the jump map, verifying that the smaller the size of the perturbation the smaller the steady-state value of the distance to $\wt\A$. 

\begin{figure}
\setlength{\unitlength}{0.023cm}
\centering
\subfigure[Distance to the set $\wt\A$ for 10 solutions with random initial conditions ${\tau(0,0) \in C}$ with $\rho_1 = 0.15$ and $\rho_2 = 0.25$.]{\label{ResetPertub_di_NotEqual_timeplot_di=[0.15,0.25]}
\includegraphics[width = .4\textwidth]{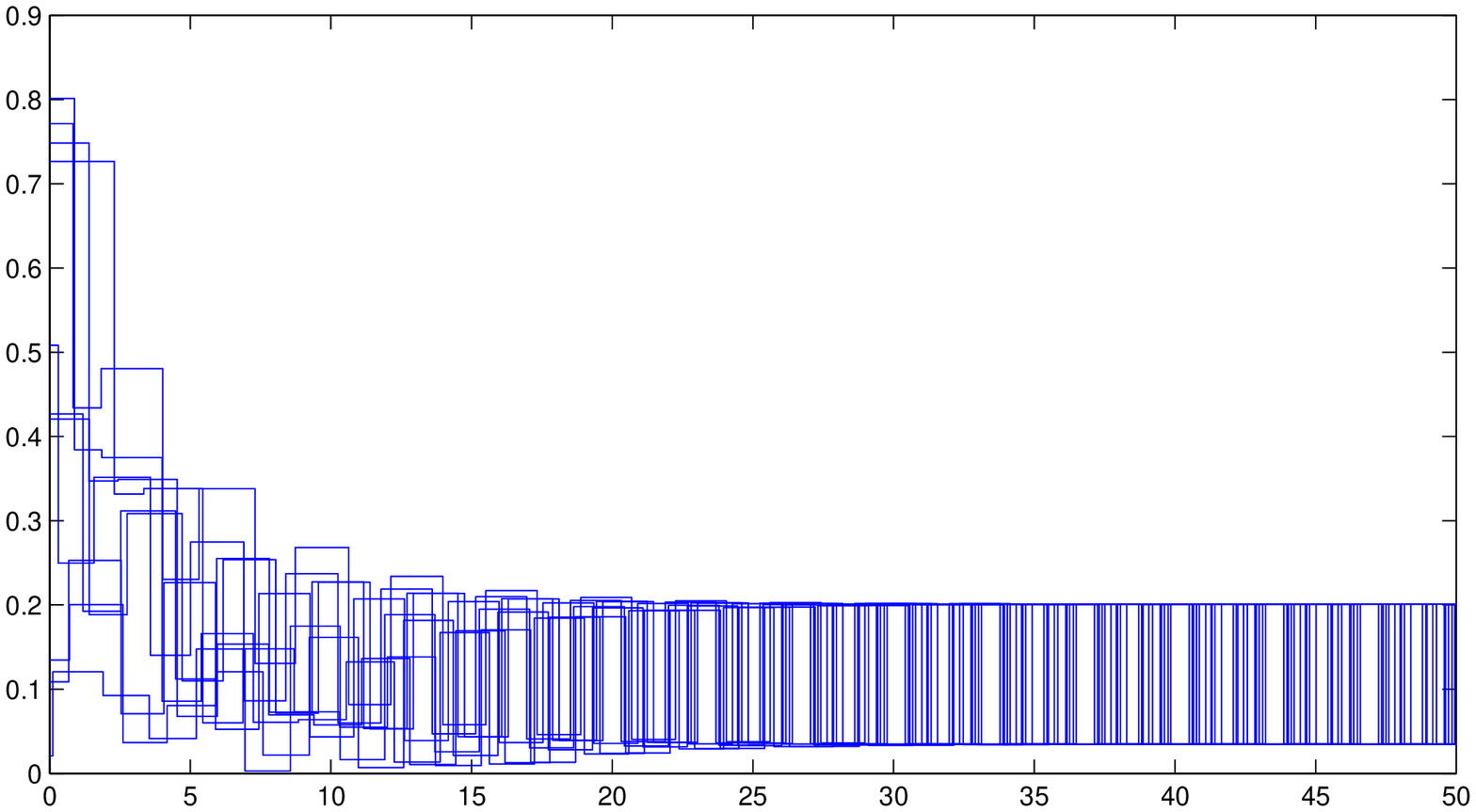}
\put(-160,0){\scalebox{.7}{ $t$ [seconds]}}
\put(-285,70){\scalebox{.7}{ $|\tau|_{\wt\A}$}}
}\hspace{.2cm}
\subfigure[Distance to the set $\wt\A$ for 10 solutions with random initial conditions ${\tau(0,0) \in C}$ with $\rho_1 = 0.02$ and $\rho_2 = 0.01$.]{\label{ResetPertub_di_NotEqual_timeplot_di=[0.01,0.02]}
\includegraphics[width = .4\textwidth]{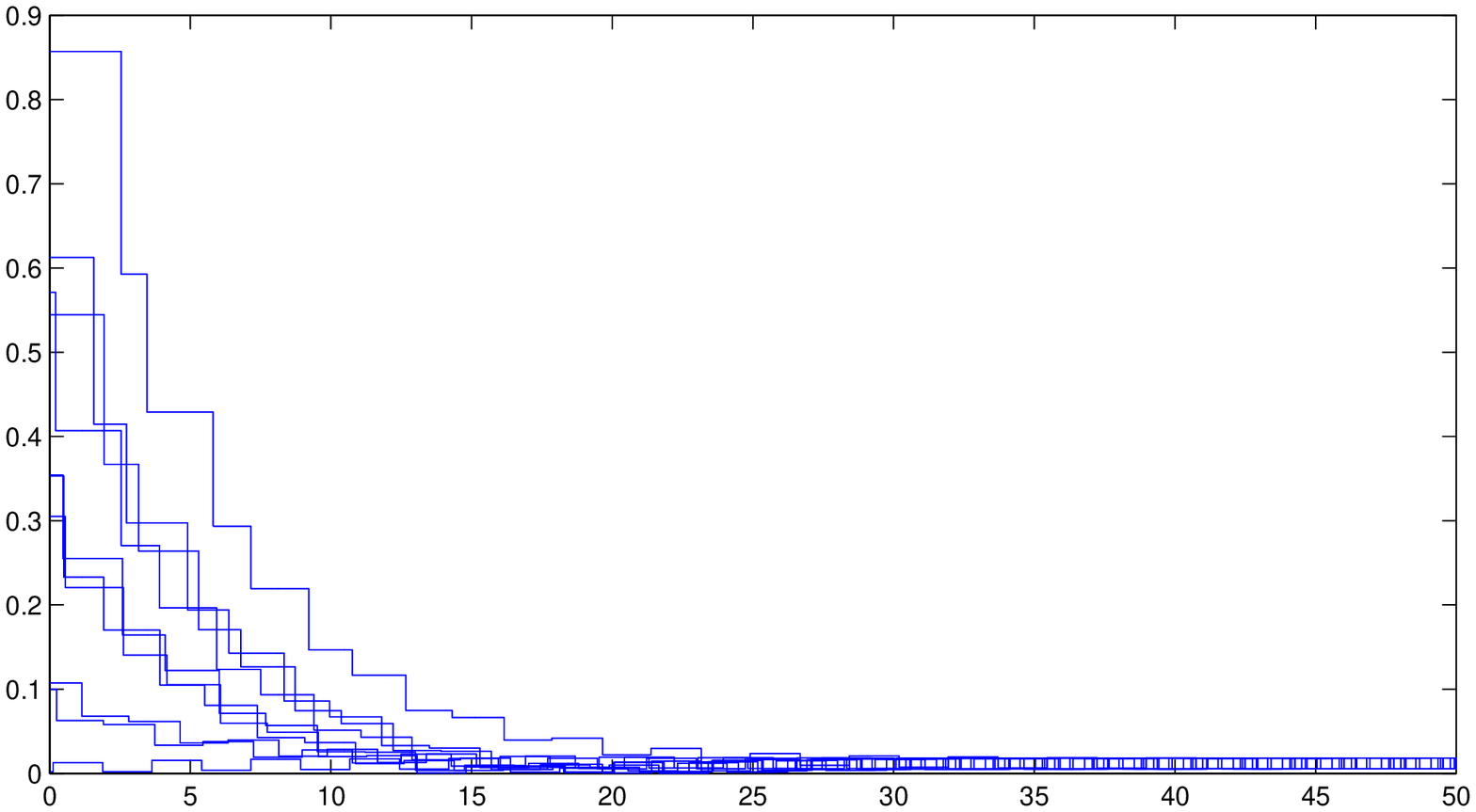}
\put(-160,0){\scalebox{.7}{ $t$ [seconds]}}
\put(-285,70){\scalebox{.7}{ $|\tau|_{\wt\A}$}}
}
\caption{Solutions to the hybrid system $\HS_2$ with the perturbed jump map with $\rho_1 \neq \rho_2$.}
\label{figs:ResetPertub_di_notequal_di}
\end{figure}
}
\IfConf{In this section, we consider perturbations on the ``bump'' component of the jump map. More precisely,}{$\bullet$ {\bf Perturbations on the ``bump'' component of the jump map:}
In this case,} the component $(1+\e)\tau_i$ of the jump map is perturbed, namely, we use $\tau_i^+ = (1+\e) \tau_i + \rho_i(\tau_i)$, where $\rho_i : \reals_{\geq 0} \to P_N \setminus \X$ is a continuous function. The perturbed jump map $G_{\rho}$ has components $g_{\rho i}$ that are given as $g_{i}$ in \eqref{eqn:gi} but with $\tau_{i}(1+\e) + \rho_{i}(\tau_{i})$ replacing $\tau_{i}(1+\e)$.

Consider the case $\rho_i(\tau_{i}) = \wt\rho_i \tau_i$ with $\wt\rho_i \in (0,|\e|)$ and let $\wt\e_i = \e + \wt\rho_i \in (-1,0)$. Then $\tau_i^+$ reduces to $\tau_i^+ = (1+\wt\e_i)\tau_i$ and the jump map $g_{\rho i}$ is given by \eqref{eqn:gi} with $\wt\e_{i}$ in place of $\e$. This type of perturbation is used to verify Theorem~\ref{thm:robustofAS} with $\rho$ affecting only the ``bump'' portion of the jump map. \IfConf{Figures~\ref{figs:di_equal_di=0.2}}{Figures~\ref{figs:ResetBumpPerturb_di_equal_di=0.2}} and \ref{figs:ResetBumpPerturb_di_notequal_di} show simulations to $\HS_N$ with the parameters $\omega = 1$, $\tb = 3$, $\e = -0.3$, and $N = 2$.

\NotForConf{\begin{figure}
\setlength{\unitlength}{0.023cm}
\centering
\subfigure[Solution $\HS_2$ on the $(\tau_1,\tau_2)$-plane with initial condition $\tau(0,0) = {[0.1,0.2]^\top}$.]{
\includegraphics[width = .4\textwidth]{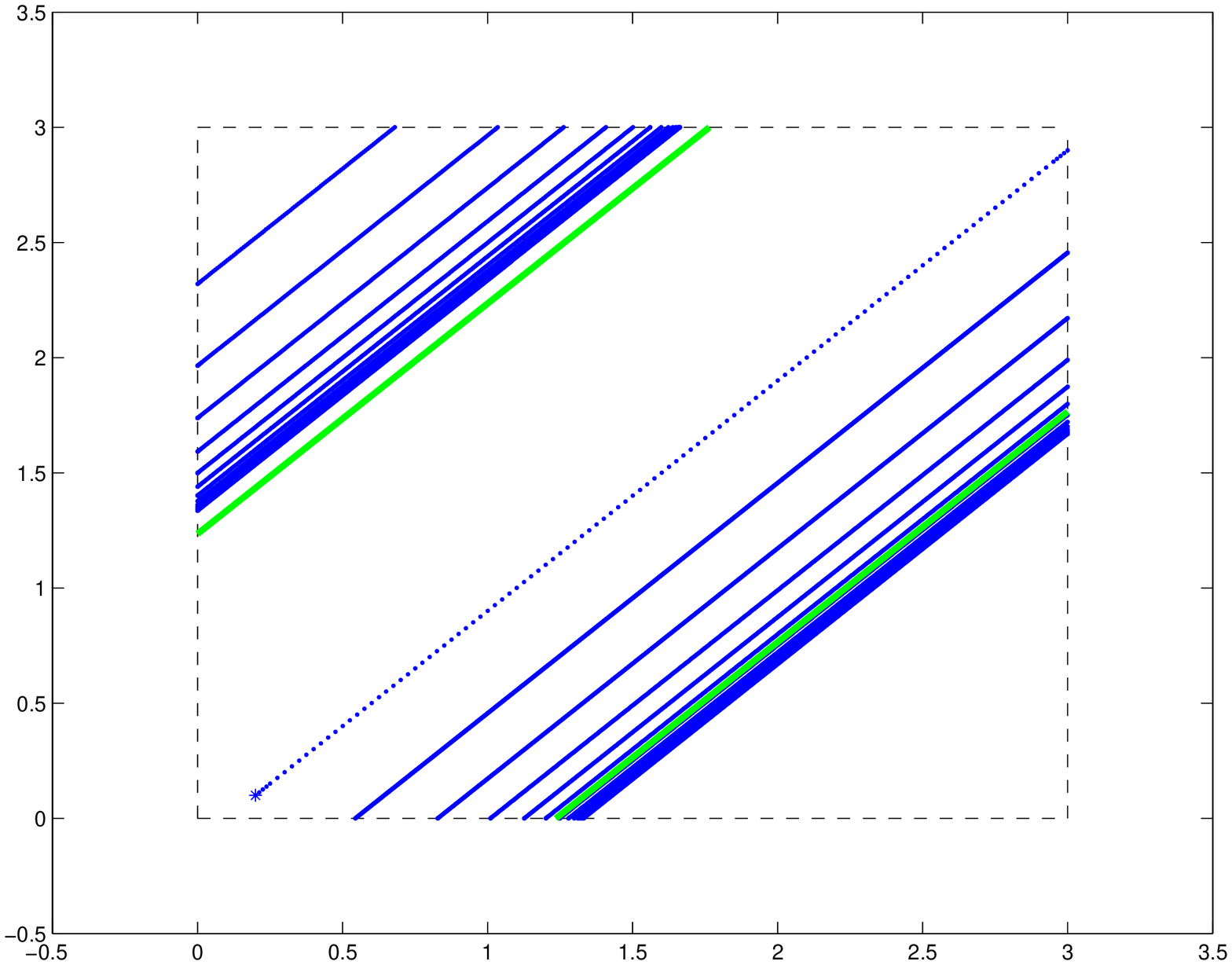}
\put(-140,6){\scalebox{.7}{ $\tau_1$}}
\put(-270,105){\scalebox{.7}{ $\tau_2$}}
}\hspace{.2cm}
\subfigure[Distance to the set $\wt\A$ for 10 solutions to $\HS_2$ with initial conditions randomly chosen from $C$. These solutions have a distance that converges to a steady state value of approximately $0.08$ at about 45 seconds.]{
\label{ResetBumpPertub_di_equal_timeplot_di=[0.1,0.1]}
\includegraphics[width = .4\textwidth]{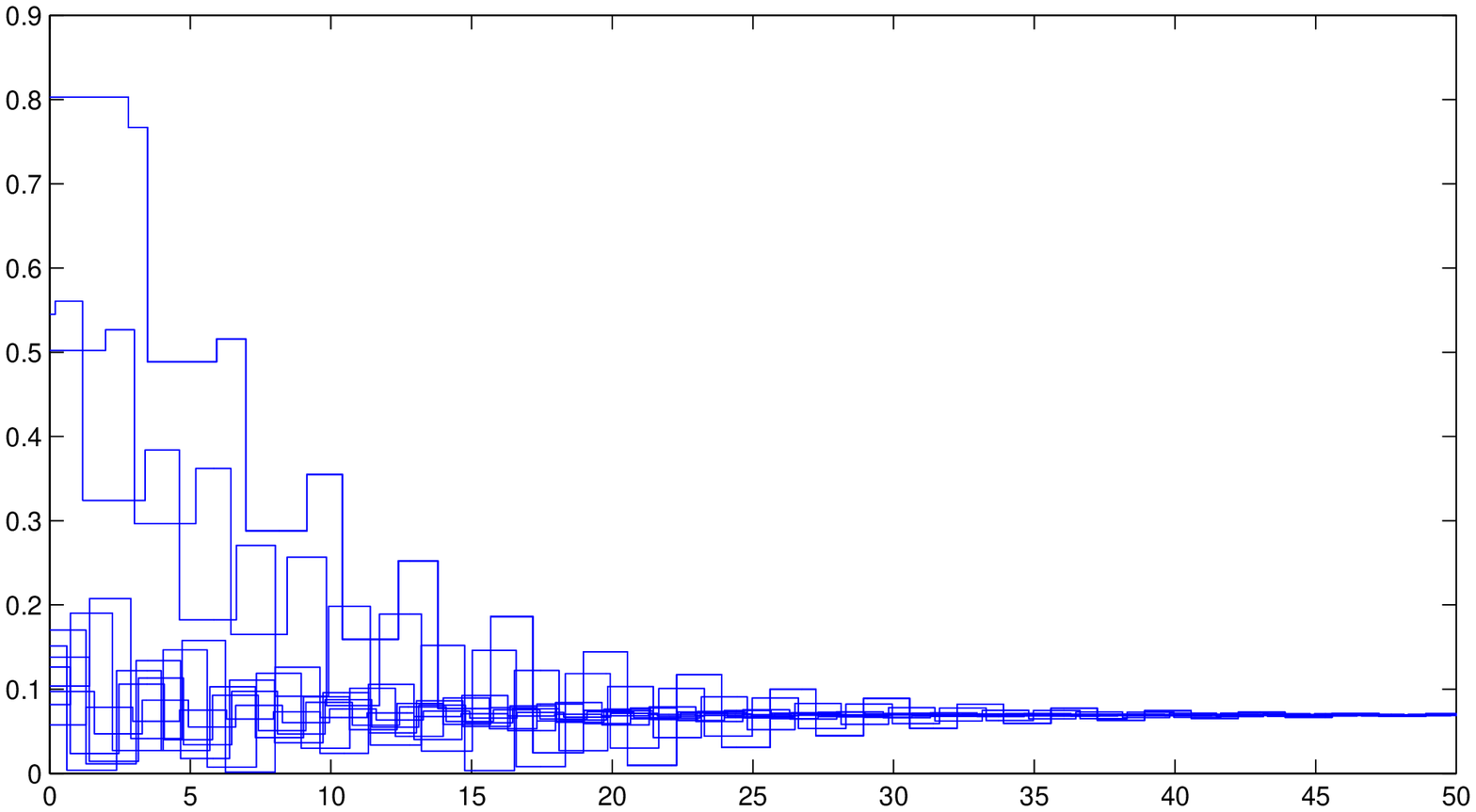}
\put(-160,0){\scalebox{.7}{ $t$ [seconds]}}
\put(-285,70){\scalebox{.7}{ $|\tau|_{\wt\A}$}}
}
\caption{Solutions to the hybrid system with perturbed ``bump'' on the jump map, with $\wt\rho_1 = \wt\rho_2 = 0.1$.}
\label{figs:ResetBumpPerturb_di_equal_di=0.2}
\end{figure} }

Consider the case of $\HS_2$ with $G_\rho$ when  $\wt\rho_1 = \wt\rho_2 = 0.1$, leading to $\wt\e_1 = \wt\e_2$ = 0.2.  \IfConf{Figure~\ref{fig:ResetBumpPertub_di_equal_t1t2plot_di=[0.1,0.1]}}{Figure~\ref{figs:ResetBumpPerturb_di_equal_di=0.2}} shows a solution on the $(\tau_1,\tau_2)$-plane for this case with initial condition $\tau(0,0) = [0.1,0.2]^\top$. Notice that the solution approaches a region around $\A$ (green line), as Theorem~\ref{thm:robustofAS} guarantees. Figure~\ref{ResetBumpPertub_di_equal_timeplot_di=[0.1,0.1]} shows the distance to the set $\wt\A$ over time for 10 solutions with initial conditions $\tau(0,0) \in C$. It shows that solutions approach a distance to $\wt\A$ of $\approx0.09$ after $\approx40$ seconds of flow time.

\IfConf{
\begin{figure}
\setlength{\unitlength}{0.0165cm}
\centering
\subfigure[Distance to the set $\wt\A$ for 10 solutions with random initial conditions ${\tau(0,0) \in C}$ with $\wt\rho_1 = 0.15$ and $\wt\rho_2 = 0.1$.]{
\label{ResetBumpPertub_di_NotEqual_timeplot_di=[0.15,0.1]}
\includegraphics[width = .22\textwidth,trim = 20mm 00mm 15mm 6mm, clip]{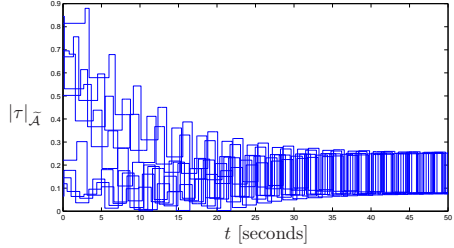}
\put(-146,0){\scalebox{.7}{$t$ [seconds]}}
\put(-265,70){\scalebox{.7}{$|\tau|_{\wt\A}$}}
}\hspace{.2cm}
\subfigure[Distance to the set $\wt\A$ for 10 solutions with random initial conditions ${\tau(0,0) \in C}$ with $\wt\rho_1 = 0.02$ and $\wt\rho_2 = 0.01$.]{
\label{fig:ResetBumpPertub_di_NotEqual_timeplot_di=[0.02,0.01]}
\includegraphics[width = .22\textwidth,trim = 20mm 00mm 15mm 8mm, clip]{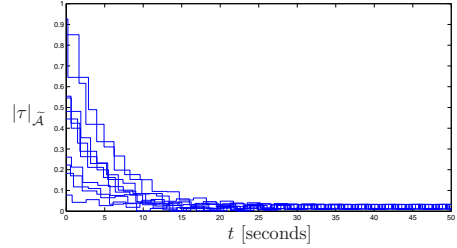}
\put(-146,0){\scalebox{.7}{$t$ [seconds]}}
\put(-265,70){\scalebox{.7}{$|\tau|_{\wt\A}$}}
}
\caption{Numerical simulations of the perturbed version of $\HS_2$ with the perturbed ``bump'' on the jump map with
 $\wt\rho_1 \neq \wt\rho_2$.}
\label{figs:ResetBumpPerturb_di_notequal_di}
\end{figure}}
{\begin{figure}
\setlength{\unitlength}{0.023cm}
\centering
\subfigure[Distance to the set $\wt\A$ for 10 solutions with random initial conditions ${\tau(0,0) \in C}$ with $\wt\rho_1 = 0.15$ and $\wt\rho_2 = 0.1$.]{
\label{ResetBumpPertub_di_NotEqual_timeplot_di=[0.15,0.1]}
\includegraphics[width = .4\textwidth]{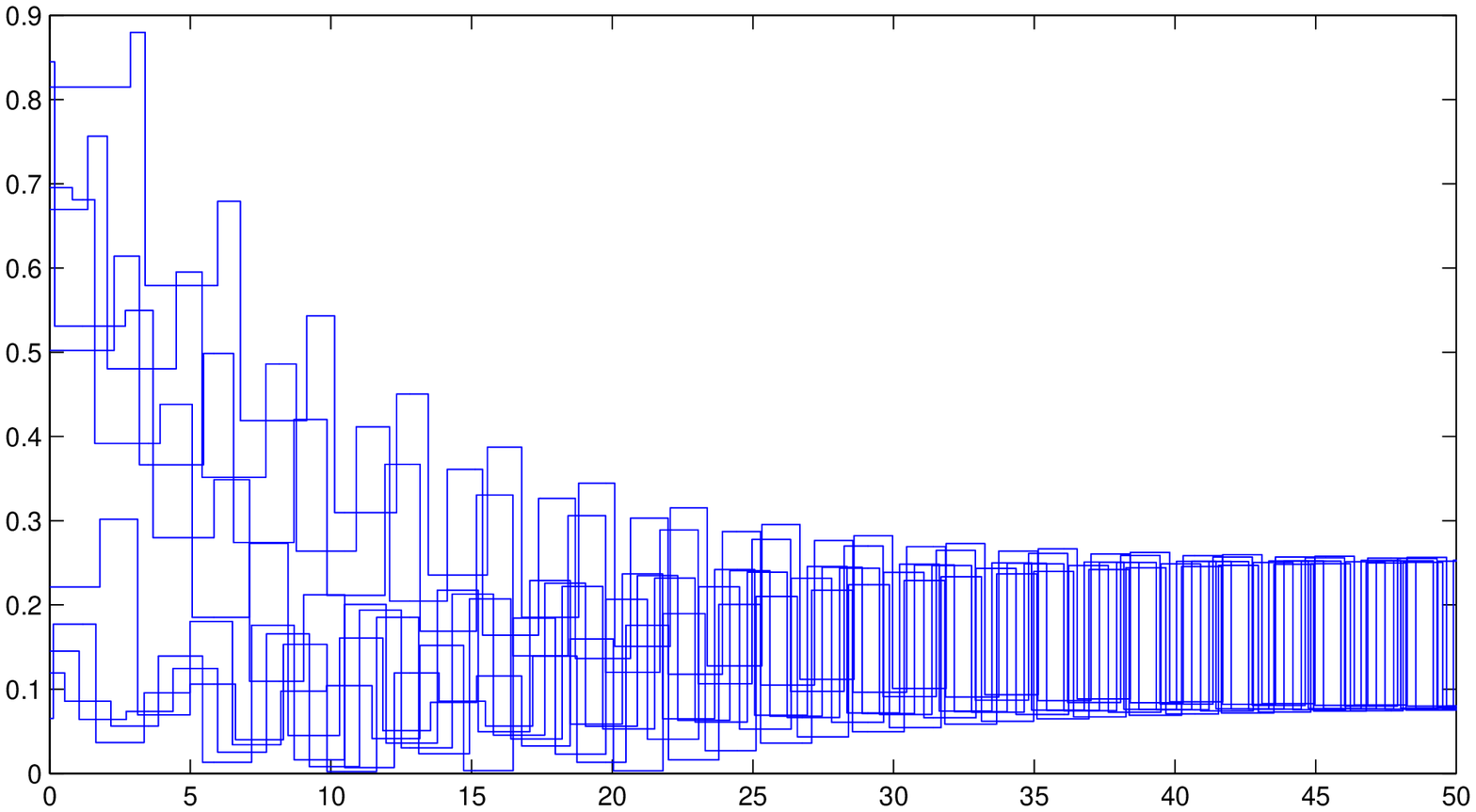}
\put(-160,0){\scalebox{.7}{ $t$ [seconds]}}
\put(-285,70){\scalebox{.7}{ $|\tau|_{\wt\A}$}}
}\hspace{.2cm}
\subfigure[Distance to the set $\wt\A$ for 10 solutions with random initial conditions ${\tau(0,0) \in C}$ with $\wt\rho_1 = 0.02$ and $\wt\rho_2 = 0.01$.]{
\label{fig:ResetBumpPertub_di_NotEqual_timeplot_di=[0.02,0.01]}
\includegraphics[width = .4\textwidth]{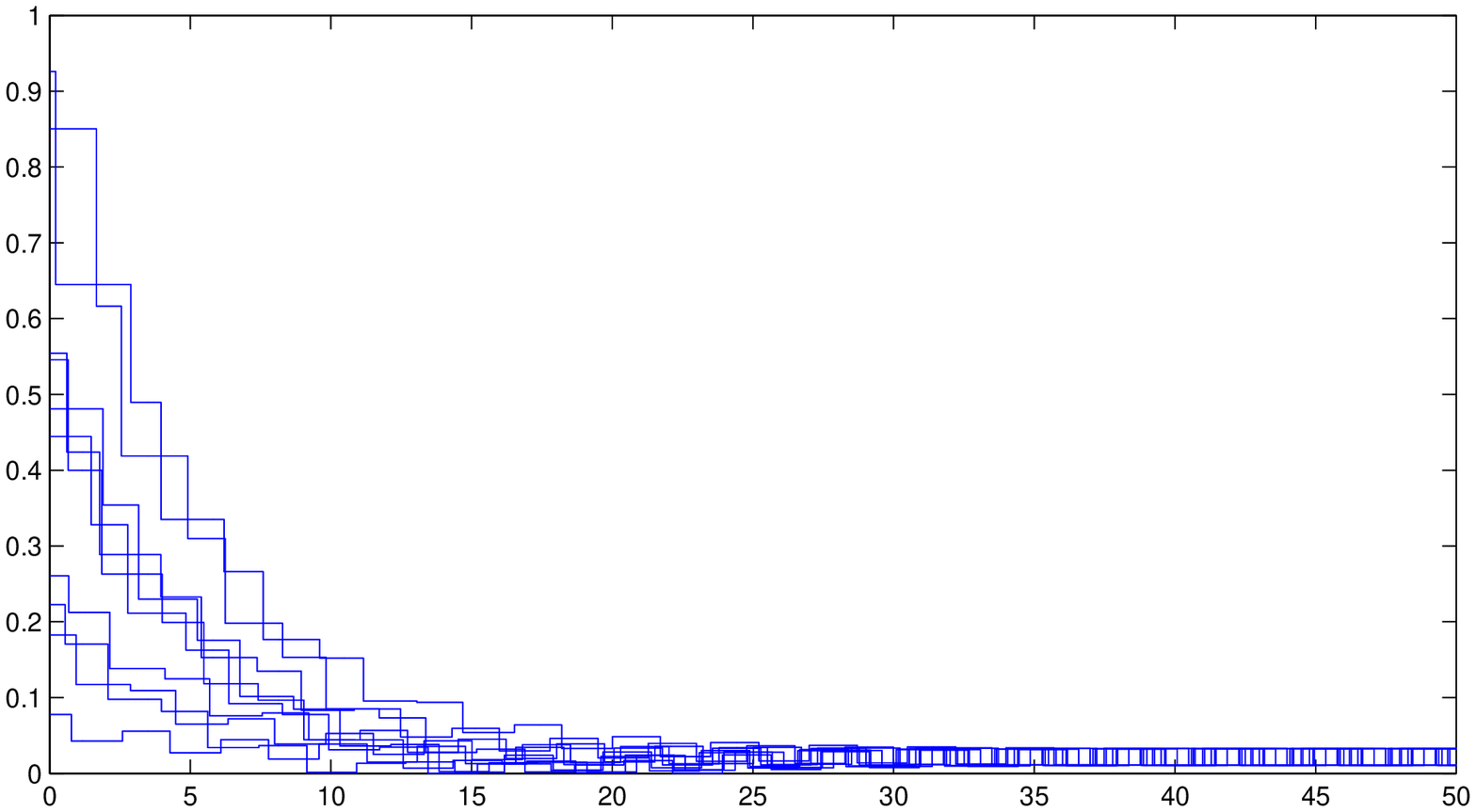}
\put(-160,0){\scalebox{.7}{ $t$ [seconds]}}
\put(-285,70){\scalebox{.7}{ $|\tau|_{\wt\A}$}}
}
\caption{Numerical simulations of the perturbed version of $\HS_2$ with the perturbed ``bump'' on the jump map with
 $\wt\rho_1 \neq \wt\rho_2$.}
\label{figs:ResetBumpPerturb_di_notequal_di}
\end{figure}}

Next, we consider the case of $G_{\rho}$ with $\wt\e_1 \neq \wt\e_2$. Figure~\ref{ResetBumpPertub_di_NotEqual_timeplot_di=[0.15,0.1]} shows the distance to $\wt\A$ for 10 solutions with perturbations given by $\wt\rho_1 = 0.15$ and $\wt\rho_2 = 0.1$. For this case, the distance to $\wt\A$ satisfies $|\tau(t,j)|_{\wt\A} \leq 0.3$ after $\approx40$ seconds of flow time. Figure~\ref{fig:ResetBumpPertub_di_NotEqual_timeplot_di=[0.02,0.01]} shows simulation results with $\wt\rho_1 = 0.02$ and $\wt\rho_2 = 0.01$. Notice that the smaller the value of the perturbation is, the closer the solutions get to the set $\wt\A$. For this case, after $\approx30$ seconds of flow time, the distance to $\wt\A$ satisfies $|\tau(t,j)|_{\wt\A} \leq 0.06$. These simulations validate Theorem~\ref{thm:robustofAS} with $\rho$ affecting only the jump map, verifying that the smaller the size of the perturbation the smaller the steady-state value of the distance to $\wt\A$ would be.

%\end{enumerate}

%%%%%%%%%%%%%%%%%%%%%%
%
%    Flow Map Perturbation Simulations
%
%%%%%%%%%%%%%%%%%%%%%%

\subsubsection{Perturbations on the Flow Map}\label{sec:flowpertrubs} 

In this section, we consider a class of perturbations on the flow map. 
More precisely, consider the case when there exists a function $(t,j) \mapsto c(t,j)$ such that $c(t,j) \leq \bar{c}$ with $\bar{c}$ as in  \eqref{eqn:ctbound}. Then, from Theorem~\ref{thm:vanishingC} with \eqref{eqn:fdelta}, we know that 
\begin{align}
\lim_{t+j \to \infty} |\tau(t,j)|_{\wt\A} \leq \left|\frac{\bar{c}\tb}{\e\omega}\right| \leq \left|\frac{\left|(\frac{1}{N}\underline{\bf 1} - {\bf I})\Delta\omega\right|\tb}{\e\omega}\right|. \label{eqn:flowpert2}
\end{align}

Figure~\IfConf{\ref{figs:allflowperturbs}}{\ref{fig:c=constant}} shows a simulation so as to verify this property. The parameters of this simulation are $N = 2$, $\omega = 1$, $\e = -0.3$, $\tb = 4$, and $\Delta\omega = [0.120,0.134]^\top$. It follows from \eqref{eqn:ctbound} that $\overline{c} = 0.0105$. Then, from \eqref{thmeqn:distupperbound}, it follows that $\lim_{t+j \to \infty}|\tau(t,j)|_{\wt\A} \leq 0.1047$. Specifically, Figure~\ref{fig:c=constant_tau1tau2plot} shows a solution  on the $(\tau_1,\tau_2)$-plane of the perturbed hybrid system $\HS_2$ with initial condition $\tau(0,0) = [0,0.01]^\top$. This figure shows the solution (blue line) converging to a region around $\wt\A$ (between dash-dotted lines about $\A$ in green). Figure~\ref{fig:c=constant_disttimetraj} shows the distance to the set $\wt\A$ of 10 solutions with initial conditions $\tau(0,0) \in C$ with a dashed line denoting the upper bound on the distance in \eqref{eqn:flowpert2}. Notice that all solutions are within this bound after approximately 15 seconds of flow time and stay within this region afterwards. 

\IfConf{\begin{figure}
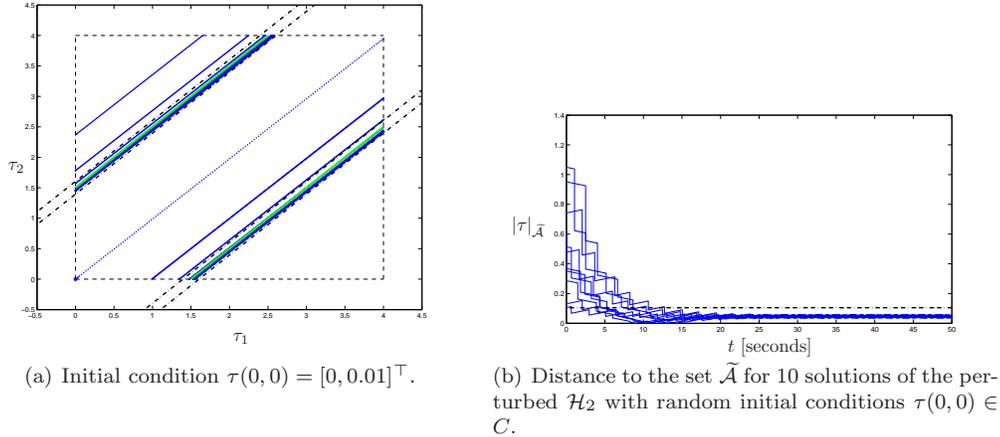

\setlength{\unitlength}{0.016cm}
\centering
%\subfigure[Initial condition $\tau(0,0) = {[0,0.1]}^\top$.]{\label{fig:hbar=.01_t1t2plot}
%\includegraphics[width = .3\textwidth]{Figures/hbar=.01_t1t2plot.eps}
%\put(-120,4){\small $\tau_1$}
%\put(-220,85){\small $\tau_2$}
%}\hspace{0.2cm}
%\subfigure[Initial condition $\tau(0,0) = {[0,0.4]}^\top$.]{\label{fig:hbar=max_t1t2plot}
%\includegraphics[width = .2\textwidth]{Figures/hbar=max_t1t2plot.eps}
%\put(-120,4){\scalebox{.7}{ $\tau_1$}}
%\put(-220,85){\scalebox{.7}{ $\tau_2$}}
%}\hspace{0.2cm}
\subfigure[Initial condition $\tau(0,0) = {[0,0.01]^\top}$.\hspace{-.1cm}]{
\label{fig:c=constant_tau1tau2plot}
\includegraphics[width=.2\textwidth,trim = 20mm 0mm 20mm 10mm, clip]{Figures/c=constant_tau1tau2plot.eps}
\put(-120,4){\scalebox{.7}{ $\tau_1$}}
\put(-240,105){\scalebox{.7}{ $\tau_2$}}
}\hspace{.2cm}
%\subfigure[Distance to the set $\A$ for 10 solutions to the perturbed $\HS_2$ with initial conditions uniformly randomly distributed in the set $C$.]{\label{fig:hbar=.01_timeplot}
%\includegraphics[width = .3\textwidth]{Figures/hbar=.01_timeplot.eps}
%\put(-135,0){\small $t$ [seconds]}
%\put(-240,60){\small $|\tau|_{\A}$}
%}\hspace{0.2cm}
%\subfigure[Distance to the set $\A$ for 10 solutions of the perturbed $\HS_2$ with random initial conditions $\tau(0,0) \in C$.]{\label{fig:hbar=max_timeplot}
%\includegraphics[width = .2\textwidth]{Figures/hbar=max_timeplot.eps}
%\put(-120,0){\scalebox{.7}{$t$ [seconds]}}
%\put(-240,60){\scalebox{.7}{$|\tau|_{\A}$}}
%}
\hspace{0.2cm}
\subfigure[Distance to the set $\A$ for 10 solutions of the perturbed $\HS_2$ with random initial conditions $\tau(0,0) \in C$.]{
\label{fig:c=constant_disttimetraj}
\includegraphics[width=.22\textwidth,trim = 20mm 00mm 15mm 8mm, clip]{Figures/c=constant_disttimetraj.eps}
\put(-146,0){\scalebox{.7}{$t$ [seconds]}}
\put(-275,70){\scalebox{.7}{$|\tau|_{\wt\A}$}}
}
\caption{Solutions to the hybrid system $\HS_2$ with perturbed flow map given by the cases covered in Section~\ref{sec:flowpertrubs}. 
%Specifically, Figures (a), and (b), are of the perturbation $f (\tau) = \omega\one + h(\tau)$, $h(\tau) = {[\bar{h}V(\tau), 0]}^\top$ in Section~\ref{sec:flowpertrubs}(1) with varying parameter $\bar{h}$. The 
%Figures (a) and (c) have $\hbar = \frac{|\log(1+\e)|\omega}{\tb} - 10^{-3}$. 
Figures (a) and (b) show solutions given by the flow perturbation $\Delta\omega = {[0.120,0.134]}^\top$ given in Section~\ref{sec:flowpertrubs}(2). Note that these figures have a dashed black line denoting the calculated distance from $\wt\A$ in \eqref{eqn:flowpert2}.}
\label{figs:allflowperturbs}
\end{figure}}
{\begin{figure}
\setlength{\unitlength}{0.023cm}
\centering
\subfigure[Initial condition $\tau(0,0) = {[0,0.01]^\top}$.\hspace{-.1cm}]{
\label{fig:c=constant_tau1tau2plot}
\includegraphics[width=.4\textwidth]{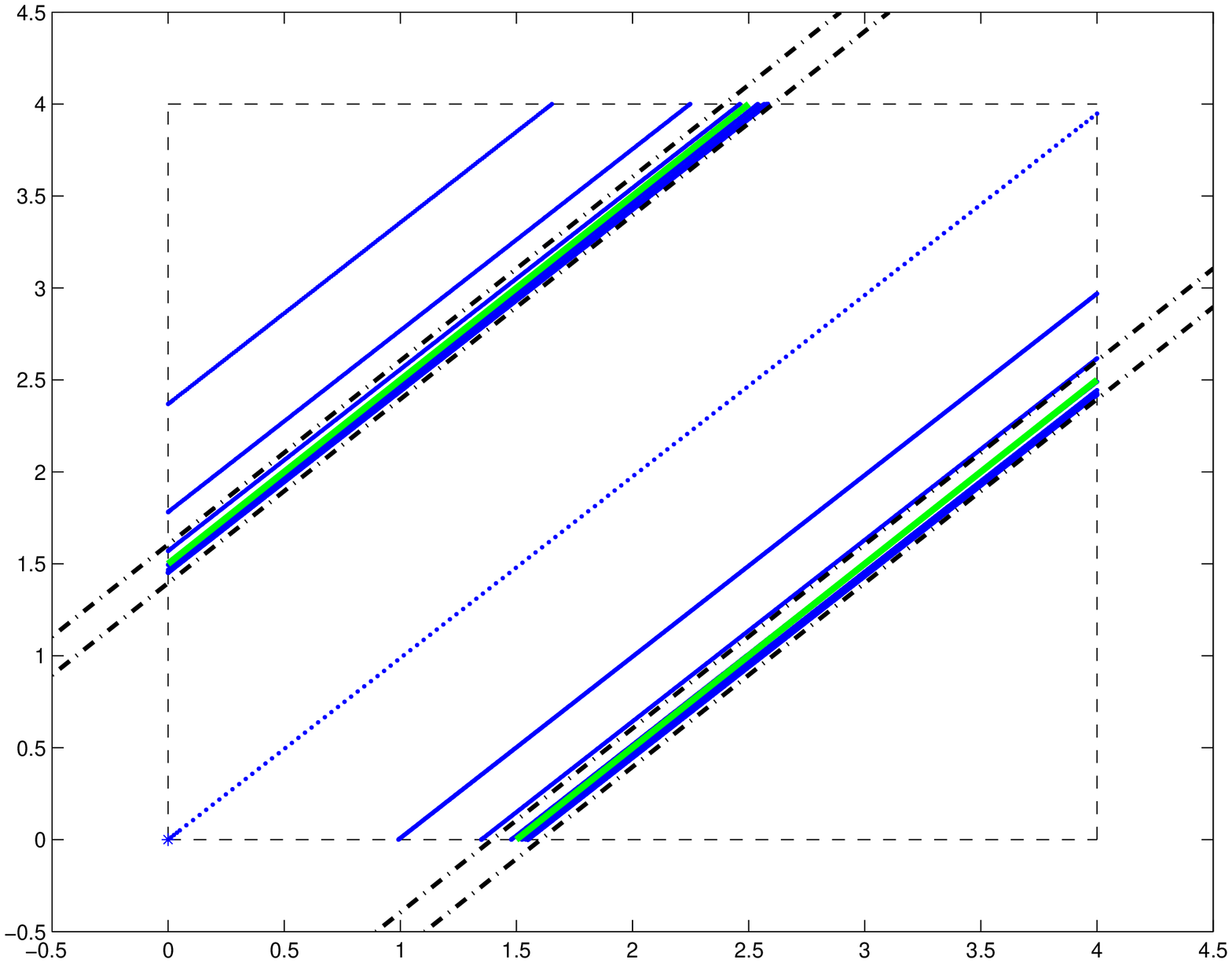}
\put(-140,6){\scalebox{.7}{ $\tau_1$}}
\put(-270,105){\scalebox{.7}{ $\tau_2$}}
}\hspace{0.2cm}
\subfigure[Distance to the set $\wt\A$ for 10 solutions of the perturbed $\HS_2$ with random initial conditions $\tau(0,0) \in C$.]{
\label{fig:c=constant_disttimetraj}
\includegraphics[width=.4\textwidth]{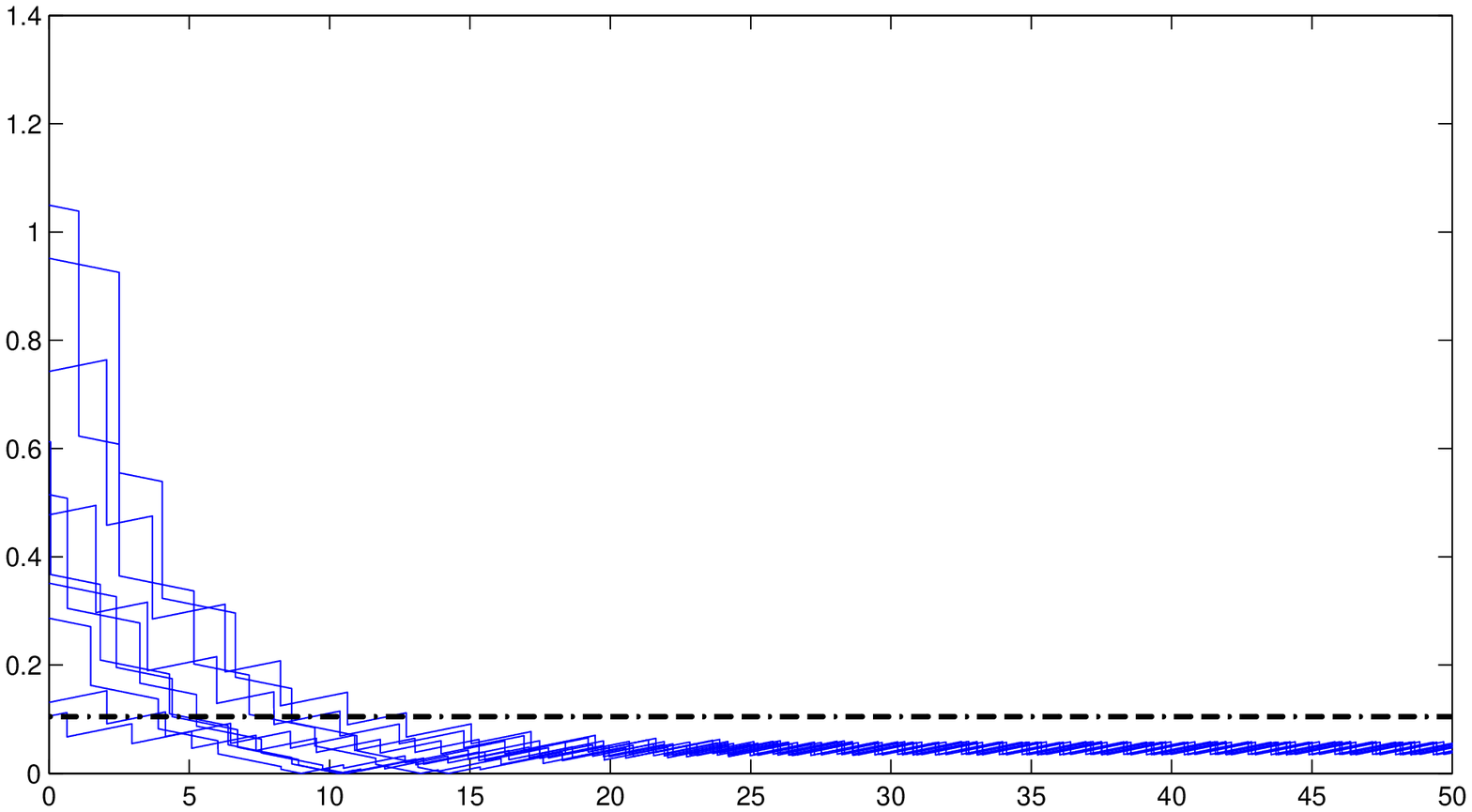}
\put(-160,0){\scalebox{.7}{ $t$ [seconds]}}
\put(-285,70){\scalebox{.7}{ $|\tau|_{\wt\A}$}}
}
\caption{Solutions to the hybrid system $\HS_2$ with perturbed flow map given by the cases covered in Section~\ref{sec:flowpertrubs}. 
Figures (a) and (b) show solutions given by the flow perturbation $\Delta\omega = {[0.120,0.134]}^\top$ given in Section~\ref{sec:flowpertrubs}. Note that these figures have a dashed black line denoting the calculated distance from $\wt\A$ in \eqref{eqn:flowpert2}.}
\label{fig:c=constant}
\end{figure}}

\section{Conclusion}\label{sec:conclusion}
We have shown that desynchronization in a class of impulse-coupled oscillators is an asymptotically stable and robust property. These properties are established within a solid framework for modeling and analysis of hybrid systems, which is amenable for the study of synchronization and desynchronization in other impulse-coupled oscillators in the literature. The main difficulty in applying these tools lies on the construction of a Lyapunov-like quantity certifying asymptotic stability. As we show here, invariance principles can be exploited to relax the conditions that those functions have to satisfy, so as to characterize convergence, stability, and robustness in the class of systems under study.  Future directions of research include the study of nonlinear reset maps, such as those capturing the phase-response curve of spiking neurons, as well as impulse-coupled oscillators connected via general graphs.

\balance
\bibliographystyle{IEEEtran}
\bibliography{Biblio,RGS,SAP}

\appendix
\section{Appendix}\label{sec:appendix}
The following result derives the solution to $\Gamma\tau_{s} = b$ with $\Gamma$ given in \eqref{eqn:Amatrix} and $b = \tb\one$ via Gaussian elimination. 
\begin{lemma}\label{eqn:tauSsolution}
For each $\e \in (-1,0)$, the solution $\tau_s$ to $\Gamma\tau_{s} = b$ with $\Gamma$ given in \eqref{eqn:Amatrix} and $b = \tb\one$ is such that its elements, denoted as $\tau_s^k$ for each $k \in \{1,2,\ldots,N\}$, are given by  $\tau_{s}^{k} = \frac{\sum_{i=0}^{N-k}(\e+1)^i}{\sum_{i=0}^{N-1}(\e+1)^i}\tb$.
\end{lemma}
\begin{proof}
The $N \times N$ matrix in \eqref{eqn:Amatrix} and the $N \times 1$ matrix $b = \tb\one$ leads to the augmented matrix $[\Gamma|b]$ given by 
\IfConf{\begin{align}
\left[\begin{array}{cccccc|c}
1 & 0 & 0 & 0 & \ldots & 0 & \tb \\
0 & (\e+2) & -(\e+1) & 0 & \ldots & 0 & \tb \\
0 & (\e+1) & 1 & -(\e+1)  & \ldots & 0 & \tb \\
%0 & (\e+1) & 0 & 1 & -(\e+1) & \ddots & 0 & \tb \\
%0 & (\e+1) & 0 & 0 & 1 & \ddots & 0 & \tb \\
\vdots & \vdots & \vdots & \ddots & \ddots & \ddots & \vdots  \\
0 & (\e+1) & 0 & 0  & \ddots & -(\e+1) & \tb \\
0 & (\e+1) & 0 & 0  & \hdots & 1 & \tb \\
\end{array}\right]. \label{eqn:augGb}
\end{align}}
{\begin{align}
\left[\begin{array}{ccccccc|c}
1 & 0 & 0 & 0 & 0 & \ldots & 0 & \tb \\
0 & (\e+2) & -(\e+1) & 0 & 0 & \ldots & 0 & \tb \\
0 & (\e+1) & 1 & -(\e+1) & 0 & \ldots & 0 & \tb \\
0 & (\e+1) & 0 & 1 & -(\e+1) & \ddots & 0 & \tb \\
%0 & (\e+1) & 0 & 0 & 1 & \ddots & 0 & \tb \\
\vdots & \vdots & \vdots & \ddots & \ddots & \ddots & \vdots & \vdots \\
0 & (\e+1) & 0 & 0 & 0 & \ddots & -(\e+1) & \tb \\
0 & (\e+1) & 0 & 0 & 0 & \hdots & 1 & \tb \\
\end{array}\right]. \label{eqn:augGb}
\end{align}}
To solve for $\tau_{s}^{k}$, we apply the Gauss-Jordan elimination technique to \eqref{eqn:augGb} to remove the elements $-(\e+1)$ above the diagonal. Starting from the $N$-th row to remove the $-(\e+1)$ component in the $N-1$ row, and continuing up to the second row, gives
\IfConf{\begin{align}
\left[\begin{array}{cccccc|c}
1 & 0 & 0 & 0  & \ldots & 0 & \tb \\
0 & \sum_{i = 0}^{N-1} (\e+1)^{i} & 0 & 0  & \ldots & 0 & \sum_{i = 0}^{N-2} (\e+1)^{i} \tb \\
0 & \sum_{i = 1}^{N-2} (\e+1)^{i} & 1 & 0  & \ldots & 0 & \sum_{i = 0}^{N-3} (\e+1)^{i} \tb \\
%0 & \sum_{i = 1}^{N-3} (\e+1)^{i} & 0 & 1 & 0 & \ddots & 0 & \sum_{i = 0}^{N-4} (\e+1)^{i} \tb \\
%0 & (\e+1) & 0 & 0 & 1 & \ddots & 0 & \tb \\
\vdots & \vdots & \vdots & \ddots  & \ddots & \vdots & \vdots \\
0 & \sum_{i = 1}^{2} (\e+1)^{i}& 0 & 0  & \ddots & 0 & \tb + (1+\e)\tb \\
0 & (\e+1) & 0 & 0  & \hdots & 1 & \tb \\
\end{array}\right]. \label{eqn:augGbp}
\end{align}}{\begin{align}
\left[\begin{array}{ccccccc|c}
1 & 0 & 0 & 0 & 0 & \ldots & 0 & \tb \\
0 & \sum_{i = 0}^{N-1} (\e+1)^{i} & 0 & 0 & 0 & \ldots & 0 & \sum_{i = 0}^{N-2} (\e+1)^{i} \tb \\
0 & \sum_{i = 1}^{N-2} (\e+1)^{i} & 1 & 0 & 0 & \ldots & 0 & \sum_{i = 0}^{N-3} (\e+1)^{i} \tb \\
0 & \sum_{i = 1}^{N-3} (\e+1)^{i} & 0 & 1 & 0 & \ddots & 0 & \sum_{i = 0}^{N-4} (\e+1)^{i} \tb \\
%0 & (\e+1) & 0 & 0 & 1 & \ddots & 0 & \tb \\
\vdots & \vdots & \vdots & \ddots & \ddots & \ddots & \vdots & \vdots \\
0 & (\e+1)^{2} + (\e+1) & 0 & 0 & 0 & \ddots & 0 & \tb + (1+\e)\tb \\
0 & (\e+1) & 0 & 0 & 0 & \hdots & 1 & \tb \\
\end{array}\right]. \label{eqn:augGbp}
\end{align}}
Denoting the augmented matrix in \eqref{eqn:augGbp} as $[\Gamma'|b']$, with  $\tau_{s}^{1} = \tb$ and $\tau_{s}^{2} = \frac{\sum_{i=0}^{N-2}(\e+1)^i}{\sum_{i=0}^{N-1}(\e+1)^i}\tb$, the solution for each element of $\tau_{s}^{k}$ with $k > 2$ can be derived from \eqref{eqn:augGb} as $\Gamma'_{k,2}\tau^{2}_{s} + \tau^{k}_{s} = b'_{k}$ where $\Gamma'_{k,2}$ denotes the $(k,2)$ entry of $\Gamma'$. 
Noting that $\tau_{s}^{1}$ can be rewritten as $\tau_{s}^{1} = \frac{\sum_{i=0}^{N-1}(\e+1)^i}{\sum_{i=0}^{N-1}(\e+1)^i}\tb$ leads to $\tau_{s}^{k} = \frac{\sum_{i=0}^{N-k}(\e+1)^i}{\sum_{i=0}^{N-1}(\e+1)^i}\tb$\NotForConf{\footnote{For example consider $k = 3$, the expression reduces to $\sum_{i = 1}^{N - 2} (\e + 1)^{i} \tau_{s}^{2} + \tau_{s}^{3} = \sum_{i = 0}^{N - 3} (\e + 1)^{i} \tb $ which leads to
\begin{align*}
\tau_{s}^{3} &= \sum_{i = 0}^{N - 3} (\e + 1)^{i} \tb - \sum_{i = 1}^{N - 2} (\e + 1)^{i} \tau_{s}^{2} 
= \frac{\sum_{i=0}^{N-3}(\e+1)^i \sum_{i=0}^{N-1}(\e+1)^i - (\e + 1) \sum_{i=0}^{N-3}(\e+1)^i \sum_{i=0}^{N-2}(\e+1)^i}{\sum_{i=0}^{N-1}(\e+1)^i}\tb \\
&= \frac{\sum_{i=0}^{N-3}(\e+1)^i \left[\sum_{i=0}^{N-1}(\e+1)^i - (\e + 1) \sum_{i=0}^{N-2}(\e+1)^i \right]}{\sum_{i=0}^{N-1}(\e+1)^i}\tb 
%&= \frac{\sum_{i=0}^{N-3}(\e+1)^i \left[\sum_{i=0}^{N-1}(\e+1)^i -  \sum_{i=1}^{N-1}(\e+1)^i \right]}{\sum_{i=0}^{N-1}(\e+1)^i}\tb 
= \frac{\sum_{i=0}^{N-3}(\e+1)^i}{\sum_{i=0}^{N-1}(\e+1)^i}\tb 
\end{align*}}}. 
\end{proof}
\begin{lemma}\label{lem:consum1}
For each $x \neq 1$, and $m,n \in \nats$ such that $n-1 \geq m$, the finite sum $\sum_{i=m}^{n-1} x^i$ satisfies
$
\sum_{i=m}^{n-1} x^i = \frac{x^n - x^m}{x-1}.
$
\end{lemma}
\IfConf{
For a proof of Lemma~\ref{lem:consum1} see \cite{Phillips2013TechReport}.
}{
\begin{proof}
Let
$
S_n = \sum_{i=m}^{n-1} x^i = x^m + x^{m+1} + x^{m+2} + \ldots + x^{n-1} $
multiply by $x$ to get
$
xS_n = x^{m+1} + x^{m+2} + x^{m+3}+ \ldots + x^{n} 
$
Subtracting these expressions leads to 
\IfConf{$xS_n - S_n= (x^{m+1} + x^{m+2} + x^{m+3}+ \ldots + x^{n} )-(x^{m+1} + x^{m+2} + x^{m+3}+ \ldots + x^{n-1})$ 
}{\begin{align*}
xS_n - S_n= (x^{m+1} + x^{m+2} + x^{m+3}+ \ldots + x^{n} )-(x^{m+1} + x^{m+2} + x^{m+3}+ \ldots + x^{n-1}). 
\end{align*}}
Then, it follows that
$
S_n = \frac{x^n - x^m}{x-1}.
$
\EndPF
\end{proof}}
\begin{lemma}\label{lem:consum2}
For each $x \neq 1$, \seannew{and each $m,N \in \nats$ such that $N \geq m$}, the finite sum $\sum_{n=m}^N\sum_{i=0}^{N-n}x^i$ satisfies
\IfConf{$\sum_{n=m}^N\sum_{i=0}^{N-n}x^i = \frac{x^{N-m+2} +(m-N-2)x + (N -m+1)}{(x-1)^2}.$}{\begin{eqnarray}\sum_{n=m}^N\sum_{i=0}^{N-n}x^i = \frac{x^{N-m+2} +(m-N-2)x + (N -m+1)}{(x-1)^2}.\label{eqn:doublesum2} \end{eqnarray}}
\end{lemma}
\IfConf{
For a proof of Lemma~\ref{lem:consum2} see \cite{Phillips2013TechReport}.
}{\begin{proof}
Let $S_{n} = \sum_{n=m}^N\sum_{i=0}^{N-n}x^i$ as in \eqref{eqn:doublesum2}. Expanding $S_{n}$ leads to 
$
S_{n} = \sum_{i=0}^{N-m}x^i + \sum_{i=0}^{N-(m+1)}x^i +\sum_{i=0}^{N-(m+2)}x^i +  \sum_{i=0}^{N-(m+3)}x^i + \ldots + \sum_{i=0}^{1}x^i + \sum_{i=0}^{0}x^i
$
Then, expanding the sum $\sum_{i=0}^{N-m}x^i = x^{N-m} + \sum_{i=0}^{N-(m+1)}x^i$ leads to
\begin{align*}
S_{n} &= x^{N-m} + 2\sum_{i=0}^{N-(m+1)}x^i +\sum_{i=0}^{N-(m+2)}x^i +  \sum_{i=0}^{N-(m+3)}x^i + \ldots + \sum_{i=0}^{1}x^i + \sum_{i=0}^{0}x^i.
\end{align*}
Expanding $\sum_{i=0}^{N-(m+1)}x^i = x^{N-(m+1)} + \sum_{i=0}^{N-(m+2)}x^i$ leads to
\begin{align*}
S_{n} = x^{N-m} + 2x^{N-(m+1)} +3\sum_{i=0}^{N-(m+2)}x^i +  \sum_{i=0}^{N-(m+3)}x^i + \ldots + \sum_{i=0}^{1}x^i + \sum_{i=0}^{0}x^i.
\end{align*}
\IfConf{Continuing this summation expansion, it follows that}{The next two sums follow similarly and we arrive to
\begin{align*}
\sum_{n=m}^N\sum_{i=0}^{N-n}x^i &= x^{N-m} + 2x^{N-(m+1)} +3x^{N-(m+2)} +  4x^{N-(m+3)} + \ldots + \sum_{i=0}^{1}x^i + \sum_{i=0}^{0}x^i.
\end{align*}
Proceeding this way for each sum and noticing that there are exactly $(N-m)$ summations of the form $\sum_{i=0}^{1}x^i$, it follows then that
\begin{align*}
\sum_{n=m}^N\sum_{i=0}^{N-n}x^i &= x^{N-m} + 2x^{N-(m+1)} +3x^{N-(m+2)} +  4x^{N-(m+3)} + \ldots + (N-m)x^{1} + (N-m+1)\sum_{i=0}^{0}x^i
\end{align*}
and finally }
$$
S_{n} = x^{N-m} + 2x^{N-(m+1)} +3x^{N-(m+2)} +  4x^{N-(m+3)} + \ldots + (N-m)x^{1} +  (N-m+1)x^{0} 
$$
which reduces to
$\sum_{n=m}^N\sum_{i=0}^{N-n}x^i = \sum_{i=1}^{N-m+1}i x^{N-i-m+1}. $
It follows that 
\begin{align*}
xS_{n} &= x^{N-m+1} + 2x^{N-m} +3x^{N-(m+1)} +  4x^{N-(m+2)} + \ldots + (N-m)x^{2} +  (N-m+1)x^{1}\\
x^{2}S_{n} &= x^{N-m+2} + 2x^{N-m+1} +3x^{N-m} +  4x^{N-m-1} + \ldots + (N-m)x^{3} +  (N-m+1)x^{2}.
\end{align*}
Then, 
\begin{align*}
x^{2}S_{n}-2xS_{n}+S_{n} &=  (x^{N-m+2} + 2x^{N-m+1} +3x^{N-m} +  4x^{N-m-1} + \ldots + (N-m)x^{3} +  (N-m+1)x^{2}) \\ & - 2(x^{N-m+1} + 2x^{N-m} +3x^{N-(m+1)} +  4x^{N-(m+2)} + \ldots + (N-m)x^{2} +  (N-m+1)x^{1}) \\ & + (x^{N-m} + 2x^{N-(m+1)} +3x^{N-(m+2)} + 4x^{N-(m+3)} + \ldots + (N-m)x^{1} +  (N-m+1)x^{0}).
\end{align*}
which reduce to
$
(x-1)^{2}S_{n} = x^{N-m+2} + (N-m+1) + (m-N-2)x,
$
leading to 
$
S_{n} = \frac{x^{N-m+2} + (N-m+1) + (m-N-2)x}{(x-1)^{2}}.
$\EndPF
\end{proof}}
%
%%
%
%
%
%
%
%
%
%
%
%
%
%
%
%\newpage
%\vspace{-.1cm}
\IfConf{
\begin{IEEEbiography}[{\includegraphics[width=1.in,keepaspectratio]{Phillips_2015biopic_small.eps}}]{Sean Phillips} received his B.S. and M.S. with a focus on dynamics and controls in Mechanical Engineering from the University of Arizona (UA) in Tucson, Arizona. Currently, he is pursuing a Ph.D. in the area of control systems in the Department of Computer Engineering from the University of California, Santa Cruz (UCSC). 

His current research interests include modeling, stability, control and robustness analysis of hybrid systems. 
\end{IEEEbiography}
\begin{IEEEbiography}[{\includegraphics[width=1.in,keepaspectratio]{Sanfelice5x7a_2012smallIEEE.eps}}]{Ricardo G. Sanfelice} received the B.S. degree in Electronics Engineering from the Universidad Nacional de Mar del Plata, Buenos Aires, Argentina, in 2001. He joined the Center for Control, Dynamical Systems, and Computation at the University of California, Santa Barbara in 2002, where he received his M.S. and Ph.D. degrees in 2004 and 2007, respectively. During 2007 and 2008, he was a Postdoctoral Associate at the Laboratory for Information and Decision Systems at the Massachusetts Institute of Technology. He visited the Centre Automatique et Systemes at the Ecole de Mines de Paris for four months.  He is Associate Professor of Computer Engineering at the University of California, Santa Cruz. Prof. Sanfelice is the recipient of the 2013 SIAM Control and Systems Theory Prize, the National Science Foundation CAREER award, the Air Force Young Investigator Research Award, and the 2010 IEEE Control Systems Magazine Outstanding Paper Award.  His research interests are in modeling, stability, robust control, observer design, and simulation of nonlinear and hybrid systems with applications to power systems, aerospace, and biology.
\end{IEEEbiography}}

\end{document}